\documentclass{article}
\usepackage[utf8]{inputenc}

\title{The Geometric Completion of a Doctrine}
\author{J. L. Wrigley\footnote{Università degli Studi dell'Insubria, Via Valleggio n. 11, 22100 Como CO, email: \texttt{jwrigley@uninsubria.it}}}

\usepackage[only,llbracket,rrbracket,shortleftarrow,shortrightarrow]{stmaryrd}

\usepackage{amssymb}
\usepackage{amsmath}
\usepackage{amsthm}
\usepackage{tikz-cd}
\usepackage{fancyhdr}
\usepackage{mathtools}
\usepackage{enumitem}
\usepackage{extpfeil}
\usepackage{bussproofs}
\usepackage{mathtools}
\usepackage{hyperref}
\usepackage{mathrsfs}
\usepackage{cleveref}

\RequirePackage{tikz-cd}
\RequirePackage{amssymb}
\usetikzlibrary{calc}
\usetikzlibrary{decorations.pathmorphing}

\tikzset{curve/.style={settings={#1},to path={(\tikztostart)
			.. controls ($(\tikztostart)!\pv{pos}!(\tikztotarget)!\pv{height}!270:(\tikztotarget)$)
			and ($(\tikztostart)!1-\pv{pos}!(\tikztotarget)!\pv{height}!270:(\tikztotarget)$)
			.. (\tikztotarget)\tikztonodes}},
	settings/.code={\tikzset{quiver/.cd,#1}
		\def\pv##1{\pgfkeysvalueof{/tikz/quiver/##1}}},
	quiver/.cd,pos/.initial=0.35,height/.initial=0}

\tikzset{tail reversed/.code={\pgfsetarrowsstart{tikzcd to}}}
\tikzset{2tail/.code={\pgfsetarrowsstart{Implies[reversed]}}}
\tikzset{2tail reversed/.code={\pgfsetarrowsstart{Implies}}}
\tikzset{no body/.style={/tikz/dash pattern=on 0 off 1mm}}

\tikzset{
	labl/.style={anchor=south, rotate=90, inner sep=.5mm}
}
\usepackage[whole]{bxcjkjatype}

\usepackage{CJK}
\newcommand{\yo}{\text{\begin{CJK}{UTF8}{min}よ\end{CJK}}\!}

\usepackage[X2,T1]{fontenc} % or OT1 in place of T1

\DeclareTextSymbolDefault{\CYRABHDZE}{X2}
\DeclareTextSymbolDefault{\cyrabhdze}{X2}
\newcommand{\Ezh}{\CYRABHDZE}

\newcommand{\GC}{\text{\rm\Ezh}}

\tikzset{%
	symbol/.style={%
		draw=none,
		every to/.append style={%
			edge node={node [sloped, allow upside down, auto=false]{$#1$}}}
	}
}

\setlist{listparindent = \parindent, parsep=0pt,}
\setenumerate[1]{label = (\roman*), ref = (\roman*)}

\newtheorem{thm}{Theorem}[section]

\newtheorem{lem}[thm]{Lemma}

\newtheorem{coro}[thm]{Corollary}

\newtheorem{prop}[thm]{Proposition}

\newtheorem{df}[thm]{Definition}

\newtheorem{rem}[thm]{Remark}
\newtheorem{ex}[thm]{Example}
\newtheorem{exs}[thm]{Examples}

\newtheorem{nota}[thm]{Notation}
\newtheorem{constr}[thm]{Construction}

\newcommand{\sets}{{\bf Sets}}
\newcommand{\Finsets}{{\bf FinSets}}
\newcommand{\id}{{\rm{id}}}

\newcommand{\Hom}{{\rm{Hom}}}

\newcommand{\theory}{\mathbb{T}}
\newcommand{\cat}{\mathcal{C}}
\newcommand{\Sh}{{\bf Sh}}

\newcommand{\Sub}{{\rm{Sub}}}

\newcommand{\Loc}{{\bf Loc}}
\newcommand{\Frm}{{\bf Frm}}
\newcommand{\Poset}{{\bf PoSet}}

\newcommand{\Sort}{{{\rm Sort}_\Sigma}}
\newcommand{\Geom}{{\bf Geom}}

\newcommand{\topos}{\mathcal{E}}

\newcommand{\class}[1]{\llbracket \, {#1} \, \rrbracket}

\newcommand{\form}[2]{\{\,\vec{{#2}} : {#1}\,\}}
\newcommand{\formm}[2]{(\,\vec{{#2}} \,,\, {#1}\,)}
\newcommand{\Lb}{\mathbb{L}}
\newcommand{\fb}{\mathfrak{f}}
\newcommand{\Pt}{{\rm Pt}}

\newcommand{\2}{{\bf 2}}

\newcommand{\PreOrd}{{\bf PreOrd}}
\newcommand{\Idl}{{\bf Idl}}

\newcommand{\ab}{{\bf a}}

\newcommand{\dcat}{\mathcal{D}}
\newcommand{\Bool}{{\bf Bool}}
\newcommand{\Syn}{{\bf Syn}}

\newcommand{\ftopos}{\mathcal{F}}
\newcommand{\etaPJ}{\eta^{(P,J)}}
\newcommand{\ADoc}{A\text{-}{\bf Doc}}

\newcommand{\pr}{{\text{pr}}}
\newcommand{\Con}{{{\bf Con}_\Sigma}}

\newcommand{\Free}{{\rm Fr}}

\usepackage[
backend=biber,
style=ieee,
sorting=nyt
]{biblatex}
\usepackage{geometry}
\begin{document}

	%%%%%%%%%%%%%%%%%%%%%%
	
	%%%%%%%%%%%%%%%%%%%%%%%
	
	\maketitle
	
	%%%%%%%%%%%%%%%%%%
	%%%%%%%%%%%%%%%%%%%
	
	\begin{abstract}
		As several different formal systems with inequivalent syntax may describe equivalent semantics, it is possible to find `completions' to more expressive syntaxes that are semantically invariant.  Doctrine theory, in the sense of Lawvere, is the natural categorical framework in which to express completions for first-order logic.  We study the suitability of a fibred generalisation of the ideal completion of a preorder to act as a completion for doctrines to the syntax of geometric logic.  In contrast to other completions of doctrines considered in the literature, our completion takes a Grothendieck topology as a second argument.  As a result, the geometric completion is idempotent, as well as being `semantically invariant' for any doctrine whose models can be expressed as a category of continuous flat functors, encompassing a wide class of the most commonly considered doctrines.
		
		We also relate the geometric completion to other completions of doctrines considered in the literature: first, by studying the behaviour of the geometric completion when the second argument is omitted, and then by studying the interaction with those completions of doctrines that complete to a fragment of geometric logic.  Throughout, we reference how these completions of doctrines yield completions of categories via the syntactic category construction.  We demonstrate that it is equivalent to represent logical theories by either doctrines or syntactic categories, in so far as they have equivalent classifying toposes.

	\end{abstract}
	
	%%%%%%%%%%%%%

	\tableofcontents
	
	%%%%%%%%%%%%%%%%%%%%%%%%%%%%%%%%%%%%%%%
	%%%%%%%%%%%%%%%%%%%%%%%%%%%%%%%%%%%%%%%%%
	
	\section{Introduction}

	Doctrines (also called indexed or fibred preorders), as introduced by Lawvere in \cite{adjoint} and expanded upon in \cite{lawcompr}, are the natural categorical generalisation of Lindenbaum-Tarski algebras for propositional logics to the first-order setting, as evidenced by the following prototypical example.
	
	\begin{ex}\label{prototypical}
		{\rm
			Let $\theory$ be a theory in a fragment of first-order logic over a signature $\Sigma$.  By ${\bf Con}_\Sigma$, we denote the category of contexts (i.e. finite strings of sorted variables) and relabellings of contexts.  The functor $F^\theory \colon {\bf Con}_\Sigma \to \Poset$ acts as follows:
			\begin{enumerate}
				\item each context $\vec{x}$ is sent to the set $F^\theory(\vec{x})$ of $\theory$-provable equivalence classes of formulae over $\Sigma$ in context $\vec{x}$, ordered by syntactic proof in the theory $\theory$,
				\item and, for each relabelling $\sigma \colon \vec{x} \to \vec{y}$ of contexts, $F^\theory(\sigma)$ acts by substituting the free variables of a formulae $\varphi \in F^\theory(\vec{x})$ according the relabelling $\sigma$.
			\end{enumerate}
		}
	\end{ex}
	
	Doctrines are a powerful tool within categorical logic as the doctrine of a theory can be seen to express certain logical syntax, interpreted by categorical constructions, even when this is not present in the explicit symbolic syntax of the logic.  As an example, a \emph{primary doctrine}, a doctrine $P \colon \cat^{op} \to {\bf MSLat}$ which takes values in ${\bf MSLat}$, the category of meet-semilattices and their homomorphisms, and where $\cat$ is \emph{cartesian} (i.e. $\cat$ has all finite limits), interprets the binary conjunction $\land$ and truth $\top$ symbols.
	
	This categorical formulation of logic and syntax suggests the following question: given a doctrine 
	\[P \colon \cat^{op} \to \PreOrd\]
	and a certain syntax we wish $P$ to model, is there a universal way of completing $P$ to this new syntax?  Many such logical completions have been studied in recent years.  In \cite{pasquali}, Pasquali constructs a co-free completion of a primary doctrine to an elementary doctrine (a primary doctrine which, intuitively, expresses equality, see \parencite[Definition 1.2]{pasquali}).  The quotient completion of an elementary doctrine has been extensively studied by Maietti and Rosolini in \cite{quotcomp}, \cite{quotcompfound}, \cite{maiettirosolinirelating}.  The existential completion, introduced by Trotta in \cite{trotta}, universally completes a primary doctrine to an existential doctrine, a doctrine capable of interpreting existential quantification (see \parencite[Definition 3.3]{trotta}).  The existential completion is adapted by Trotta and Spadetto in \parencite[\S 3]{trottaspadetto} to give a completion of a primary doctrine to one which interprets universal quantification.  In \cite{coumans} Coumans gives a completion of coherent doctrines that generalises the canonical extension of distributive lattices.

	The purpose of this paper is to provide another logical completion: the geometric completion of a doctrine.  As we will see in Theorem \ref{thm:univprop}, this defines a universal and idempotent way of completing any doctrine to a \emph{geometric doctrine}, a member of a class of doctrines that carry the expressive power of a theory of geometric logic.  For terminology and an introduction to the syntax of geometric logic, the reader is directed to \parencite[\S D1]{elephant}.  In short, geometric logic is that fragment of infinitary first-order logic whose permissible symbols are equality $=$, truth $\top$, falsity $\bot$, finite conjunction $\land$, infinitary disjunction $\bigvee$ and existential quantification $\exists$.  Geometric logic carries a powerful spatial intuition (explored in \cite{contandlogic}) and a rich and expressive model theory, as evidenced by classifying topos theory (see \parencite[\S 2.1]{TST}).  In \S 2.3 of \cite{abramskyphd}, Abramsky gives several perspectives on geometric logic, one being that it ought to be interpreted as the logic of \emph{observable properties} -- those properties that can be determined to hold on the basis of a finite amount of information.  We use the term `geometric doctrines' to emphasise the connection with geometric first-order logic and the wider field of doctrine theory, but the term is synonymous with the nomenclature \emph{internal frames} and \emph{internal locales} also found in the literature.

	%%%%%%%%%%%%%%%%%
	%%%%%%%%%%%%%%
	%%%%%%%%%%%%%
	
	\paragraph{The ideal completion for preorders.}  The geometric completion we will study is a fibred generalisation of the ideal completion for preorders, which is described in \S III.4 \cite{JT} and \S II.2.11 \cite{stone}.  Recall that given a preorder $P$, the preorder $2^{P^{op}}$ of all monotone maps $f \colon P^{op} \to 2$, given their point-wise ordering, which is moreover isomorphic to the set of down-sets of $P$ ordered by inclusion, is the \emph{free join completion} of $P$ (the universal property of $2^{P^{op}}$ can be deduced from Remark \ref{rem:preorders}).  Furthermore, $2^{P^{op}}$ is additionally a frame.
	
	Recall also that if we endow $P$ with a Grothendieck topology (also sometimes called a \emph{covering system}), then the $J$-closed down-sets $I \subseteq P$, i.e. those down-sets such that $x \in I$ whenever $\{\,y_j \leqslant x \mid j \in J\,\}$ is a $J$-covering sieve with each $y_j$ in $I$, ordered by inclusion, forms a frame $J\text{-}\Idl(P)$, which we call the $J$-\emph{ideals} on $P$.  The map $\etaPJ \colon P \to J\text{-}\Idl(P)$, which sends an element $x \in P$ to the $J$-closure of the principal down-set $\downarrow \! x$ generated by $x$, constitutes a `geometric completion' of the pair $(P,J)$ in the sense that it satisfies a universal property.

	\begin{thm}[Proposition II.2.11 \cite{stone}]\label{thm:ideals}
		For a meet-semilattice $P$ and a Grothendieck topology $J$ on $P$, the frame $J\text{-}\Idl(P)$ satisfies the universal property that for each meet-semilattice homomorphism $a \colon P \to L$ into a frame $L$ which is $J$-\emph{continuous}, meaning that $x = \bigvee_{y \leqslant x \in S} y$ for each $J$-covering sieve $S$ on $x$, there exists a unique frame homomorphism $\mathfrak{a} \colon J\text{-}\Idl(P) \to L$ for which the triangle
		\[\begin{tikzcd}
			P \ar{r}{\etaPJ} \ar{rd}[']{a} & J\text{-}\Idl(P) \ar[dashed]{d}{\mathfrak{a}} \\
			& L
		\end{tikzcd}\]
		commutes.
	\end{thm}
	
	\begin{rem}[Theorem 6.2 \cite{stonetype}]\label{rem:preorders}{\rm
			In Theorem \ref{thm:ideals}, the requirement that $P$ be a meet-semilattice can be relaxed.  If $P$ is any preorder, and $J$ is a Grothendieck topology on $P$, then the frame $J\text{-}\Idl(P)$ satisfies the universal property that for each $J$-continuous monotone map $a \colon P \to L$ into a frame $L$ there exists a unique monotone map $\mathfrak{a} \colon J\text{-}\Idl(P) \to L$ such that $\mathfrak{a}$ preserves all joins and $\mathfrak{a} \circ \etaPJ = a$.  The map $\mathfrak{a}$ also preserves finite meets, and hence is a frame homomorphism, if and only if $a \colon P \to L$ defines a \emph{morphism of sites} $a \colon (P,J) \to (L,J_{\rm can})$ (see Definition \ref{morphismofsites}).
	}\end{rem}

	The ideal completion of a preorder can be internalised to a topos and this internal ideal completion can be understood externally via a fibred topos theory perspective, as developed in \S 6 of \cite{fibredsites}.  A doctrine $P \colon \cat^{op} \to \PreOrd$, viewed as an \emph{internal preorder} of the presheaf topos $\sets^{\cat^{op}}$, and a Grothendieck topology $J$ on the Grothendieck construction $\cat \rtimes P$ (see \parencite[Definition B1.3.1]{elephant} or Notation \ref{notation:groth}) acting as an internal covering system on $P$, admits a \emph{fibred ideal completion}, as established in \parencite[Theorem 6.1]{fibredsites}, that generalises the ideal completion for preorders.  This is what we call the geometric completion of the doctrine $P$, relative to the Grothendieck topology $J$ on $\cat \rtimes P$.

	%%%%%%%%%%%%%%%%%%
	%%%%%%%%%%%%%%%%%%%
	%%%%%%%%%%%%%%%%%
	
	\paragraph{Our contribution.}  The universal property of the geometric completion we present in \S \ref{subsec:univprop} is an extension (to include change of base category) of the universal property of the fibred ideal completion established in \cite{fibredsites}.  We will show that the geometric completion defines an idempotent 2-monad on the category of doctrinal sites, defined in Definition \ref{df:docsites}.  We also provide a simpler description of the geometric completion for certain doctrines, which can then be used to recover the geometric completion on an arbitrary doctrine.
	
	We also relate the geometric completion to two other classes of completions of doctrines.  The first are \emph{coarse geometric completions}, which are obtained when we `forget' that the geometric completion of a doctrinal site is geometric by equipping the resultant geometric doctrine with a coarser topology.  The coarse geometric completions thus obtained are no longer idempotent but are instead lax-idempotent.  The latter class we study are \emph{sub-geometric completions}, which are intended to capture completions of doctrines to some, but not all, of the data of geometric syntax -- for example, in \S \ref{subsec:compatible} we will prove that Trotta's existential completion \cite{trotta} is sub-geometric according to Definition \ref{df:subgeo}.

	We will also demonstrate how these above completions of doctrines yield corresponding completions of categories, such as the \emph{regular completion} \cite{carboni}, via the \emph{syntactic category} construction.  The syntactic category construction sends a doctrine to its category of \emph{provably functional relations}.  We will also construct a functor from the Grothendieck construction of a regular doctrine to its syntactic category that, when both categories are endowed with suitable Grothendieck topologies, induces an equivalence of sheaf toposes.  Thereby, we obtain the equivalence between representing a theory, or rather its classifying topos, via doctrines and via syntactic categories.

	\paragraph{The philosophical motivation.}  Why have we included a Grothendieck topology as a second argument in the geometric completion?  The topology we choose for the category $\cat \rtimes P$ is intended to capture further syntactic or semantic information of the doctrine $P \colon \cat^{op} \to \PreOrd$.
	
	In addition to the syntactic data of a first-order theory $\theory$, there is also a notion of \emph{semantics} of $\theory$: a category of models of $\theory$ and homomorphisms of these models, often considered as a subcategory of the category of doctrine morphisms $\Hom_{\bf Doc}(F^\theory, \mathscr{P})$ at the suggestion of Kock and Reyes \parencite[Definition 3.5]{kockreyes} (here, $\mathscr{P}$ denotes the powerset doctrine $\mathscr{P} \colon \sets^{op} \to {\bf CBool}$).  Written this way, it is easy to identify the desired models of $\theory$ if they can be characterised as those doctrine morphisms $F^\theory \to \mathscr{P}$ that preserve certain categorical structure.  However, such an ad hoc treatment depends on the syntax of the fragment of logic to which $\theory$ belongs and resists a unified approach.

	As becomes apparent in \parencite[\S D1]{elephant}, the semantics of a first-order theory can be rephrased in terms of \emph{continuous flat functors}.  For a doctrine $P \colon \cat^{op} \to \PreOrd$, there is an embedding of $\Hom_{\bf Doc}(P,\mathscr{P})$ into the functor category $[\cat \rtimes P,\sets]$.  As is well known, and which we recall in \S \ref{subsec:desired}, for many of the most widely considered doctrines, including those doctrines which express coherent, geometric or even classical first-order logic, by choosing a suitable Grothendieck topology $J$ on the category $\cat \rtimes P$ we can obtain equivalences
	\begin{equation}\label{eq:flatandmod}
		J\text{-}{\bf Flat}(\cat \rtimes P,\sets) \simeq {\bf Mod}(P),
	\end{equation}
	where $J\text{-}{\bf Flat}(\cat \rtimes P,\sets)$ is the category of $J$-continuous flat functors from $\cat \rtimes P$ to $\sets$, and ${\bf Mod}(P) $ is the subcategory of $ \Hom_{\bf Doc}(P,\mathscr{P})$ of desired models of $P$.  Therefore, we elect to work with pairs $(P,J)$, where $P \colon \cat^{op} \to \PreOrd$ is a doctrine and $J$ is a Grothendieck topology on $\cat \rtimes P$.
	
	While syntactic completions of doctrines are obviously of a philosophical interest for their universal property, it is also desirable that they be \emph{semantically invariant}, i.e. the category of models associated with a doctrine $P$ and its completion $TP$ are categorically equivalent.  Thus, one is able to study the semantics of the doctrine $P$ but within the potentially more familiar framework of the syntax of $TP$.  We will observe in Theorem \ref{thm:univprop} that, if the desired models of a doctrine $P$ can be expressed as $J\text{-}{\bf Flat}(\cat \rtimes P)$ for some Grothendieck topology $J$ on $\cat \rtimes P$, then the geometric completion of $(P,J)$ is semantically invariant.  It is this property which allows the intended use of the geometric completion: to re-express a study of the model theory for various doctrines by a single treatment for geometric doctrines.

	%%%%%%%%%%%%%%%%%%%%
	%%%%%%%%%%%%%%%%%%%%%%
	%%%%%%%%%%%%%%%%%%%%%
	
	\paragraph{Outline.}  The paper is divided as follows.
	
	\begin{itemize}
		\item In \S \ref{sec:prelim} we recall the necessary background material.  The first half \S \ref{morphcomorphsec} is dedicated to notions used in a site-theoretic approach to the theory of geometric morphisms as can be found in \parencite[\S VII.10]{SGL} and \cite{dense}; we also recall Lemma 2.5 from \cite{myselfintlocmorph} and set out our notation.   In the latter half \S \ref{subsec:intloc} we recall properties of internal locales of Grothendieck toposes, including their external characterisation as established in \parencite[\S VI.2]{JT} and \parencite[\S 4]{fibredsites}, and the interaction between morphisms of internal locales and morphisms of sites that can be found in \cite{myselfintlocmorph}.
		
		\item Our discussion on the use of Grothendieck topologies to model logical syntax is completed in \S \ref{sec:docsites}.  In \S \ref{subsec:desired} we sketch, in the case of primary, existential, coherent and Boolean doctrines, the equivalence (\ref{eq:flatandmod}) between models of a doctrine à la Kock and Reyes \cite{kockreyes} and models of a doctrine as continuous flat functors.  In \S \ref{subsec:docsites} we define the 2-category ${\bf DocSites}$ of \emph{doctrinal sites} on which the geometric completion acts.  We also sketch that the categories of primary, existential, coherent and Boolean doctrines, as normally formulated, form full and faithful 2-subcategories of ${\bf DocSites}$.  We also take this opportunity to fix our definition of geometric doctrines and their morphisms.
		
		\item The geometric completion is developed in \S \ref{sec:geocomp} as an application of the fibred ideal completion of \parencite[\S 6]{fibredsites}.  We demonstrate in \S \ref{subsec:altcont} that, when a doctrine has (preserved) top elements, the geometric completion admits a simpler explicit description to that presented in \parencite[\S 6]{fibredsites}.  The geometric completion of an arbitrary doctrine can be recovered from this particular case, which we demonstrate in order to introduce an example of a sub-geometric completion - the \emph{free top completion}.  Sub-geometric completions are explored in greater detail in \S \ref{sec:subgeo}.  That the geometric completion of a doctrinal site is universal, semantically invariant and idempotent is proved in \S \ref{subsec:univprop}, and this universal property is used to develop the 2-monadic aspects of the geometric completion in \S \ref{subsec:monad}.
		
		\item Section \ref{sec:syn} is concerned with relating doctrinal sites and \emph{regular sites} (see Definition \ref{df:regsites}), and specifically how the geometric completion for doctrinal sites can be used to deduce a geometric completion for regular sites.  This is performed in \S \ref{subsec:geocompforregcat}.  We phrase our development in terms of \emph{existential fibred sites}, recalled from \cite[\S 5]{fibredsites} in \S \ref{subsec:existsites}.  The quasi 2-adjunction (where here we follow the terminology of \cite[\S I.7]{graycategories}) between existential doctrines and regular categories, which sends an existential doctrine to its \emph{syntactic category}, is recalled in \S \ref{subsec:syncat}, and extended in \S \ref{subsec:synsites} to give a quasi 2-adjunction between existential doctrinal sites and regular sites.  Also in \S \ref{subsec:synsites}, we show that the topos of sheaves on an existential doctrinal site and the topos of sheaves on the corresponding \emph{syntactic site} are equivalent.

		\item The geometric completion is idempotent since we can keep track of the geometricity of a geometric doctrine by assigning a suitable Grothendieck topology.  Section \ref{sec:coarse} is dedicated to the study of completions when some of this information is `forgotten' by assigning a \emph{coarser} Grothendieck topology.  We develop a general framework for \emph{coarse geometric completions} and also prove that every coarse geometric completion is lax-idempotent.

		\item Finally, we study \emph{sub-geometric completions} in \S \ref{sec:subgeo}.  Vaguely speaking, a 2-monad $T$ on a 2-subcategory of doctrines, viewed as a completion of doctrines at the suggestion of \cite{trotta2}, is `sub-geometric' if a suitable sub-class of geometric doctrines are all $T$-algebras and the data added to the completion $TP$ of a doctrine $P$ can be `seen' by a certain Grothendieck topology $J^T_P$ on $\dcat \rtimes TP$.  We will show that the geometric completion of $P$ (according to the trivial topology on $\cat \rtimes P$) is isomorphic to the doctrine obtained by completing $P$ according to $T$, keeping track of the new data by the topology $J^T_P$, and then geometrically completing.  Having previously seen that the free top completion is an example of a sub-geometric completion in \S \ref{subsec:altcont}, we begin in \S \ref{freeandex} with another motivating example by demonstrating that Trotta's existential completion of \cite{trotta} satisfies the loose sense of being sub-geometric expressed above.  The formal definition of a sub-geometric completion is introduced in \S \ref{subsec:gensubgeo}, where we also give sufficient conditions for a sub-geometric completion to be lax-idempotent.
		
		We then turn to examples of sub-geometric completions.  In \S \ref{subsec:compatible} we develop a general theory for obtaining sub-geometric completions via subdoctrines of the free geometric completion, which encompasses the existential completion and the coherent completion of a primary doctrine.  We also relate these completions to the \emph{regular} and \emph{coherent} completion of a cartesian category (see \cite{carboni}).  Finally, the free top completion, and two other completions - the \emph{free join completion} and the \emph{binary meet completion} -- are shown to be sub-geometric according to the general definition in \S \ref{subsec:exsubgeo}.
	\end{itemize}

	%%%%%%%%%%%%%%%
	%%%%%%%%%%%%%%
	
	%\subsection{The Propositional Story}

	%%%%%%%%%%%%%%%%%%%%%%%%%%%%%%
	%%%%%%%%%%%%%%%%%%%%%%%%%%%%%%%
	
	\section{Preliminaries}\label{sec:prelim}
	
	We first review some necessary preliminaries.  We will assume a familiarity with introductory topos theory and the theory of frames and locales, as can be found in, respectively, \cite{SGL} and \cite{picado}.  In other places it is useful to be familiar with fragments of first-order logic and their categorical semantics as can be found in \parencite[\S D1]{elephant}.  In the subsequent subsections, we include further details on two areas of topos theory of which we will make the most frequent use: morphisms and comorphisms of sites, and internal locale theory.
	
	\subsection{Morphisms and comorphisms of sites}\label{morphcomorphsec}

	A thorough treatment of the site-theoretic approach to topos theory that we recall here can be found in \cite{dense}.  We recall comorphisms of sites and morphisms of sites, which are, respectively, covariant and contravariant ways of generating geometric morphisms between toposes via functors on their presenting sites.  Finally, we recall in Lemma \ref{fibrcomoprhmorphthm} a result concerning the commutativity of the geometric morphisms induced by a certain mixed diagram of comorphisms and morphisms of sites.
	
	\begin{df}[Definition 2.1, Exposé III \cite{SGA4-3}]\label{comorphismofsites}
		{\rm
			Let $(\cat,J)$ and $(\mathcal{D},K)$ be sites.  A \emph{comorphism of sites} $F \colon (\cat,J) \to (\mathcal{D},K)$ is a functor $F \colon \cat \to \mathcal{D}$ with the cover lifting property: for each object $c$ of $\cat$ and $K$-covering sieve $S$ on $F(c)$, there exists a $J$-covering sieve $R$ on $c$ such that $F(R) \subseteq S$.
		}
	\end{df}
	A comorphism of sites $F \colon (\cat,J) \to (\mathcal{D},K)$ induces a geometric morphism $C_F \colon \Sh(\cat,J) \to \Sh(\mathcal{D},K)$ (see \parencite[Theorem VII.10.5]{SGL}) whose inverse image $C_F^\ast$ is given by $\ab_J(- \circ F)$.  The composite of two comorphisms of sites $F$ and $G$ is still a comorphism of sites for which $C_{F \circ G} = C_F \circ C_G$.

	\begin{df}[Definition 1.1, Exposé III \cite{SGA4-3}]\label{morphismofsites}{\rm Let $(\cat,J)$ and $(\mathcal{D},K)$ be sites.  A \emph{morphism of sites} $F \colon (\cat,J) \to (\mathcal{D},K)$ is a functor $F \colon \cat \to \mathcal{D}$ satisfying the following conditions.
			\begin{enumerate}
				\item If $S$ is a $J$-covering sieve on $c \in \cat$, then $F(S)$ is a $K$-covering family of morphisms on $F(c)$.
				\item Every object $d$ of $\mathcal{D}$ admits a $K$-covering sieve $\{\,d_i \to d \mid i \in I\,\}$ such that each $d_i$, for $i \in I$, has a morphism $d_i \to F(c_i)$ to the image of some $c_i \in \cat$.
				\item For any pair of objects $c_1,\,c_2$ of $\cat$ and any pair of morphisms
				\[g_1 \colon d \to F(c_1), \ g_2 \colon d \to F(c_2)\]
				of $\mathcal{D}$, there exists a $K$-covering family \[\{\,h_i \colon d_i \to d \mid i \in I\,\}\] of morphisms in $\mathcal{D}$, a pair of families
				\[\{\,f^1_i \colon c_i \to c_1 \mid i \in I\,\}, \ \{\,f^2_i \colon c_i \to c_2\mid i \in I\,\} \]
				of morphisms in $\cat$, and, for each $i \in I$, a morphism $k_i \colon d_i \to F(c'_i)$ such that the squares
				\[\begin{tikzcd}
					d_i \ar{r}{h_i} \ar{d}{k_i} & d \ar{d}{g_1} & d_i \ar{r}{h_i} \ar{d}{k_i} & d \ar{d}{g_2} \\
					F(c_i) \ar{r}{F(f^1_i)} & F(c_1) & F(c_i) \ar{r}{F(f^2_i)} & F(c_2)
				\end{tikzcd}\]
				commute.
				\item For any pair of parrallel arrows $f_1, \, f_2 \colon c' \to c$ of $\cat$, and any arrow $g \colon d \to F(c')$ of $\mathcal{D}$ such that $F(f_1) \circ g = F(f_2) \circ g$, there exists a $K$-covering family
				\[\{\,h_i \colon d_i \to d\mid i \in I\,\}\] of morphisms of $\mathcal{D}$, a family of morphisms
				\[\{\,e_i \colon c_i \to c'\mid i \in I\,\}\]
				of $\cat$ such that $f_1 \circ e_i = f_2 \circ e_i$ for all $ i \in I$, and, for each $i \in I$, a morphism $k_i \colon d_i \to F(c_i)$ such that the square
				\[\begin{tikzcd}
					d_i \ar{r}{h_i} \ar{d}{k_i} & d \ar{d}{g} \\
					F(c_i) \ar{r}{F(e_i)} & F(c')
				\end{tikzcd}\]
				commutes for each $i \in I$.
			\end{enumerate}
	}\end{df}
	A morphism of sites $F \colon (\cat,J) \to (\mathcal{D},K)$ induces a geometric morphism $\Sh(F) \colon \Sh(\mathcal{D},K) \to \Sh(\cat,J)$ (see \parencite[Theorem VII.10.2]{SGL}) whose direct image $\Sh(F)_\ast$ sends a sheaf $P \colon \mathcal{D}^{op} \to \sets$ of $\Sh(\mathcal{D},K)$ to $P \circ F^{op}$.  Just as with comorphisms of sites, the composite of two morphisms of sites $F$ and $G$ is still a morphism of sites for which $\Sh(F \circ G) = \Sh(G) \circ \Sh(F)$.
	
	A \emph{dense morphism of sites} is a functor $F \colon \cat \to \dcat$ that induces an equivalence $\Sh(\cat,J) \simeq \Sh(\mathcal{D},K)$.  By \parencite[Theorem 11.2]{shulman}, each dense morphism of sites is a morphism of sites.
	
	\begin{df}[\S 2 \cite{kockmoer}]\label{densemorph}
		{\rm 
			A \emph{dense morphism of sites} $F \colon (\cat,J) \to (\mathcal{D},K)$ is a functor $F \colon \cat \to \mathcal{D}$ such that:
			\begin{enumerate}
				\item\label{enumdesnsemorph1} $S$ is a $J$-covering family in $\cat$ if and only if $F(S)$ is $K$-covering in $\mathcal{D}$,
				\item\label{enumdensemorph2} for every $d \in \mathcal{D}$, there exists a $K$-covering family of morphisms $F(c_i) \to d$,
				\item\label{enumdensemorph3} for every $c_1, c_2 \in \cat$ and arrow $g \colon F(c_1) \to F(c_2)$ in $\mathcal{D}$, there is a $J$-covering family of arrows $f_i \colon c'_i \to c_1$ and a family of arrows $k_i \colon c'_i \to c_2$ such that $g \circ F(f_i) = F(k_i)$,
				\item\label{enumdesnemorph4} for any arrows $f_1, f_2 \colon c_1 \to c_2$ in $\cat$ such that $F(f_1) = F(f_2)$, there exists a $J$-covering family of arrows $k_i \colon c'_i \to c_1$ such that $f_1 \circ k_i = f_2 \circ k_i$.
			\end{enumerate}
		}
	\end{df}

	We can recover the comparison lemma (see \parencite[\S 2]{kockmoer} or \cite[Theorem C2.2.3]{elephant}) via the special case of a dense morphism of sites $F \colon (\cat,J) \to (\mathcal{D},K)$ when $F \colon \cat \to \mathcal{D}$ is an inclusion of a subcategory.
	\begin{df}\label{densesubcat}{\rm
			A subcategory $\cat \subseteq \mathcal{D}$ of a site $(\mathcal{D},K)$ is \emph{dense} for $K$ if:
			\begin{enumerate}
				\item for every $d \in \mathcal{D}$, there is a covering family $S \in K(d)$ generated by morphisms whose domains are in~$\mathcal{C}$,
				\item for every arrow $c \xrightarrow{g} d \in \mathcal{D}$, there is a covering family $S \in J(c)$ generated by morphisms $b \xrightarrow{f} c$ such that $g \circ f$ is in $\mathcal{C}$.
			\end{enumerate}
	}\end{df}
	
	\begin{lem}[The Comparison Lemma]\label{complem}
		Let $(\mathcal{D},K)$ be a site and let $\mathcal{C}$ be a dense subcategory.  There is an equivalence $\Sh(\mathcal{D},K) \simeq \Sh(\mathcal{C},K|_{\mathcal{C}})$.
	\end{lem}

	%$%%%%%%%%%%%%%%%%%%%%%%%
	
	%%%%%%%%%%%%%%%%%%%%%%%%%%%

	%%%%%%%%%%%%%%%%%%%%%%%%%%%

	\paragraph{A mixed diagram of morphisms and comorphisms of sites.}
	
	Finally, we recall Lemma 2.5 of \cite{myselfintlocmorph} that concerns the commutativity of a certain mixed diagram of morphisms and comorphisms of sites.
	
	Recall from \parencite[Definition B1.3.4]{elephant} that a fibration $A \colon \cat \to \topos$ is a functor such that, for each object $c$ of $\cat$ and an arrow $e \xrightarrow{f} A(c)$, there exists a \emph{cartesian lifting} $d \xrightarrow{g} c$ of $f$, i.e. an arrow of $\cat$ such that $A(g) = f$ and, for any arrows $d' \xrightarrow{h} c$ of $\cat$ and $A(d') \xrightarrow{k} A(d)$ of $\topos$ for which
	\[\begin{tikzcd}
		A(d') \ar{r}{k} \ar{rd}[']{A(h)} & A(d) \ar{d}{A(g)} \\
		& A(c)
	\end{tikzcd}\]
	commutes, there exists an arrow $d' \xrightarrow{k'} d$ of $\cat$ such that $A(k') =k$ (note that we are using the terminology `cartesian arrow' where Johnstone uses `prone').  An arrow $d \xrightarrow{g} c$ of $\cat$ with this property is said to be cartesian.  Recall also that, given a pair of fibrations $A \colon \cat \to \topos$ and $B \colon \dcat \to \ftopos$, a morphism of the fibrations $A \to B$ consists of a pair of functors $F \colon \cat \to \dcat$ and $G \colon \topos \to \ftopos$ such that the square
	\begin{equation*}
		\begin{tikzcd}
			\cat \ar{r}{F} \ar{d}{A} & \dcat \ar{d}{B} \\
			\topos \ar{r}{G} & \mathcal{F}
		\end{tikzcd}
	\end{equation*}
	commutes and, if $d \xrightarrow{g} c \in \cat$ is cartesian, so is $F(d) \xrightarrow{F(g)} F(c)$.

	\begin{lem}[Lemma 2.5 \cite{myselfintlocmorph}]\label{fibrcomoprhmorphthm}
		Let $(\cat,J)$, $(\dcat,K)$, $(\topos,L)$ and $(\mathcal{F},M)$ be sites.  Let $A \colon \cat \to \topos$ and $B \colon \dcat \to \mathcal{F}$ both be fibrations and let $F \colon \cat \to \dcat$ and $G \colon \topos \to \mathcal{F}$ be functors such that the square
		\begin{equation}\label{morphoffibr}
			\begin{tikzcd}
				\cat \ar{r}{F} \ar{d}{A} & \dcat \ar{d}{B} \\
				\topos \ar{r}{G} & \mathcal{F}
			\end{tikzcd}
		\end{equation}
		is a morphism of fibrations.  Suppose that the functors $A$ and $B$ yield comorphisms of sites $A \colon (\cat,J) \to (\topos,L)$ and $B \colon (\dcat,K) \to (\mathcal{F},M)$, and that the functors $F$ and $G$ yield morphisms of sites $F \colon (\cat,J) \to (\dcat,K)$ and $G \colon (\topos,L) \to (\mathcal{F},M)$.  Then the induced square of geometric morphisms
		\[\begin{tikzcd}
			\Sh(\cat,J) \ar{d}{C_A} & \ar{l}[']{\Sh(F)} \Sh(\dcat,K) \ar{d}{C_B} \\
			\Sh(\topos,L) & \ar{l}[']{\Sh(G)} \Sh(\mathcal{F},M)
		\end{tikzcd}\]
		also commutes.
	\end{lem}

	\begin{nota}\label{notation:groth}
		{\rm
			We will be concerned exclusively with fibrations arising from functors $P \colon \cat^{op} \to \PreOrd$, and so take some time to fix our notation.  Given a functor $P \colon \cat^{op} \to \PreOrd$, we denote by $\cat \rtimes P$ the Grothendieck construction (see \parencite[Definition B1.3.1]{elephant} for the general definiton), the category which has:
			\begin{enumerate}
				\item as objects, pairs $(c,x)$ where $c$ is an object of $\topos$ and $x$ is an element of $P(c)$,
				\item and an arrow $f \colon (c,x) \to (d,y)$ for each arrow $f \colon c \to d$ in $\topos$ such that $x \leqslant P(f)(y)$.
			\end{enumerate}
			Given a second functor $Q \colon \dcat \to \PreOrd$, a functor $F \colon \cat \to \dcat$ and a natural transformation $a \colon P \to A \circ F^{op}$, we denote by
			\[F \rtimes a \colon \cat \rtimes P \to \dcat \rtimes Q\]
			the functor that sends an object $(c,x) \in \cat \rtimes P$ to $(F(c),a_c(x)) \in \dcat \rtimes Q$ and an arrow $(c,x) \xrightarrow{f} (d,y) \in \cat\rtimes P$ to $(F(c),a_c(x)) \xrightarrow{F(f)} (F(d),a_d(y)) \in \dcat \rtimes Q$.  The evident projection functors $p_P \colon \cat \rtimes P \to \cat$ and $p_Q \colon \dcat \rtimes Q \to \dcat$ are both fibrations and the commutative square
			\[\begin{tikzcd}
				\cat \rtimes P \ar{r}{F \rtimes a} \ar{d}{p_P} & \dcat \rtimes Q \ar{d}{p_Q} \\
				\cat \ar{r}{F} & \dcat,
			\end{tikzcd}\]
			constitutes a morphism of fibrations.
		}
	\end{nota}

	%%%%%%%%%%%%%%%%%%%%%%%%%
	%%%%%%%%%%%%%%%%%%%%%%%
	%%%%%%%%%%%%%%%%%%%%%%%
	
	%%%%%%%%%%%%%%%%%%%%%%%
	%%%%%%%%%%%%%%%%%%%%%%%%%%%
	%%%%%%%%%%%%%%%%%%%%%%%%

	\subsection{Internal locales}\label{subsec:intloc}

	In this subsection we list some facts regarding internal locales that will prove necessary in our development.  An internal locale $\Lb$ of a topos $\topos$ is an object that behaves as a locale according to the internal logic of $\topos$.  Internal locales are, of course, also internal frames (although their morphisms will differ).  We will also use the terminology `geometric doctrines' (the 2-category of geometric doctrines ${\bf GeomDoc}$ is defined in \S \ref{subsec:docsites}) to emphasise their role in interpreting geometric logic, mirroring the use of locales and frames to model geometric propositional logic, and to distinguish that we will be comparing internal locales/geometric doctrines that are internal to different toposes.

	We first recall the following fact recounted on p. 528 \parencite[\S C1.6]{elephant}: for each geoemtric morphism $f \colon \ftopos \to \topos$, the direct image functor $f_\ast \colon \ftopos \to \topos$ preserves internal locales.  In particular, every internal locale of $\topos$ is of the form $f_\ast (\Omega_{\ftopos})$ for some geometric morphism $f \colon \ftopos \to \topos$.
	
	In addition we recall the characterisation of the internal locales of an arbitrary Grothendieck topos and the characterisation of their internal locale morphisms.

	%%%%%%%%%%%%%%%%%%%%%%%%%%%%
	\paragraph{Internal locales of presheaf toposes.}
	
	We first require an external classification of the internal locales of a Grothendieck topos $\topos$, i.e. if $(\cat,J)$ is a site for $\topos$, a classification of which $J$-sheaves on $\cat$ define internal locales.
	
	Recall from \parencite[\S V.1]{JT} that a frame homomorphism $f^{-1} \colon L \to K$ is \emph{open} if $f^{-1}$ has a left adjoint satisfying the Frobenius condition, equivalently $f^{-1}$ is a morphism of complete Heyting algebras by \parencite[Proposition V.1.1]{JT}.  We denote the category of frames and open frame homomorphisms by $\Frm_{\rm open}$.
	
	Let $\Lb \colon \cat^{op} \to \Frm_{\rm open}$ be a doctrine.  For each arrow $d \xrightarrow{f} c \in \cat$, we denote the left adjoint to $\Lb(f) \colon \Lb(c) \to \Lb(d)$ by $\exists_{\Lb(f)}$, and when there is no ambiguity by just $\exists_f$.  Every internal locale of the topos $\Sh(\cat,J)$ is of the form $\Lb \colon \cat^{op} \to \Frm_{\rm open}$ for some $\Lb$, but the converse is not true.  The classification when $\cat$ is cartesian was first established in \cite{JT}, while the classification over an arbitrary base category is established in \cite{fibredsites}.
	
	\begin{prop}[Proposition 5.10 \cite{fibredsites} \& Proposition VI.2.2 \cite{JT}]\label{prop:intloc}
		The internal locales of the presheaf topos $\sets^{\cat^{op}}$ are precisely those functors $\Lb \colon \cat^{op} \to \Frm_{\rm open}$ which satisfy the \emph{relative Beck-Chevalley condition}: given an arrow $e \xrightarrow{h} d$ of $\cat$, and a sieve $S$ of $\cat \rtimes \Lb$ on the object $(d,V)$ for which $V = \bigvee_{f \in S} \exists_f U$, then
		\[h^{-1}(V) = \bigvee_{g \in h^\ast(S)} \exists_g W,\]
		where $h^\ast(S)$ is the sieve on $(e,h^{-1}(V))$ consisting of those arrows $(c,W) \xrightarrow{g} (e,h^{-1}(V))$ such that the composite $(c,W) \xrightarrow{g} (e,h^{-1}(V)) \xrightarrow{h} (d,V)$ is in $S$ (if $S$ is a large sieve, when we write $V = \bigvee_{f \in S} \exists_f U$, we mean that there is a small $S' \subseteq S$ such that $V = \bigvee_{f \in S'} \exists_f U$).

		If $\cat$ has all pullbacks, the relative Beck-Chevalley condition for $\Lb$ is equivalent to the Beck-Chevalley condition: for each pullback square \[\begin{tikzcd}
			c \times_e d \ar{d}{k} \ar{r}{g} & d \ar{d}{h} \\
			c \ar{r}{f} & e
		\end{tikzcd}\]
		of $\cat$, the square
		\[\begin{tikzcd}
			\mathbb{L}({c \times_e d}) \ar{r}{\exists_{\Lb(g)}} & \mathbb{L}(d) \\
			\mathbb{L}(c) \ar{r}{\exists_{\Lb(f)}} \ar{u}{\Lb(k)} & \mathbb{L}(e) \ar{u}{\Lb(h)}
		\end{tikzcd}\]
		commutes.
	\end{prop}
	
	The relative Beck-Chevalley condition is one component in a wider theory of \emph{existential sites} developed in \parencite[\S 5]{fibredsites}, which we recall in the doctrinal case in \S \ref{subsec:existsites}.

	Let $\Lb \colon \cat^{op} \to \Frm_{\rm open}$ be an arbitrary functor.  We denote by $K_\Lb$ the function that assigns to each $(d,V) \in \cat \rtimes \Lb$ the collection $K_\Lb(d,V)$ of sieves $\{\,(c_i,U_i) \xrightarrow{f_i} (d,V) \mid i \in I\,\}$ in $\cat \rtimes \Lb$ such that
	\[V = \bigvee_{i \in I} \exists_{f_i} U_i.\]
	The function $K_\Lb$ is not necessarily a Grothendieck topology on the category $\cat \rtimes \Lb$.  It clearly satisfies the maximality and transitivity conditions of Definition III.2.1 \cite{SGL}.  However, the stability of $K_\Lb$ is equivalent to the relative Beck-Chevalley condition (see \parencite[Theorem 5.1]{fibredsites}), and hence $\Lb$ is an internal locale of $\sets^{\cat^{op}}$ if and only if $K_\Lb$ defines a Grothendieck topology on $\cat \rtimes \Lb$.
	
	\begin{df}\label{df:intsheaves}
		{\rm 
			Let $\Lb \colon \cat^{op} \to \Frm_{\rm open}$ be an internal locale of the presheaf topos $\sets^{\cat^{op}}$.  The topos $\Sh(\cat \rtimes \Lb,K_\Lb)$ is called the \emph{topos of internal sheaves} on $\Lb$, and we will employ the shorthand $\Sh(\Lb)$ for this topos.
		}
	\end{df}

	\paragraph{Localic geometric morphisms.}
	The projection $p_\Lb \colon \cat \rtimes \Lb \to \cat$ yields a comorphism of sites
	\[p_\Lb \colon (\cat \rtimes \Lb,K_\Lb) \to (\cat, J_{\rm triv})\]
	and therefore a geometric morphism $C_{p_\Lb} \colon \Sh(\Lb) \to \sets^{\cat^{op}}$.  Since the projection functor $p_\Lb$ is faithful, by \parencite[Proposition 7.11]{dense} it follows that $C_{p_\Lb}$ is a \emph{localic geometric morphism}.  In fact, the characterisation of internal locales in \cite{fibredsites} as given above is proved by exploiting the bijection between internal locales and localic geometric morphisms.
	\begin{df}
		{\rm
			A geometric morphism $f \colon \mathcal{F} \to \topos$ is \emph{localic} if every object $F$ of $\mathcal{F}$ is a subquotient of $f^\ast(E)$ for some $E \in \topos$, i.e. there exists $F' \in \mathcal{F}$ and a diagram
			\[\begin{tikzcd}
				F & F' \ar[two heads]{l} \ar[tail]{r} & f^\ast(E).
			\end{tikzcd}\]
		}
	\end{df}

	\begin{thm}[Theorem 5.37 \cite{topos}, Lemma 1.2 \cite{fact1}, \& Proposition 4.2 \cite{fibredsites}]\label{localicmorph}
		Let $\Sh(\cat,J)$ be a Grothendieck topos.  There is, up to equivalence, a bijective correspondence between internal locales of $\Sh(\cat,J)$ and localic geometric morphisms $f \colon \ftopos \to \Sh(\cat,J)$ given by associating:
		\begin{enumerate}
			\item each localic geometric morphism $f \colon \ftopos \to \Sh(\cat,J)$ with the internal locale $f_\ast(\Omega_{\ftopos})\colon \cat^{op} \to \Frm_{\rm open}$ of $\Sh(\cat,J)$,
			\item and each internal locale $\Lb \colon \cat^{op} \to \Frm_{\rm open}$ of $\Sh(\cat,J)$ with the localic geometric morphism \[C_{p_\Lb} \colon \Sh(\Lb) \to \Sh(\cat,J).\]
		\end{enumerate}
	\end{thm}

	\paragraph{Internal locales of sheaf toposes.}
	Although not used in our main argument, we include for completeness the external classification of internal locales of an arbitrary Grothendieck topos $\Sh(\cat,J)$.  Since the inclusion $\Sh(\cat,J) \rightarrowtail \sets^{\cat^{op}}$ is localic (see \parencite[Example A4.6.2(a)]{elephant}), we can deduce the classification of internal locales of $\Sh(\cat,J)$ by the classification for $\sets^{\cat^{op}}$ mentioned above.
	
	\begin{coro}[Proposition 5.10 \cite{fibredsites} \& Corollary C1.6.10 \cite{elephant}]\label{intlocoversites}
		Let $\Sh(\cat,J)$ be a Grothendieck topos.  The internal locales of $\Sh(\cat,J)$ are precisely the internal locales $\Lb \colon \cat^{op} \to \Frm_{\rm open}$ of $\sets^{\cat^{op}}$ that satisfy one of the equivalent conditions:
		\begin{enumerate}
			\item\label{shflem:2} $\Lb$ is an internal locale of $\sets^{\cat^{op}}$ and a $J$-sheaf,
			\item\label{shflem:3} $K_\Lb$ is stable and contains the \emph{Giraud topology} $J_{p_\Lb}$ (see \parencite[Proposition 2.1]{giraud}),
			\item\label{shflem:4} $K_\Lb$ is stable and there exists a factorisation
			\[\begin{tikzcd}
				& \Sh(\Lb) \ar{d}{C_{p_\Lb}} \ar[dashed]{ld} \\
				\Sh(\cat,J) \ar[tail]{r} & \sets^{\cat^{op}}.
			\end{tikzcd}\]
		\end{enumerate} 
	\end{coro}
	
	%%%%%%%%%%%%%%%%%%%%%%%%%
	\paragraph{Internal locale morphisms.}  We will also need an external classification of internal locale morphisms.  Internal locale morphisms were identified in \parencite[Proposition VI.2.1]{JT}, while their connection with localic geometric morphisms is well known.  We will require also the link with morphisms of sites established in \cite{myselfintlocmorph}.

	\begin{df}[Proposition VI.2.1 \cite{JT}]
		{\rm
			An \emph{internal locale morphism} $\fb \colon \mathbb{L}_1 \to \mathbb{L}_2$, between internal locales $\mathbb{L}_1,\mathbb{L}_2\colon \cat^{op} \to \Frm_{\rm open}$ of the topos $\Sh(\cat,J)$, is a natural transformation $\fb^{-1} \colon \mathbb{L}_2 \to \mathbb{L}_1$ (as functors into $\sets$, or indeed $\Frm$) such that, for each object $c$ of $\cat$, $\fb^{-1}_c \colon \mathbb{L}_2(c)\to \mathbb{L}_1(c)$ is a frame homomorphism and, for each morphism $g \colon c \to d$ of $\cat$, the diagram
			\[\begin{tikzcd}
				\Lb_2(d) \ar{d}{\fb_d^{-1}} \ar[shift right]{r}[']{\Lb_2(g)} & \Lb_2(c) \ar[shift right]{l}[']{\exists_{\Lb_2(g)}} \ar{d}{\fb_c^{-1}} \\
				\Lb_1(d)\ar[shift right]{r}[']{\Lb_1(g)} & \Lb_1(c) \ar[shift right]{l}[']{\exists_{\Lb_1(g)}}
			\end{tikzcd}\]
			is a \emph{morphism of adjunctions}: that is, $\Lb_1(g) \circ \fb_d^{-1} = \fb_c^{-1} \circ \Lb_2(g)$ and $\exists_{\Lb_1(g)} \circ \fb_c^{-1} = \fb_d^{-1}\circ \exists_{\Lb_2(g)}$.
		}
	\end{df}
	
	\begin{prop}[Proposition 4.2 \& Proposition 4.3 \cite{myselfintlocmorph}]\label{prop:intlocmorph}
		Let $\Lb_1 , \Lb_2 \colon \cat^{op} \to \Frm_{\rm open}$ be internal locales of $\Sh(\cat,J)$.  The map that sends a morphism of internal locales $\fb \colon \Lb_1 \to \Lb_2$ to the morphism of sites 
		\[\id_\cat \rtimes \fb^{-1} \colon (\cat \rtimes \Lb_2,K_{\Lb_2}) \to (\cat \rtimes \Lb_1,K_{\Lb_1})\]
		and the map that sends a morphism of sites $G \colon (\cat \rtimes \Lb_2,K_{\Lb_2}) \to (\cat \rtimes \Lb_1,K_{\Lb_1})$ to the geometric morphism $\Sh(G) \colon \Sh(\Lb_2) \to \Sh(\Lb_1)$ induce a bijective correspondence between:
		\begin{enumerate}
			\item the internal locale morphisms $\fb \colon \Lb_1 \to \Lb_2$,
			\item the morphisms of sites $G \colon (\cat \rtimes \Lb_2,K_{\Lb_2}) \to (\cat \rtimes \Lb_1,K_{\Lb_1})$ for which the triangle
			\[\begin{tikzcd}[column sep = tiny]
				\cat \rtimes \Lb_2 \ar{rd}[']{p_{\Lb_2}}\ar{rr}{G} && \cat \rtimes \Lb_1 \ar{ld}{p_{\Lb_1}} \\
				& \cat &
			\end{tikzcd}\]
			commutes,
			\item the geometric morphisms $f \colon \Sh(\Lb_1) \to \Sh(\Lb_2)$ for which the triangle
			\[\begin{tikzcd}
				\Sh(\Lb_1) \ar{rr}{f}\ar{rd}[']{C_{p_{\Lb_1}}} && \Sh(\Lb_2)\ar{ld}{C_{p_{\Lb_2}}} \\
				&\Sh(\cat,J) &
			\end{tikzcd}\]
			commutes.
		\end{enumerate}
	\end{prop}
	
	Let $\topos$ be a Grothendieck topos.  We will use ${\bf Loc}(\topos)$ to denote the 2-category of internal locales of $\topos$, their internal locale morphisms, while 2-cells are equivalently 2-cells between the corresponding geometric morphisms
	\[\begin{tikzcd}
		{{\bf Sh}(\mathbb{L}_1)} && {{\bf Sh}(\mathbb{L}_2)} \\
		\\
		& {{\bf Sh}(\mathcal{C},J),}
		\arrow[""{name=0, anchor=center, inner sep=0}, "f"', curve={height=18pt}, from=1-1, to=1-3]
		\arrow["{C_{p_{\mathbb{L}_1}}}"', from=1-1, to=3-2]
		\arrow["{C_{p_{\mathbb{L}_2}}}", from=1-3, to=3-2]
		\arrow[""{name=1, anchor=center, inner sep=0}, "g", curve={height=-18pt}, from=1-1, to=1-3]
		\arrow["\alpha", shorten <=3pt, shorten >=3pt, Rightarrow, from=1, to=0]
	\end{tikzcd}\]
	i.e. a natural transformation $g^\ast \to f^\ast$ between the inverse image functors, or natural transformations
	\[\begin{tikzcd}
		{(\mathcal{C}\rtimes \mathbb{L}_2 ,K_{\mathbb{L}_2})} && {(\mathcal{C}\rtimes \mathbb{L}_1 ,K_{\mathbb{L}_1})}
		\arrow[""{name=0, anchor=center, inner sep=0}, "{\id_{\mathcal{C}} \rtimes \mathfrak{f}^{-1}}"', curve={height=12pt}, from=1-1, to=1-3]
		\arrow[""{name=1, anchor=center, inner sep=0}, "{\id_{\mathcal{C}} \rtimes \mathfrak{g}^{-1}}", curve={height=-12pt}, from=1-1, to=1-3]
		\arrow["{\alpha'}", shorten <=3pt, shorten >=3pt, Rightarrow, from=1, to=0]
	\end{tikzcd}\]
	between the corresponding morphisms of sites (see \parencite[Theorem 4.4]{myselfintlocmorph}).

	%%%%%%%%%%%%%%%%%%%%%%%%%%%%%%%%%%%%%
	%%%%%%%%%%%%%%%%%%%%%%%%%%%%%%%%%%%%%
	
	%%%%%%%%%%%%%%%%
	%%%%%%%%%%%%%%%%
	%%%%%%%%%%%%%%%%%%%
	
	\section{Doctrines and doctrinal sites}\label{sec:docsites}

	Doctrines have come in many flavours since their original introduction by Lawvere in \cite{adjoint}.  For us, a doctrine is simply any functor $P \colon \cat^{op} \to \PreOrd$, where there are no restrictions on $\cat$ (these are also called fibred preorders).  Being fibred categories, doctrines naturally form a 2-category ${\bf Doc}$ as follows.
	\begin{enumerate}
		\item The objects of ${\bf Doc}$ are doctrines $P \colon \cat^{op} \to \PreOrd$.
		\item The arrows of ${\bf Doc}$ between two doctrines $P \colon \cat^{op} \to \PreOrd$ and $Q \colon \dcat^{op} \to \PreOrd $ consist of pairs $(F,a)$ where $F$ is a functor $F \colon \cat \to \dcat$ and $a$ is a natural transformation $a \colon P \to Q \circ F^{op}$.
		\item A 2-cell $\alpha \colon (F,a) \to (F',a')$ is a natural transformation $\alpha \colon F \to F'$ such that $a_c(x) \leqslant Q(\alpha_c)(a'_c(x))$ for each object $c \in \cat$ and $x \in P(c)$.
	\end{enumerate}

	As currently formulated, an arbitrary doctrine does not interpret any particular logical syntax.  To model a certain logical syntax, we can require a doctrine to possess the appropriate categorical structure.
	
	\begin{exs}\label{exs:doctrines}
		{\rm
			\begin{enumerate}
				\item\label{ex:primdoc} A \emph{primary doctrine} is a doctrine $P \colon \cat^{op} \to \PreOrd$ for which $\cat$ is cartesian and $P$ factors through the subcategory ${\bf MSLat} \subseteq \PreOrd$ of meet-semilattices and meet-semilattice homomorphisms (for us, a meet-semilattice has all finite meets, including the empty one - i.e. a top element).  Primary doctrines thus carry the necessary categorical structure to interpret the symbols $\{\, \land , \, \top\,\}$.  A morphism $(F,a) \colon P \to Q$ in ${\bf Doc}$, between primary doctrines, is said to be a morphism of primary doctrines if the finite limit data are preserved, i.e. $F \colon \cat \to \dcat$ is left exact and $a_c \colon P(c) \to Q(F(c))$ is a meet-semilattice homomorphism for each $c \in \cat$.  We denote by ${\bf PrimDoc}$ the 2-full 2-subcategory of ${\bf Doc}$ of primary doctrines and primary doctrine morphisms (by 2-full we mean that the 2-subcategory is full on 2-cells).
				
				\item An \emph{existential doctrine} is a primary doctrine where in addition, for each arrow $d \xrightarrow{f} c\in \cat$, $P(f)$ has a left adjoint $\exists_{P(f)}$ such that both the Frobenius and Beck-Chevalley conditions are satisfied, and so is able to interpret the logical symbols $\{\,\land, \, \top, \, \exists, \, =\,\}$.  A morphism of existential doctrines $(F,a) \colon P \to Q$ is one which preserves the interpretation of these symbols, i.e. a morphism of primary doctrines such that the square
				\[\begin{tikzcd}
					P(c)  \ar{d}{a_c} & \ar{l}[']{\exists_{P(f)}} P(d) \ar{d}{a_d} \\
					Q(F(c)) & \ar{l}[']{\exists_{Q(f)}} Q(F(d))
				\end{tikzcd}\]
				commutes for each arrow $d \xrightarrow{f} c$ of $\cat$.  We denote the resultant 2-full 2-subcategory of ${\bf Doc}$ by ${\bf ExDoc}$.
				
				\item A \emph{coherent doctrine} is an existential doctrine that takes values in the category ${\bf DLat}$ of distributive lattices and lattice homomorphisms, and thus interprets the symbols $\{\,\land, \, \top, \, \exists, \, =, \, \lor, \, \bot \,\}$.  Morphisms of coherent doctrines are those morphisms $(F,a)$ of existential doctrines where $a$ is a natural transformation of ${\bf DLat}$-valued functors.  We denote the resultant 2-full 2-subcategory of ${\bf Doc}$ by ${\bf CohDoc}$.
				
				%\item \emph{Heyting (intuitionistic) doctrines} are coherent doctrines that take values in the category ${\bf Heyt}$ of Heyting algebras and Heyting algebra morphisms, and so interpret (intuitionsitically) the symbols $\{\,\land, \, \top, \, \exists, \, =, \, \lor, \, \bot , \, \to \,\}$.  Analogously with coherent doctrines, we can define a subcategory ${\bf HeytDoc}$ of ${\bf Doc}$ of Heyting doctrines and their morphisms.
				
				\item A \emph{Boolean (classical) doctrine} is a coherent doctrine that takes values in the category ${\bf Bool}$ of Boolean algebras and Boolean algebra homomorphisms, and for each arrow $d \xrightarrow{f} c \in\cat$, $P(f)$ also has a right adjoint $\forall_{P(f)}$.  A Boolean doctrine interprets classical first-order logic.  A morphism of Boolean doctrines is a morphism $(F,a)$ of coherent doctrines that preserves the interpretation of classical first order logic.  We denote the resultant 2-full 2-subcategory of ${\bf Doc}$ by ${\bf BoolDoc}$.
			\end{enumerate}
		}
	\end{exs}
	
	\begin{rem}\label{subscriptcart}{\rm
			We have required existential and coherent doctrines to be examples of primary doctrines, but we shall observe in Example \ref{exmorecoherentdoctrines} how one might define existential doctrines, etc., over non-cartesian base categories.
	}\end{rem}

	What does it mean to say a doctrine `interprets' a certain logical theory?  The precise relationship between doctrines and logical theories was elucidated in Theorem 6.1 and Theorem 6.2 in \cite{seely} -- in short, given a theory $\theory$ in a certain fragment of logic, the doctrine $F^\theory$ from Example \ref{prototypical} is of the appropriate form (e.g. if $\theory$ is a coherent theory, i.e. a theory in the fragment of intuitionistic first order logic containing the symbols $\{\,\land, \, \top, \, \exists, \, =, \, \lor, \, \bot \,\}$, then $F^\theory$ is a coherent doctrine) while conversely, given a doctrine $P \colon \cat^{op} \to \PreOrd$ of a certain form, e.g. a Boolean doctrine, there exists a theory $\theory'$ of the appropriate fragment of logic, in this case classical first-order logic, such that there exists an isomorphism $P \cong F^{\theory'}$. 
	
	While a particular fragment of logical syntax can be effectively modelled by requiring certain categorical structure on a case to case basis, to describe the geometric completion of a doctrine we will need a unified method to keep track of which fragment of logical syntax our doctrine is intended to interpret.  To that end, in the following subsections we discuss the use of Grothendieck topologies in this role.  We illustrate, as is well known, that all the classes of doctrines mentioned in Examples \ref{exs:doctrines}
	%, save Heyting doctrines, 
	can be entirely described by the assignment of certain Grothendieck topologies to each doctrine in the class.
	
	We first motivate this switch from categorical syntax to Grothendieck topologies by discussing models of doctrines in \S \ref{subsec:desired}, before defining the 2-category ${\bf DocSites}$ of doctrinal sites in \S \ref{subsec:docsites}, as well as introducing the 2-subcategory ${\bf GeomDoc} \subseteq {\bf DocSites}$ of geometric doctrines.

	%%%%%%%%%%%%%%%%%%%%
	%%%%%%%%%%%%%%%%%%%%%%
	%%%%%%%%%%%%%%%%%%%
	
	\subsection{Desired models}\label{subsec:desired}

	In addition to the syntactic data of a logical theory $\theory$, represented by the doctrine $F^\theory \colon {\bf Con}_\Sigma \to \Poset$, there is also an intuitive notion of the \emph{semantics} of the theory $\theory$, consisting of a category of models and model homomorphisms.  Similarly, one can define models and model homomorphisms for doctrines interpreting certain syntax.  Importantly, the notion of model of a doctrine depends on the syntax we wish that doctrine to interpret.
	
	\begin{df}[Definition 3.5 \cite{kockreyes}]{\rm
			A \emph{model} of a (primary/existential/coherent/%Heyting/
			Boolean) doctrine $P$ is a (primary/existential/coherent/%Heyting/
			Boolean) doctrine morphism
			\[(F,a) \colon P \to \mathscr{P} ,\]
			where $\mathscr{P} \colon \sets^{op} \to {\bf C}\Bool$ is the powerset doctrine.
	}\end{df}
	
	We define the category of models ${\bf Mod}_{\rm Prim}(P)$ of a primary doctrine $P$ (respectively, ${\bf Mod}_{\rm Ex}(P)$ of a coherent doctrine $P$, etc.) as the category $\Hom_{\bf PrimDoc}(P,\mathscr{P})$ of primary doctrine morphisms $P \to \mathscr{P}$ (resp., $\Hom_{\bf ExDoc}(P,\mathscr{P})$, etc.).  
	
	Let $\theory$ be a theory over a signature $\Sigma$.  The models of $\theory$ are easily seen to coincide with the models of $F^\theory$, considered as a doctrine in the appropriate class: for a model $(F,a) \colon F^\theory \to \mathscr{P}$, the functor $F \colon {\bf Con}_\Sigma \to \sets$ picks out the interpretation of each context in the model, while $a_{\vec{x}} \colon F^\theory(\vec{x}) \to \mathscr{P}(F(\vec{x}))$ sends a proposition in context $\vec{x}$ to its interpretation as a subset of $F(\vec{x})$.  
	
	Note that if $P$ is, say, a coherent doctrine, then there is a chain of inclusions
	\[{\bf Mod}_{\rm Coh}(P) \subseteq {\bf Mod}_{\rm Ex}(P) \subseteq {\bf Mod}_{\rm Prim}(P).\]
	We observe further that, if $P$ is a Boolean doctrine, there is an equivalence
	\[{\bf Mod}_{\rm Coh}(P) \simeq {\bf Mod}_{\rm Bool}(P).\]
	This is because a doctrine morphism $(F,a) \colon P \to \mathscr{P}$ is coherent if and only if it is Boolean.  In one direction, every morphism of Boolean doctrines is necessarily coherent.  Conversely, suppose $(F,a) \colon P\to \mathscr{P}$ is a coherent doctrine morphism.  We deduce that, since complements are unique (see \parencite[\S 4.13]{order}), and we have that, for each $x \in P(c)$,
	\[a_c(x \land \neg x) = a_c(\bot) = \bot, \ \ a_c(x \lor \neg x) = a_c(1) = 1,\]
	the lattice homomorphism $a_c \colon P(c) \to \mathscr{P}(F(c))$ must also preserve negation and hence is a Boolean aglebra homomorphism.  Then, using that, for a Boolean doctrine, $\forall_{P(f)} = \neg \exists_{P(f)} \neg$ for each arrow $d \xrightarrow{f} c$ of $\cat$, we conclude that $(F,a)$ also preserves the interpretation of universal quantification, completing the equivalence.
	
	\paragraph{Models as flat functors.}  A topos-theorist may be more familiar with models of theories in terms of \emph{continuous flat functors}.  We sketch the equivalence of the two approaches for a primary doctrine $P \colon\cat^{op} \to {\bf MSLat}$.  Let $(F,a) \colon P \to \mathscr{P}$ be a model of a doctrine à la Kock and Reyes as above.  We can construct a functor
	\[F \rtimes a \colon \cat \times P \to \sets\]
	where an object $(c,x) \in \cat \rtimes P$ is sent to the subset $a_c(x) \subseteq F(c)$ and an arrow $(c,x) \xrightarrow{f} (d,y)$ is sent to $F(f)|_{a_c(x)} \colon a_c(x) \to a_d(y)$.  If $P \colon \cat^{op} \to \sets$ is a primary doctrine, the functors $G \colon \cat \rtimes P \to \sets$ of the form $F \rtimes a $, for a model
	\[(F,a) \colon P \to \mathscr{P}\]
	of $P$ as a primary doctrine, are precisely the left exact functors (i.e. finite limit preserving), which we also call \emph{flat} functors.  It is easily checked that $F \rtimes a \colon \cat \rtimes P \to \sets$ is flat for each model $(F,a) \colon P \to \mathscr{P}$ of $P$ as a primary doctrine.  Conversely, given a flat functor $G \colon \cat \rtimes P \to \sets$, we can construct a model $(F^G,a^G) \colon P \to \mathscr{P}$ of $P$ by setting $F^G \colon \cat \to \sets$ as the functor that sends $c \in \cat$ to $G(c,\top_c) \in \sets$ (where $\top_c$ is the top element in $P(c)$) and $a^G_c \colon P(c) \to \mathscr{P}(G(c,\top_c))$ sends $x \in P(c)$ to $G(c,x) \subseteq G(c,\top_c)$.  That these constructions extend to an equivalence is easily shown.
	
	\begin{prop}[cf. Theorem D1.4.7 \cite{elephant}]
		For each primary doctrine $P \colon \cat^{op} \to {\bf MSLat}$, there is an equivalence of categories
		\begin{equation}\label{eq:flatandmod2}{\bf Flat}(\cat \rtimes P,\sets) \simeq {\bf Mod}_{\rm Prim}(P)\end{equation}
		where ${\bf Flat}(\cat \rtimes P,\sets)$ represents the category of flat functors $G \colon \cat \rtimes P \to \sets$ and natural transformations between these.
	\end{prop}

	\paragraph{Models as continuous flat functors.}
	
	If $P \colon \cat^{op} \to {\bf MSLat}$ interprets a richer syntax, e.g. $P$ is an existential doctrine, restricting to the (full) subcategory ${\bf Mod}_{\rm Ex}(P) \subseteq {\bf Mod}_{\rm Prim}(P)$ of models of $P$ as an existential doctrine coincides with restricting ${\bf Flat}(\cat \rtimes P,\sets)$ to the full subcategory of continuous flat functors, relative to a certain Grothendieck topology on the category $\cat \rtimes P$.
	
	Recall that, for a cartesian category $\dcat$ endowed with a Grothendieck topology $K$, a flat functor 
	\[G \colon \dcat \to \sets\]
	is said to be $K$-\emph{continuous} if each $K$ covering sieve is sent to a jointly surjective family of functions.  The classes of doctrines we have seen in Examples \ref{exs:doctrines} have their models expressed by continuous flat functors for certain Grothendieck topologies.  Since these equivalences are relatively well-known, and can easily be obtained by restricting the equivalence (\ref{eq:flatandmod2}), we omit the proof.
	
	\begin{prop}[cf. Theorem D3.1.4 \& Theorem D3.1.9 \cite{elephant}]\label{prop:exandcohmodels}
		Let $P \colon \cat^{op} \to {\bf MSLat}$ be a primary doctine.
		\begin{enumerate}
			\item\label{exmodelsi} If $P$ is an existential doctrine, there is an equivalence
			\[J_{\rm Ex}\text{-}{\bf Flat}(\cat \rtimes P,\sets) \simeq {\bf Mod}_{\rm Ex}(P),\]
			where $J_{\rm Ex}$ is the Grothendieck topology on $\cat \rtimes P$ generated by covering families of the form
			\[\begin{tikzcd}
				(d,x) \ar{r}{g} &(c, \exists_g x),
			\end{tikzcd}\]
			for $x \in P(d)$ and $d \xrightarrow{g} c \in \cat$.
			\item\label{cohmodelsi} If $P$ is a coherent doctrine, there is an equivalence
			\[J_{\rm Coh}\text{-}{\bf Flat}(\cat \rtimes P,\sets) \simeq {\bf Mod}_{\rm Coh}(P),\]
			where $J_{\rm Coh}$ is the Grothendieck topology on $\cat \rtimes P$ generated by covering families of the form
			\[\begin{tikzcd}
				(d,x) \ar{r}{g} & (c, \exists_g x \lor \exists_h y ) & \ar{l}[']{h} (e,y),
			\end{tikzcd}\]
			for $x \in P(d)$, $y \in P(d)$, and arrows $d \xrightarrow{g} c , e \xrightarrow{h} c \in \cat$.
			\item If $P$ is a Boolean doctrine, there is an equivalence
			\[J_{\rm Coh}\text{-}{\bf Flat}(\cat \rtimes P,\sets) \simeq {\bf Mod}_{\rm Coh}(P) \simeq {\bf Mod}_{\rm Bool}(P).\]
		\end{enumerate}
	\end{prop}
	
	%\begin{rem}{\rm
	%    In addition to unifying a treatment of 
	%}\end{rem}

%\begin{rem}{\rm
%   Let $P \colon \cat^{op} \to {\bf Heyt}$ be a Heyting doctrine.  As aforementioned, there does not exist a Grothendieck topology $J$ on $\cat \rtimes P$ such that
%  \[J \text{-}{\bf Flat}(\cat \rtimes P, \sets) \simeq {\bf Mod}_{\rm Heyt}(P).\]
% We demonstate this by considering a `propositional' Heyting doctrine, i.e. a single Heyting algebra $H$.

%A model of $H$ is now simply a Heyting algebra morphism $a \colon H \to \2$.
%}\end{rem}

%%%%%%%%%%%%%%%%%%%%%%%%%%%%%
\paragraph{Desired models.}  Since the models of a doctrine $P \colon \cat^{op} \to \PreOrd$ that one wishes to consider depend on the syntax that doctrine is intended to interpret, in what follows we will be deliberately vague and make reference only to the \emph{desired models} of $P$, denoted ${\bf Mod}(P)$, where the category ${\bf Mod}(P)$ is intended to have been intuitively defined, and only clarify as to what the desired models are when needed.

%%%%%%%%%%%%%%%%%%%%%%
%%%%%%%%%%%%%%%%%%%%%%%%%%
%%%%%%%%%%%%%%%%%%%%%%%%

\subsection{Doctrinal sites}\label{subsec:docsites}

Since we wish to exploit the unified treatment of doctrines and their desired models afforded by assigning a Grothendieck topology, we elect to work with the category of \emph{doctrinal sites}.  We sketch that ${\bf PrimDoc}$, ${\bf ExDoc}$, ${\bf CohDoc}$ and ${\bf BoolDoc}$ form 1-full 2-subcategories of the category of doctrinal sites (by 1-full, we mean that the 2-subcategory is full on both 1-cells and 2-cells).  Finally, the category of geometric doctrines is also defined.

Let $\cat$ and $\dcat$ be categories.  If $\cat$ is non-cartesian, we say that a functor $F \colon \cat \to \dcat$ is \emph{flat} if it defines a morphism of sites
\[F \colon (\cat, J_{\rm triv}) \to (\dcat,J_{\rm triv})\]
when both $\cat$ and $\dcat$ are endowed with the trivial topology.  We note that this generalises the definition of flat as synonymous with left exact we have used previously.  A flat functor $F \colon \cat \to \dcat$ induces a geometric morphism $\Sh(F)$ such that the square
\[\begin{tikzcd}
\cat \ar{r}{F} \ar{d}{\yo_\cat} & \dcat \ar{d}{\yo_\dcat} \\
\sets^{\cat^{op}} \ar{r}{\Sh(F)^\ast} & \sets^{\dcat^{op}}
\end{tikzcd}\]
commutes, where $\yo_\cat$ and $\yo_\dcat$ are the Yoneda embeddings which both preserve and reflect limits.  Hence, a flat functor $F \colon \cat \to \dcat$ preserves any finite limits that exist in $\cat$ (cf. \parencite[Corollary VII.6.4]{SGL}).

\begin{df}\label{df:docsites}{\rm 
\begin{enumerate}
\item A \emph{doctrinal site} (also called a \emph{fibred preorder site} in \cite{fibredsites}) consists of a doctrine 
\[P \colon \cat^{op} \to \PreOrd\]
and a Grothendieck topology $J$ on the category $\cat \rtimes P$.
\item The \emph{2-category of doctrinal sites}, denoted by ${\bf DocSites}$, has doctrinal sites as objects.  An arrow 
\[(P,J) \to (Q,K)\]
of ${\bf DocSites}$ is a pair $(F,a)$ consisting of a \emph{flat} functor $F \colon \cat \to \dcat$ and a natural transformation $a \colon P \to Q \circ F$ such that the induced functor
\[F \rtimes a \colon (\cat \rtimes P,J) \to (\dcat \rtimes Q,K)\]
is a morphism of sites.  A 2-cell between two morphisms of doctrines $(F,a),(F',a') \colon (P,J) \to (Q,K)$ is a natural transformation $\alpha \colon F \to F'$ such that $a_c(x) \leqslant Q(\alpha_c)(a'_c(x))$ for each object $c \in \cat$ and $x \in P(c)$.
\end{enumerate}
}\end{df}

We could have also chosen natural transformations $F \rtimes a \to F' \rtimes a'$ as the 2-cells of ${\bf DocSites}$.  We observe that, in the non-pathological cases, the two coincide.  Firstly, every 2-cell $\alpha \colon (F,a) \to (F',a')$ of ${\bf DocSites}$ induces a natural transformation $\breve{\alpha} \colon F \rtimes a \to F' \rtimes a'$, with components ${\breve{\alpha}_{(c,x)} \colon (F(c),a_c(x)) \to (F'(c),a'_c(x))}$ named by the arrows $\alpha_c \colon F(c) \to F'(c)$.  Conversely, if $P \colon \cat^{op} \to \PreOrd$ has non-empty fibres ({$P(c) \neq \emptyset$} for all $c \in \cat$), then any natural transformation $\beta \colon F \rtimes a \to F' \rtimes a'$ yields a natural transformation $\beta' \colon F\to F'$ for which $a_c(x) \leqslant Q(\beta'_c)(a'_c(x))$ by taking $\beta'_c \colon F(c) \to F'(c)$ as the $\dcat$-arrow $\beta_{(c,x)} \colon (F(c),a_c(x)) \to (F'(c),a'_c(x))$, for some $x \in P(c)$.

\begin{ex}{\rm
Although the definition of a morphism of doctrinal sites may appear initially unmotivated within the context of standard doctrine theory, we note that, for many of the examples of doctrines that we have encountered, the definition coincides with the pre-existing notions we have for morphisms of doctrines.  

If $P \colon \cat^{op} \to {\bf MSLat}$ and $Q \colon \dcat^{op} \to {\bf MSLat}$ are both primary doctrines, then a morphism
\[(P,J) \xrightarrow{(F,a)} (Q,K) \in {\bf DocSites}\]
of doctrinal sites is a morphism of primary doctrines in the sense of Example \ref{exs:doctrines}\ref{ex:primdoc} such that
\[F \rtimes a \colon (\cat \rtimes P,J) \to (\dcat \rtimes Q,K)\]
is cover preserving.  In particular, the morphisms $(P,J_{\rm triv}) \to (Q,J_{\rm triv})$ of ${\bf DocSites}$ are precisely morphisms of primary doctrines.  Thus, there exists a full and faithful 2-embedding
\[{\bf PrimDoc} \hookrightarrow {\bf DocSites}\]
that sends a primary doctrine $P \colon \cat^{op} \to {\bf MSLat}$ to the doctrinal site $(P,J_{\rm triv})$.

Similarly, there exist full and faithful 2-embeddings
\[\begin{tikzcd}
{\bf ExDoc} \ar[hook]{rd} & \\
{\bf CohDoc} \ar[hook]{r} & {\bf DocSites} \\
{\bf BoolDoc} \ar[hook]{ru} & 
\end{tikzcd}\]
which send an existential (respectively, coherent/Boolean) doctrine $P \colon \cat^{op} \to {\bf MSLat}$ to the doctrinal site $(P,J_{\rm Ex})$ (respectively, $(P,J_{\rm Coh})$).
}\end{ex}

\paragraph{Geometric doctrines.}  Finally, we describe geometric doctrines, defining the category of geometric doctrines as a full 2-subcategory of ${\bf DocSites}$.

\begin{df}
{\rm
A \emph{geometric doctrine} is a doctrine $\Lb \colon \cat^{op} \to \Frm_{\rm open}$ that defines an internal locale of the presheaf topos $\sets^{\cat^{op}}$, equivalently $\Lb$ satisfies the relative Beck-Chevalley condition (see \cite[Definition 5.1(e)(i)]{fibredsites} or Definition \ref{df:exsites}\ref{df:exsites:relbc}).
}
\end{df}

If $\cat$ has all pullbacks, recall from \S \ref{subsec:intloc} that a geometric doctrine $\Lb \colon \cat^{op} \to \Frm_{\rm open}$ satisfies the relative Beck-Chevalley condition if and only if the standard Beck-Chevalley condition is satisfied.  Thus, over a cartesian category $\cat$, a geometric doctrine $\Lb \colon \cat^{op} \to \Frm_{\rm open}$ is just a coherent doctrine that is valued in $\Frm$.  By the theorems of Seely \parencite[Theorem 6.1 \& Theorem 6.2]{seely}, a geometric doctrine $\Lb \colon \cat^{op} \to \Frm_{\rm open}$, where $\cat$ is a cartesian category, is isomorphic to $F^\theory$ for some theory $\theory$ of geometric first-order logic.

Recall that, to each geometric doctrine $\Lb \colon \cat^{op} \to \Frm_{\rm open}$, we can endow the category $\cat \rtimes \Lb$ with the Grothendieck topology $K_\Lb$.  We denote by ${\bf GeomDoc}$ the full subcategory of ${\bf DocSites}$ on objects of the form $(\Lb,K_\Lb)$ for a geometric doctrine $\Lb$. Observe also, by Proposition \ref{prop:intlocmorph}, that given two geometric doctrines $\Lb, \Lb' \colon \cat^{op} \to \Frm_{\rm open}$ fibred over the same category $\cat$, a morphism $(\Lb,K_\Lb) \to (\Lb',K_{\Lb'}) \in {\bf DocSites}$ is precisely the data of a morphism of internal locales $\Lb' \to \Lb \in \Loc\left(\sets^{\cat^{op}}\right)$, i.e. for each category $\cat$ there is a full and faithful 2-embedding 
\[\begin{tikzcd}
\Loc\left(\sets^{\cat^{op}}\right)^{op} \ar[hook]{r} & {\bf GeomDoc} \subseteq {\bf DocSites}.
\end{tikzcd}\]
In analogy with classifying topos theory (see \parencite[\S 2.1.2]{TST}), we also prescribe that the desired models ${\bf Mod}(\Lb)$ of a geometric doctrine $\Lb \colon \cat^{op} \to \Frm_{\rm open}$ to be one of the equivalent categories:
\begin{itemize}
\item $\Hom_{\bf GeomDoc}(\Lb,\mathscr{P})$,
\item $K_{\Lb}\text{-}{\bf Flat}(\cat \rtimes \Lb,\sets)$,
\item or the category of geometric morphisms $\sets \to \Sh(\Lb)$ and their natural transformations
\end{itemize}
(the equivalence of the above categories can be deduced from Proposition \ref{prop:intlocmorph} and \parencite[Theorem 4.4]{myselfintlocmorph}).

%%%%%%%%%%%%%%%%%%%%%%%%%%%%%%%%%%%%%%%%%%%
\section{The geometric completion of a doctrine}\label{sec:geocomp}

In this section we describe the geometric completion of a doctrinal site.  As aforementioned, this is an application of the fibred ideal completion established in \cite{fibredsites} that generalises the ideal completion of a preorder with a covering system.  We are able to define the geometric completion of a doctrinal site using only the results on internal locales recalled in \S \ref{subsec:intloc}.

Let $(P,J) \in {\bf DocSites}$ be a doctrinal site.  Since $J$ is a Grothendieck topology on the category $\cat \rtimes P$, we can construct the topos of sheaves $\Sh(\cat \rtimes P,J)$.  The projection $p_P \colon \cat \rtimes P \to \cat$ yields a comorphism of sites
\[(\cat \rtimes P) \xrightarrow{p_P} (\cat,J_{\rm triv})\]
whose induced geometric morphism $C_{p_P} \colon \Sh(\cat \rtimes P,J) \to \sets^{\cat^{op}}$ factors as
\[\begin{tikzcd}
\Sh(\cat \rtimes P,J) \ar[tail]{r} & \sets^{(\cat \rtimes P)^{op}} \ar{r} &\sets^{\cat^{op}}.
\end{tikzcd}\]
Since both factors are localic (see \parencite[Examples A4.6.2(a) \& (c)]{elephant}), $C_{p_P}$ is localic by \parencite[Lemma 1.1]{fact1} (alternatively, $C_{p_P}$ is localic by \parencite[Proposition 7.11]{dense} alone).  Thus, by Theorem \ref{localicmorph}, the topos $\Sh(\cat \rtimes P,J)$ is the topos of sheaves on an internal locale (i.e. geometric doctrine) ${C_{p_P}}_\ast(\Omega_{\Sh(\cat \rtimes P, J)}) $ of $\sets^{\cat^{op}}$.

\begin{df}[cf. Definition 6.2 \cite{fibredsites}]\label{df:gc}
{\rm
The \emph{geometric completion} of a doctrinal site $(P,J)$ is the geometric doctrine 
\[{C_{p_P}}_\ast(\Omega_{\Sh(\cat \rtimes P, J)}) \colon \cat^{op} \to \Frm_{\rm open}.\]
We denote the geometric completion of $(P,J)$ by $\GC(P,J)$.  Recall (from \cite[Proposition 4.2]{fibredsites} or \cite[\S 3.1]{myselfintlocmorph}) that the doctrine ${\GC(P,J) \colon \cat^{op} \to \Frm_{\rm open}}$ is isomorphic to the functor
\[\Sub_{\Sh(\cat \rtimes P,J)}(C_{p_P}^\ast\circ \yo_\cat(-)) \colon \cat^{op} \to \Frm_{\rm open}.\]
}
\end{df}

We claim that this choice of geometric completion is 2-functorial in ${\bf DocSites}$, universal, idempotent and semantically invariant.  The proof of these facts is delayed until \S \ref{subsec:univprop} and \S \ref{subsec:monad}.  We proceed as follows.
\begin{itemize}
\item Immediately below, in \S \ref{subsec:altcont}, we recall the explicit description of the geometric completion $\GC(P,J) $ of a doctrinal site $(P,J)$, as described in \cite[\S 6]{fibredsites}.  We also demonstrate that, in the special case where each fibre of $P$ has a top element and these are preserved by transition maps, the geometric completion admits a simpler description.  By showing that the \emph{free top completion} is sub-geometric, we can recover the geometric completion of an arbitrary doctrine.

\item In \S \ref{subsec:univprop}, the unit of the geometric completion is defined and the universal property is proved.
\item We demonstrate the 2-monadic aspects of the geometric completion in \S \ref{subsec:monad} and identify the algebras of the monad as the geometric doctrines.
%\item Finally, in \S \ref{subsec:preserv} we study the preservation of some logically pertinent properties under the unit of the geometric completion.
\end{itemize}

%%%%%%%%%%%%%%%%%%%%%%%%%%%%%%%%%%%%%%%%%%%%%%%
\subsection{Calculating the geometric completion}\label{subsec:altcont}

An explicit description of the geometric completion $\GC(P,J) \colon \cat^{op} \to \Frm_{\rm open}$ for a doctrinal site $(P,J)$ can be computed directly using the description of the subobject classifier of a Grothendieck topos found in \parencite[\S III.7]{SGL}, as is done in \parencite[Proposition 6.2]{fibredsites}.  This returns $\GC(P,J) \colon \cat^{op} \to \Frm_{\rm open}$ as the doctrine where:
\begin{enumerate}
\item for each $c \in \cat$, $\GC(P,J)(c)$ is the frame of $J$-\emph{closed subobjects} in $\sets^{(\cat \rtimes P)^{op}}$ of the presheaf 
\[\Hom_\cat(p_P(-),c) \colon (\cat \rtimes P)^{op} \to \sets\]
(for a description of $J$-closed subobjects, see \cite[\S 2.1]{dense}),
\item and for each arrow $d \xrightarrow{f} c$ of $\cat$, the transition map $\GC(P,J)(f) \colon \GC(P,J)(c) \to \GC(P,J)(d)$ sends a $J$-closed subobject 
\[\varsigma \rightarrowtail \Hom_\cat(p_P(-),c)\]
to the pullback
\[
\begin{tikzcd}
f^\ast (\varsigma) \ar[tail]{d} \ar{r} & \varsigma \ar[tail]{d} \\
\Hom_\cat(p_P(-),d)  \ar{r} & \Hom_\cat(p_P(-),c) .  
\end{tikzcd}
\]
\end{enumerate}
By unravelling definitions, this is equivalent to the concrete description presented below. 
\begin{constr}\label{constrgc}{\rm
Let $P \colon \cat^{op} \to \PreOrd$ be a doctrine and $J$ a Grothendieck topology on $\cat \rtimes P$.  The geometric completion $\GC(P,J) \colon \cat^{op} \to \Frm_{\rm open}$ admits the following description.
\begin{enumerate}
\item For each object $c$ of $\cat$, an element $S$ of $\GC(P,J)(c)$ is a set of pairs $(f,x)$, where $d \xrightarrow{f} c$ is an arrow of $\cat$ and $x \in P(d)$, such that:
\begin{enumerate}
	\item if $(f,x) \in S$, then $(f \circ g, y) \in S$ for each arrow $e \xrightarrow{g} d$ of $\cat$ and $y \in P(e)$ with $y \leqslant P(g)(x)$;
	\item for each arrow $d \xrightarrow{f} c$ of $\cat$, given a subset $\{\,(g_i,y_i) \mid i \in I\,\} \subseteq S$ such that, for each $i \in I$, $g_i$ factors as
	\[\begin{tikzcd}
		e_i \ar{d}{h_i} \ar{rd}{g_i} & \\
		d \ar{r}{f} & c,
	\end{tikzcd}\]
	if there is an $x \in P(d)$ and, for all $i \in I$, $y_i \leqslant P(h_i)(x)$ for which the family
	\[\{\,(e_i,y_i) \xrightarrow{h_i} (d,x) \mid i \in I\,\}\]
	of morphisms in $\cat \rtimes P$ is $J$-covering, then $(f,x) \in S$.
\end{enumerate}
We then order $\GC(P,J)(c)$ by inclusion.
\item For each arrow $d \xrightarrow{f} c$ of $\cat$, $\GC(P,J)(f) \colon \GC(P,J)(c) \to \GC(P,J)(d)$ sends $S \in \GC(P,J)(c)$ to $f^\ast(S)$, where
\[f^\ast(S) = \{\,(g,y) \mid (f \circ g,y) \in S\,\} \in \GC(P,J)(d).\]
\end{enumerate}
}\end{constr}

Clearly, if $J$ and $J'$ are Grothendieck topologies on $\cat \rtimes P$ with $J' \subseteq J$, then $\GC(P,J)(c) \subseteq \GC(P,J')(c)$ for each object $c$ of $\cat$.  Hence, for every $J$, $\GC(P,J)(c)$ is a subset of $\GC(P,J_{\rm triv})(c)$, where $J_{\rm triv}$ is the trivial topology on $\cat \rtimes P$.  An element $S \in \GC(P,J_{\rm triv})(c)$ that is contained in the subset $\GC(P,J) (c)\subseteq \GC(P,J_{\rm triv})(c)$, i.e. $S$ satisfies property (b) above, is said to be $J$-\emph{closed}.  This is precisely what it means for the subobject $\varsigma \rightarrowtail \Hom_\cat(p_P(-),c)$ corresponding to $S$ to be $J$-closed in the sense of \cite{dense}.

A \emph{closure operation for subobjects} is described in \S 2.1 \cite{dense}.  In the particular case of subobjects of the presheaf $\Hom_\cat(p_P(-),c)$, i.e. elements $S \in \GC(P,J)(c)$, the $J$-closure can be understood entirely in terms of internal locale theory.  Since the embedding $\Sh(\cat \rtimes P, J) \rightarrowtail \sets^{(\cat\rtimes P)^{op}}$ is a geometric morphism for which the triangle
\[\begin{tikzcd}[column sep = tiny]
\Sh(\cat \rtimes \GC(P,J),K_{\GC(P,J)}) \simeq \Sh(\cat \rtimes P,J) \ar[tail]{rr} \ar{rd}[']{C_{p_P}} && \sets^{(\cat \rtimes P)^{op}} \simeq \Sh(\cat \rtimes \GC(P,J_{\rm triv}),K_{\GC(P,J_{\rm triv})}) \ar{ld} \\
& \sets^{\cat^{op}} &    
\end{tikzcd}\]
commutes, by \parencite[Theorem 5.7]{myselfintlocmorph}, the geometric morphism is induced by an embedding of internal locales $\GC(P,J) \rightarrowtail \GC(P,J_{\rm triv})$: that is, for each $c \in \cat$, a surjective frame homomorphism 
\[{\overline{(-)}}_c \colon \GC(P,J_{\rm triv})(c) \to \GC(P,J)(c)\]
such that, for each arrow $d \xrightarrow{g} c \in \cat$, the diagram
\[\begin{tikzcd}
\GC(P,J_{\rm triv})(d) \ar{dd}{\overline{(-)}_d} \ar[shift left = 2]{rr}{\exists_{\GC(P,J_{\rm triv})(f)}} && \ar[shift left = 2]{ll}{\GC(P,J_{\rm triv})(f)} \GC(P,J_{\rm triv})(c) \ar{dd}{\overline{(-)}_c} \\
& \\
\GC(P,J)(d) \ar[shift left=2]{rr}{\exists_{\GC(P,J)(c)}} & &\GC(P,J)(c) \ar[shift left=2]{ll}{\GC(P,J)(f)}
\end{tikzcd}\]
is a morphism of adjunctions.

\begin{df}\label{df:closure}{\rm
Let $S$ be an element of $\GC(P,J_{\rm triv})(c)$.  We call the image of $S$ under 
\[{\overline{(-)}}_c \colon \GC(P,J_{\rm triv})(c) \to \GC(P,J)(c)\]
the $J$-\emph{closure} of $S$, and denote it by $\overline{S}$.  The corresponding subobject $\overline{\varsigma}$ of $\Hom_\cat(p_P(-),c)$ is precisely the $J$-closure of $\varsigma$.
}\end{df}

\paragraph{When top elements are available.} 

In the remainder of this subsection we demonstrate that, in the special case of a doctrine $P \colon \cat^{op} \to \PreOrd$ where $P(c)$ has a top element, for each object $c \in \cat$ and $P(f)$ preserves that top element, for each arrow $d \xrightarrow{f} c$ of $ \cat$, the description of the geometric completion given in Construction \ref{constrgc} can be simplified further.

We will then show how the description in Construction \ref{constrgc} can be recovered for an arbitrary doctrinal site $(P,J)$ by freely adding (preserved) top elements to the doctrine $P$.  We do so to illustrate our first example of a \emph{sub-geometric completion}.  A sub-geometric completion is a partial completion to the data of a geometric doctrine which can be `subsumed' by the geometric completion.  This will be explored further in \S \ref{sec:subgeo}.  To avoid confusion in the subsequent paragraphs, we will temporarily relabel the doctrine described in Construction \ref{constrgc} as $\GC'(P,J) \colon \cat^{op} \to \Frm_{\rm open}$ while we prove the isomorphism $\GC(P,J) \cong \GC'(P,J)$.

If, for each object $c $ of $ \cat$, $P(c)$ has a top element $\top_c$ which is preserved by $P(f)$ for each arrow $d \xrightarrow{f} c$ of $ \cat$, then the projection $p_P \colon \cat \rtimes P \to \cat$ has a right adjoint: the functor
\[ t_P \colon \cat \to \cat \rtimes P\]
which sends $c \in \cat$ to $(c, \top_c) \in \cat \rtimes P$.  Thus, we can apply the description of the direct image of $C_{p_P}$ given in \parencite[Theorem VII.10.4]{SGL}, to obtain that
\[\GC(P,J) \cong {C_{p_P}}_\ast (\Omega_{\Sh(\cat \rtimes P,J)}) = \Omega_{\Sh(\cat \rtimes P, J)} \circ t_P^{op} \colon \cat^{op} \to \Frm_{\rm open}.\]
Therefore, using the description of the subobject classifier of $\Sh(\cat \rtimes P,J)$ found in \parencite[\S III.7]{SGL}, for each $c \in\cat$, an element of $\GC(P,J)(c)$ is a $J$-closed sieve $S$ on $(c,\top_c)$ and, for each $d \xrightarrow{f} c \in \cat$, $\GC(P,J)(f)$ sends $S$ to
\[f^\ast(S) = \{\,g \colon (e,V) \to (d,\top_d) \mid f \circ g \colon (e,V) \to (c,\top_c) \in S\,\}.\]
We therefore observe that $\GC(P,J)$ is indeed isomorphic to the doctrine $\GC'(P,J)$ as described in Construction \ref{constrgc}.  The witnessing isomorphism is given by sending a $J$-closed sieve $S$ on $(c,\top_c)$ to the set
\[\left\{\,(f,x) \,\middle\vert\, (d,x) \xrightarrow{f} (c,\top_c) \in S\,\right\} \in \GC'(P,J)(c).\]

\paragraph{The free top completion.}  In the absence of (preserved) top elements, we can freely add them to the doctrine $P \colon \cat^{op} \to \PreOrd$ and demonstrate that, by carefully selecting a Grothendieck topology, we obtain a doctrinal site whose geometric completion is isomorphic to $\GC(P,J)$ -- that is to say, adding top elements is a sub-geometric completion.

\begin{df}{\rm
Let $P \colon \cat^{op} \to \PreOrd$ be a doctrine and let $J$ be a Grothendieck topology on $\cat \rtimes P$.

\begin{enumerate}
\item Denote by $P^\top \colon \cat^{op} \to \sets$ the \emph{free top completion}, the doctrine where:
\begin{enumerate}
	\item for each object $c$ of $\cat$, $P^\top(c)$ is the preorder $P(c) \oplus \top_c$, where a top element $\top_c$ has been freely added to $P(c)$;
	\item for each arrow $d \xrightarrow{f} c$ of $\cat$, $P^\top(f) \colon P^\top(c) \to P^\top(d)$ is the monotone map
	\[
	P^\top(f)(x) =
	\begin{cases}
		P(f)(x) & \text{ if $x \in P(c)$,} \\
		\top_d & \text{ if $x=\top_c$.}
	\end{cases}
	\]
\end{enumerate}
\item We define a Grothendieck topology $J^\top$ on $\cat \rtimes P^\top$ in the following way:
\begin{enumerate}
	\item for each object of the form $(c,x)$, with $x \in P(c)$, a sieve $S $ on $(c,x)$ is $J^\top$-covering if and only if $S$ is $J$-covering;
	\item for an object of the form $(c,\top_c)$, a sieve $S$ on $(c,\top_c)$ is $J^\top$-covering if and only if, for each arrow $(d,x) \xrightarrow{f} (c,\top_c)$ with $x \in P(d)$, the sieve $f^\ast(S)$ on $(d,x)$ is $J$-covering.
\end{enumerate}
\end{enumerate}
}\end{df}
The terminology free top completion is justified as a universal property is clearly satisfied: for any natural transformation $a \colon P \to Q$ between doctrines $P, Q \colon \cat^{op} \to \PreOrd$ where $Q(c)$ has a top element for each $c \in \cat$ which is preserved by $Q(f)$ for each $d \xrightarrow{f} c \in \cat$, there is a unique natural transformation $a^\top \colon P^\top \to Q$ such that the triangle
\[\begin{tikzcd}
P\ar{rd}[']{a} \ar[hook]{r} & P^\top \ar[dashed]{d}{a^\top}\\
& Q
\end{tikzcd}\]
commutes and, for each $c \in \cat$, $a_c^\top$ sends $\top_c \in P^\top(c)$ to the top element of $Q(c)$.

Note that $\cat \rtimes P$ defines a subcategory of $\cat \rtimes P^\top$ and $J$ is the restriction of $J^\top$ to this subcategory.  Note also that, for each $c \in\cat$, the family
\[\{\,(c,x) \xrightarrow{\id_c} (c,\top_c)\mid c \in P(c)\,\}\]
generates a $J^\top$-covering sieve.

\begin{lem}
For each doctrine $P \colon \cat^{op} \to \PreOrd$ and Grothendieck topology $J$ on $\cat \rtimes P$, $J^\top$ is a Grothendieck topology on $\cat \rtimes P^\top$.
\end{lem}
\begin{proof}
The maximality and stability conditions for $J^\top$ are trivially satisfied since $J$ is a Grothendieck topology on $\cat\rtimes P$, as is the transitivity condition $J^\top$ for sieves on objects of the form $(d,x)$ with $x \in P(d)$.

It thus remains to show that if $S$ is a $J^\top$-covering sieve on $(c,\top_c)$ and $R$ is a sieve on $(c,\top_c)$ such that $h^\ast(R) \in J^\top(d,x)$ for each arrow $(e,y) \xrightarrow{h} (c,\top_c)$ in $S$, then $R$ is $J^\top$-covering, i.e. $f^\ast(R) \in J(d,x)$ for each arrow $(d,x) \xrightarrow{f} (c,\top_c) $ with $x \in P(d)$.  As $f^\ast(S) \in J(d,x)$ and, for each arrow $(e,y) \xrightarrow{k} (d,x)$ of $f^\ast(S)$, i.e. for which $(e,y) \xrightarrow{k} (d,x) \xrightarrow{f} (c,\top_c)$ is an element of $S$, we have that $f^\ast(R)$ is $J$-covering since $k^\ast(f^\ast(R)) = (f \circ k)^\ast(R)$ is $J^\top$-covering (and so $J$-covering) by the transitivity condition for $J$.  Thus, by definition, $R$ is $J^\top$-covering.
\end{proof}

\begin{lem}\label{addingtopeq}
The site $(\cat \rtimes P,J)$ is a dense subsite of $(\cat \rtimes P^\top,J^\top)$.
\end{lem}
\begin{proof}
Immediate since $\cat \rtimes P$ is a full subcateogory of $\cat \rtimes P^\top$ and the only objects not contained in $\cat \rtimes P $, i.e. those objects of the form $(c,\top_c)$, are covered by objects contained in the subcategory.
\end{proof}

\paragraph{The free top completion is sub-geometric.}  Having developed the free top completion for a doctrinal site, we can finally observe that this constitutes a sub-geometric completion in the current loose sense that $\GC(P,J) \cong \GC(P^\top,J^\top)$ (a formal definition of sub-geometricity is provided in \S \ref{subsec:gensubgeo}).  As a consequence, we obtain the isomorphism $\GC(P,J) \cong \GC'(P,J)$ as desired.

\begin{prop}\label{topissubgeo}
There is a chain of isomorphisms of doctrines:
\[\GC(P,J) \cong \GC(P^\top,J^\top) \cong \GC'(P^\top,J^\top) \cong \GC'(P,J).\]
\end{prop}
\begin{proof}
That $\GC(P,J) \cong \GC(P^\top,J^\top)$ follows since the toposes $\Sh(\cat \rtimes P,J)$ and $\Sh(\cat \rtimes P^\top,J^\top)$ are equivalent by Lemma \ref{addingtopeq}.  That $\GC(P^\top,J^\top) \cong \GC'(P^\top,J^\top)$ follows as $P^\top$ has (preserved) top elements.

We will sketch the isomorphism between $\GC'(P^\top,J^\top)$ and $\GC'(P,J)$.  We observe that each $J^\top$-closed sieve $S$ on $(c, \top_c)$ is uniquely determined by the set
\[l_c(S) = \{\,(f,x) \mid (d,x) \xrightarrow{f} (c,\top_c) \in S, \, x \in P(d) \,\} \in \GC'(P,J).\]
If $l_c(S) = l_c(S')$, then $S$ and $S'$ agree on arrows of the form $(d,x) \xrightarrow{f} (c,\top_c)$ with $x \in P(d) \subseteq P^\top(d)$.  Conversely, if $(d,\top_d) \xrightarrow{f} (c,\top_c) \in S$, then both $S$ and $S'$ contain the family
\[R = \{\,(d,x) \xrightarrow{\id_d} (d,\top_d) \xrightarrow{f} (c,\top_c) \mid x \in P(d)\,\}\]
which covers $(d,\top_d) \xrightarrow{f} (c,\top_c)$.  Hence, $(d,\top_d) \xrightarrow{f} (c,\top_c) \in S'$ too.  The same argument with $S$ and $S'$ swapped completes the proof that $l_c(S) = l_c(S')$ implies that $S = S'$.  The maps $l_c$, for each $c \in\cat$, are thus evidently the components of a natural isomorphism between $\GC'(P^\top,J^\top)$ and $\GC'(P,J)$.
\end{proof}

%%%%%%%%%%%%%%%%%%%%%%%%%%%%%%%%%%%%%%%%%%

%%%%%%%%%%%%%%%%%%%%%%%%%%%%%%%%%%%%%%%%%%%
\subsection{Universal property of the geometric completion}\label{subsec:univprop}

We are now able to prove that the geometric completion of a doctrinal site is universal in ${\bf DocSites}$, idempotent and semantically invariant as claimed.  We first recall the construction of the unit of the geometric completion, which, unsurprisingly, is the same as the unit of the fibred ideal completion defined in \cite[Proposition 6.2]{fibredsites}, before turning to the universal property of the geometric completion, which extends the universal property of \cite{fibredsites}.  Finally, we discuss some of the basic preservation properties of the unit.

\paragraph{The unit of the geometric completion.}

%To define the unit of the geometric completion, a natural transformation $\eta^{(P,J)} \colon P \to \GC(P,J)$, we first define a natural embedding $P \to \GC(P,J_{\rm triv})$ with which we will compose the $J$-closure natural transformation constructed in \S \ref{Jclosuresec}.
The unit generalises the notion of taking the closure of a principal down-set for a preorder with a covering system.  Let $P \colon \cat^{op} \to \PreOrd$ be a doctrine and let $J$ be a Grothendieck topology on $\cat \rtimes P$. For each object $c \in \cat$ and $x \in P(c)$, the set
\[\downarrow \! x = \{\,(g,y) \mid e \xrightarrow{g} c \in \cat, \, y \in P(e) \text{ and } y \leqslant P(g)(x)\,\}\]
is an object of $\GC(P,J_{\rm triv}) (c)$.  For each arrow $d \xrightarrow{f} c$, $\GC(P,J_{\rm triv})(f)$ acts on this element by
\begin{equation*}
\begin{split}
\GC(P,J_{\rm triv})(f)(\downarrow\! x) &= \{\,(h,y) \mid (f \circ h,y) \in \, \downarrow \! x\,\}, \\
& = \{\,(h,y) \mid e \xrightarrow{h} c \in \cat, \, y \in P(e) \text{ and } y \leqslant P(f \circ g)(x)\,\}, \\
& = \{\,(h,y) \mid e \xrightarrow{h} c \in \cat, \, y \in P(e) \text{ and } y \leqslant P(g)P(f)(x)\,\} = \,\downarrow\! P(f)(x).
\end{split}
\end{equation*}
Hence, we obtain a natural transformation $\downarrow\!(-) \colon P \to \GC(P,J_{\rm triv})$.

\begin{df}
{\rm
Let $(P,J)$ be a doctrinal site.  We will use $\eta^{(P,J)} \colon P \to \GC(P,J)$ to denote the the composite natural transformation
\[\begin{tikzcd}
P \ar{r}{\downarrow\!(-)} & \GC(P,J_{\rm triv}) \ar{r}{\overline{(-)}} & \GC(P,J).
\end{tikzcd}\]
%and occasionally $\overdownar{(-)}$ also when there is no confusion as to which doctrine $P$ or Grothendieck topology $J$ we are referring to.
}
\end{df}

The natural transformation $\eta^{(P,J)}$ yields a morphism of doctrinal sites
\[(\id_\cat,\eta^{(P,J)}) \colon ( P,J) \to (\GC(P,J),K_{\GC(P,J)}).\]
In fact, as shown in \parencite[Proposition 7.2]{fibredsites}, one can prove a stronger statement.

\begin{prop}[Proposition 7.2 \cite{fibredsites}]\label{unitisdense}
The induced functor $\id_\cat \rtimes \eta^{(P,J)} \colon \cat \rtimes P \to \cat \rtimes \GC(P,J)$ yields a dense morphism of sites
\[\id_\cat \rtimes \eta^{(P,J)} \colon (\cat \rtimes P,J) \to (\cat \rtimes \GC(P,J),K_{\GC(P,J)}).\]
\end{prop}

From Proposition \ref{unitisdense}, we immediately deduce that for each object $c$ of $\cat$ and each $S \in \GC(P,J)(c)$,
\begin{equation}\label{eq:unitiscovering}
S = \bigvee_{(f,x) \in S} \exists_{\GC(P,J)(f)}\left( \eta_d^{(P,J)}({x})\right).
\end{equation}
We will frequently abuse notation and write $\eta^{(P,J)}$ for the natural transformation, the morphism of doctrinal sites $(\id_\cat,\eta^{(P,J)}) \colon ( P,J) \to (\GC(P,J),K_{\GC(P,J)})$, and the functor $\id_\cat \rtimes \etaPJ \colon \cat \rtimes P \to \cat \rtimes \GC(P,J)$.

\paragraph{The universal property of the geometric completion.}

\begin{thm}\label{thm:univprop}
To each doctrine $P \colon \cat^{op} \to \PreOrd$ and Grothendieck topology $J$ on $\cat \rtimes P$, the natural transformation $\eta^{(P,J)} \colon P \to \GC(P,J)$ constitutes the unit of the geometric completion of $(P,J)$ for which the following properties are satisfied.
\begin{enumerate}
\item {\bf Universality:} for each morphism of doctrinal sites $(F,a) \colon (P,J) \to (\Lb,K_\Lb)$ to a geometric doctrine $\Lb \colon \dcat^{op} \to \Frm_{\rm open}$, there exists a unique morphism of geometric doctrines $(F,\mathfrak{a}) \colon \GC(P,J) \to \Lb $ such that the triangle
\begin{equation*}
%\label{triag:universal}
\begin{tikzcd}
	P \ar{rd}[']{a} \ar{r}{\eta^{(P,J)}} & \GC(P,J) \ar[dashed]{d}{\mathfrak{a}} \\
	& \Lb \circ F^{op}
\end{tikzcd}
\end{equation*}
commutes;
\item {\bf Semantic invariance:} if the desired models ${\bf Mod}(P)$ of $P$ are equivalent to $J \text{-{\bf Flat}}(\cat \rtimes P,\sets)$, then ${\bf Mod}(P) \simeq {\bf Mod}(\GC(P,J))$;
\item {\bf Idempotency:} for each doctrinal site $(P,J)$, we have that $\GC(P,J) \cong \GC(\GC(P,J),K_{\GC(P,J)})$.
\end{enumerate}
\end{thm}
\begin{proof}
Let $(F,a) \colon (P,J) \to (\Lb,K_\Lb)$ be a morphism of doctrinal sites.  By Lemma \ref{fibrcomoprhmorphthm}, there exists a commutative square of geometric morphisms
\[\begin{tikzcd}
\Sh(\Lb) \simeq \Sh(\dcat \rtimes \Lb,K_\Lb) \ar{d}{C_{p_\Lb}} \ar{rr}{\Sh(F \rtimes a)} &&\Sh(\cat\rtimes P,J) \ar{d}{C_{p_P}} \\
\sets^{\dcat^{op}} \ar{rr}{\Sh(F)} &&\sets^{\cat^{op}}.
\end{tikzcd}\]
Let $g \colon \Sh(\Lb) \to \sets^{\cat^{op}}$ denote the composite geometric morphism
\[\begin{tikzcd}
\Sh(\Lb) \ar{r}{C_{p_\Lb}} & \sets^{\dcat^{op}} \ar{r}{\Sh(F)} & \sets^{\cat^{op}}.
\end{tikzcd}\]
By \parencite[Proposition 7.2]{fibredsites}, the factoring topos in the hyperconnected-localic factorisation of $g$ is given by the topos of sheaves on the internal locale $g_\ast(\Omega_{\Sh(\Lb)}) \in \sets^{\cat^{op}}$.  Whence, we have that
\begin{align*}
g_\ast(\Omega_{\Sh(\Lb)}) &= \Sh(F)_\ast \circ {C_{p_\Lb}}_\ast(\Omega_{\Sh(\Lb)}), \\
& = {C_{p_\Lb}}_\ast(\Omega_{\Sh(\Lb)}) \circ F^{op}, \\
& \cong \Lb \circ F^{op}.
\end{align*}
Therefore, as $C_{p_P}$ is localic, there is a factorisation of $\Sh(F \rtimes a)$:
\[\begin{tikzcd}
\Sh(\Lb) \ar{d}{C_{p_\Lb}} \ar[bend left]{rr}{\Sh(F\rtimes a)} \ar{r} & \Sh(\Lb \circ F^{op}) \ar{d}{C_{p_{\Lb \circ F^{op}}}} \ar[dashed]{r} &\Sh(\cat\rtimes P,J) \ar[bend left = 1.5em]{ld}{C_{p_P}} \\
\sets^{\dcat^{op}} \ar{r}{\Sh(F)} &\sets^{\cat^{op}}. &
\end{tikzcd}\]
Thus, as $\Sh(\cat \rtimes P,J) \simeq \Sh(\GC(P,J))$, there is a commutative triangle of geometric morphisms
\[\begin{tikzcd}[column sep = tiny]
\Sh(\Lb \circ F^{op}) \ar{rd}[']{C_{p_{\Lb \circ F^{op}}}} \ar{rr} && \Sh(\GC(P,J)) \ar{ld}{C_{p_{\GC(P,J)}}} \\
& \sets^{\cat^{op}}. &
\end{tikzcd}\]
Therefore, by Propositon \ref{prop:intlocmorph}, we obtain a morphism of internal locales $\Lb \circ F^{op} \to \GC(P,J)$, or rather a morphism of geometric doctrines $(F,\mathfrak{a}) \colon \GC(P,J) \to \Lb$, satisfying the required conditions.

That the geometric completion is semantically invariant follows from Diaconescu's equivalence:
\begin{align*}
{\bf Mod}(P)\simeq J\text{-{\bf Flat}}(\cat \rtimes P,\sets) 
& \simeq \Geom(\sets, \Sh(\cat\rtimes P,J)) , \\
& \simeq \Geom(\sets, \Sh(\cat\rtimes \GC(P,J),K_{\GC(P,J)})), \\
& \simeq K_{\GC(P,J)}\text{-{\bf Flat}}(\cat \rtimes \GC(P,J),\sets) \simeq {\bf Mod}(\GC(P,J)).
\end{align*}
That the geometric completion is idempotent follows from the fact that
\[\GC(P,J) \cong {C_{p_{\GC(P,J)}}}_\ast(\Omega_{\Sh(\GC(P,J))}) = \GC(\GC(P,J),K_{\GC(P,J)}).\]
\end{proof}

\begin{rem}\label{rem:direct}
{\rm
A direct proof of Theorem \ref{thm:univprop}, without mention of internal locales, could also be given.  Given a morphism of doctrinal sites $(F,a) \colon (P,J) \to (\Lb,K_\Lb)$, where $\Lb \colon \dcat^{op} \to \Frm_{\rm open}$ is a geometric doctrine,  we obtain the unique morphism of geometric doctrines $(F,\mathfrak{a}) \colon \GC(P,J) \to \Lb$ that makes the triangle
\begin{equation*}
%\label{triag:universal}
\begin{tikzcd}
	P\ar{rd}[']{a} \ar{r}{\eta^{(P,J)}} & \GC(P,J) \ar[dashed]{d}{\mathfrak{a}} \\
	& \Lb \circ F^{op}
\end{tikzcd}
\end{equation*}
commute by defining, for each $S \in \GC(P,J)(c)$,
\[\mathfrak{a}_c(S) = \bigvee_{(g,x) \in S} \exists_{\Lb(F(g))} a_d(x).\]
}
\end{rem}

\begin{rem}
{\rm
Given a theory $\theory$ over a signature $\Sigma$, the category ${\bf Con}_\Sigma$ of contexts is normally considered to be entirely \emph{algebraic} in content.  That is to say, the semantics of the empty theory $\mathbb{O}_\Sigma$ over the signature $\Sigma$ are equivalent to the flat functors ${\bf Flat} ({\bf Con}_{\Sigma},\sets)$ (cf. \parencite[Corollary D3.1.2]{elephant} or \parencite[\S D3.2]{elephant}).  In order to amplify the analogy with theories, we have elected to work with doctrinal sites $(P,J)$, where only the category $\cat \rtimes P$ is endowed with a Grothendieck topology $J$ representing richer syntax, while the base category $\cat$ is effectively treated as being endowed with the trivial topology.

If we wished, we could rectify this myopia by considering the 2-category ${\bf DocSites}_{\rm WTB}$, the 2-category of doctrinal sites with topologies on the base category,
\begin{enumerate}
\item whose objects are quadruples $(\cat,J,P,K)$ where $J$ is a Grothendieck topology on the category $\cat$, $P \colon \cat^{op} \to \PreOrd$ is a doctrine fibred over $\cat$, and $K$ is a Grothendieck topology on $\cat \rtimes P$ that contains the Giraud topology $J_{p_P}$ (see \parencite[Proposition 2.1]{giraud}),
\item a 1-cell $(F,a) \colon (\cat,J,P,K) \to (\dcat,J',Q,K')$ of ${\bf DocSites}_{\rm WTB}$ consists of the data of a morphism of sites
\[F \colon (\cat,J) \to (\dcat,J')\]
and a natural transformation $a \colon P \to Q \circ F^{op}$ such that
\[F \rtimes a \colon (\cat \rtimes P,K) \to (\dcat \rtimes Q,K')\]
is also a morphism of sites.  The 2-cells we include are the same as for ${\bf DocSites}$.
\end{enumerate}

We note that since for each object $(\cat,J,P,K) \in {\bf DocSites}_{\rm WTB}$ the geometric morphism 
\[C_{p_P} \colon \Sh(\cat \rtimes P,K) \to \sets^{\cat^{op}}\]
factors through $\Sh(\cat ,J) \rightarrowtail \sets^{\cat^{op}}$, by Corollary \ref{intlocoversites} the Grothendieck topology $K_{\GC(P,K)}$ on $\cat \rtimes \GC(P,K)$ contains the Giraud topology $J_{p_{\GC(P,K)}}$, where $\GC(P,K)$ denotes the geometric completion of $(P,K)$ as in Definition \ref{df:gc}.  Hence, $(\cat,J,\GC(P,K),K_{\GC(P,K)})$ defines an object of ${\bf DocSites}_{\rm WTB}$.

Now for each arrow $(F,a) \colon (\cat,J,P,K) \to (\dcat,J',Q,K')$ of ${\bf DocSites}_{\rm WTB}$, there is a commutative square
\[\begin{tikzcd}
\cat \rtimes P \ar{r}{F \rtimes a} \ar{d}{p_P} & \dcat \rtimes Q \ar{d}{p_Q} \\
\cat \ar{r}{F} & \dcat
\end{tikzcd}\]
such that, when each category is endowed with its respective Grothendieck topology, the vertical arrows $p_P$, $p_Q$ become comorphisms of sites while the horizontal arrows $F \rtimes a$ and $F$ are, by definition, morphisms of sites.  Thus, by applying Lemma \ref{fibrcomoprhmorphthm}, there is a commutative square of geometric morphisms
\[\begin{tikzcd}
\Sh(\cat \rtimes P,K)  \ar{d}{C_{p_P}} & \ar{l}[']{\Sh(F \rtimes a)} \Sh(\dcat \rtimes Q,K') \ar{d}{C_{p_Q}} \\
\Sh(\cat,J) & \ar{l}[']{\Sh(F)}  \Sh(\dcat,J').
\end{tikzcd}\]
Therefore, by applying a similar method to that employed in Theorem \ref{thm:univprop}, we can deduce that
\[(\cat,J,\GC(P,K),K_{\GC(P,K)})\]
is the universal completion of $(\cat,J,P,K)$ to an object of ${\bf DocSites}_{\rm WTB}$ of the form $(\dcat,J',\Lb,K_\Lb)$ for an internal locale $\Lb \colon \dcat \to \Frm_{\rm open}$ of $\Sh(\dcat,J')$.  

The universal property of the geometric completion as stated in Theorem \ref{thm:univprop} is therefore the restriction of this more general statement to the 1-full 2-subcategory of ${\bf DocSites}_{\rm WTB}$ on objects of the form $(\cat,J_{\rm triv},P,K)$, i.e. the 2-category ${\bf DocSites}$ from Definition \ref{df:docsites}.  However, as explained above, for the purposes of our intended applications the extra generality is not needed.
}
\end{rem}

%%%%%%%%%%%%
%%%%%%%%%%%%%

%%%%%%%%%%%%%%%%%
%%%%%%%%%%%%%%%%

\paragraph{Preservation properties of the unit.}

The unit $\eta^{(P,J)} \colon P \to \GC(P,J)$ preserves finite meets.  This can be seen directly.  If the meet $x \land y$ of two elements $x,y \in P(c)$ exists, or if $P(c)$ has a top element $\top_c$, then the meet $\downarrow \!  x \, \land \downarrow \! y \in \GC(P,J_{\rm triv})(c)$ is given by $\downarrow \! (x \land y)$ and $\downarrow\! \top_c$ defines a top element of $\GC(P,J_{\rm triv})(c)$.  Thus, as $\overline{(-)}_c \colon \GC(P,J_{\rm triv}) \to \GC(P,J)$ preserves finite meets as well, so does their composite $\eta^{(P,J)}_{c}$.

It is also easily recognised that joins and existential quantifiers are preserved by the unit $\eta^{(P,J)} $ if and only if the Grothendieck topology $J$ is of a certain form.  Given a subset $\{\,y_i \mid i \in I\,\} \subseteq P(c)$ whose join $\bigvee_{i \in I} y_i$ exists in $P(c)$, since 
\[\id_\cat \rtimes \eta^{(P,J)}\colon (\cat \rtimes P,J) \to (\cat \rtimes \GC(P,J),K_{\GC(P,J)})\]
is cover preserving and reflecting by Proposition \ref{unitisdense}, we have that
\[\eta^{(P,J)}_c\left(\bigvee_{i \in I} y_i\right) = \bigvee_{i \in I} \eta^{(P,J)}_c(y_i)\]
if and only if
\[\left\{\, (c,y_i) \xrightarrow{\id_c} \left(c,\bigvee_{i \in I} y_i\right) \,\middle\vert\,  i \in I \,\right\}\]
is a $J$-covering family.  Identically, if $P(f) \colon P(c) \to P(d)$ has a left adjoint $\exists_{P(f)}$, then, for each $x \in P(d)$,
\[\eta^{(P,J)}_c \circ \exists_{P(f)}(x) = \exists_{\GC(P,J)} \circ \eta^{(P,J)}_d(x),\] if and only if the singleton
\[\left\{\,(d,x) \xrightarrow{f} (c,\exists_{P(f)} x)\,\right\}\]
is a $J$-covering arrow.

%%%%%%%%%%%%%%%%%%%%%%%%
\subsection{The geometric completion as a monad}\label{subsec:monad}

In \S 5 of \cite{trotta}, the language of 2-monad theory is used to describe the universal property of the existential completion.  This is expanded upon in \cite{trotta2} into a rich description of the logical completions of elementary doctrines via 2-monad theory.  Thus inspired, we will use the language of 2-monad theory for investigating the geometric completion.

Recall that a 2-monad on a 2-category $\cat$ is a triple $(T,\eta,\mu)$ consisting of an 2-endofunctor $T \colon \cat \to \cat$, and 2-natural transformations $\eta \colon \id_\cat \to T$ and $\mu \colon T \circ T \to T$ such that the diagrams
\begin{equation}\label{monadaxioms}
\begin{tikzcd}
T^3 \ar{d}{\mu T} \ar{r}{T \mu} & T^2 \ar{d}{\mu} & & \id_\cat \circ T \ar{r}{\eta T} \ar[equals]{rd} &T^2 \ar{d}{\mu}& \ar{l}[']{T \eta} T \circ \id_\cat \ar[equals]{ld} \\
T^2 \ar{r}{\mu} & T, && & T & 
\end{tikzcd}
\end{equation}
strictly commute.  The geometric completion will be a 2-monad on the 2-category ${\bf DocSites}$.

We will initially develop the 1-monadic structure, and add the 2-monadic structure in Proposition \ref{prop:strict2adj}.  For any morphism $ (F,a) \colon (P,J)\to (Q,K)$ of doctrinal sites, there exists a morphism of geometric doctrinal sites $(F,\mathfrak{a}) \colon \GC(P,J) \to \GC(Q,K)$ by the universal property of the geometric completion:
\[\begin{tikzcd}
(P,J) \ar{rr}{(\id_\cat,\eta^{(P,J)})} \ar{d}{(F,a)} && (\GC(P,J),K_{\GC(P,J)}) \ar[dashed]{d}{(F,\mathfrak{a})} \\
(Q,K) \ar{rr}{(\id_\dcat,\eta^{(Q,K)})} && (\GC(Q,K),K_{\GC(Q,K)}).
\end{tikzcd}\]
Thus, the geometric completion is \emph{1-functorial} in that it yields a 1-functor
\[\GC \colon {\bf DocSites} \to {\bf GeomDoc}.\]

The universal property of the geometric completion ensures that the functor $\GC$ is a left 1-adjoint to the inclusion of geometric doctrines into doctrinal sites:
\begin{equation}\label{geocompadj}
\begin{tikzcd}
{{\bf DocSites}} && {{\bf GeomDoc}.}
\arrow[""{name=0, anchor=center, inner sep=0}, shift left=2, hook', from=1-3, to=1-1]
\arrow[""{name=1, anchor=center, inner sep=0}, "\GC", shift left=2, from=1-1, to=1-3]
\arrow["\dashv"{anchor=center, rotate=-90}, draw=none, from=1, to=0]
\end{tikzcd}
\end{equation}
The unit of the adjunction is given by natural transformation $\eta \colon \id_{\bf DocSites} \to \GC$ whose component at a doctrinal site $(P,J) \in {\bf DocSites}$ is $\etaPJ \colon (P,J) \to (\GC(P,J),K_{\GC(P,J)})$.  The counit of the adjunction is the natural transformation whose component at a geometric doctrine $\Lb$ is the isomorphism of geometric doctrines $\Lb \cong \GC(\Lb,K_\Lb)$ induced by the equivalence of toposes
\[\Sh(\Lb) \simeq \Sh(\GC(\Lb,K_\Lb)).\]

In Proposition \ref{prop:strict2adj} below we add the 2-monadic aspects.  The strict 2-adjunction we prove extends the 2-adjunction found in \parencite[Theorem 7.1]{fibredsites}, which presents the universal property of the geometric completion without base change (i.e. all doctrines considered are fibred over the same base category).

\begin{prop}[cf. Theorem 7.1 \cite{fibredsites}]\label{prop:strict2adj}
The geometric completion $\GC \colon {\bf DocSites} \to {\bf GeomDoc}$ can be made into a 2-functor such that 
\[\begin{tikzcd}
{{\bf DocSites}} && {{\bf GeomDoc}.}
\arrow[""{name=0, anchor=center, inner sep=0}, shift left=2, hook', from=1-3, to=1-1]
\arrow[""{name=1, anchor=center, inner sep=0}, "\GC", shift left=2, from=1-1, to=1-3]
\arrow["\dashv"{anchor=center, rotate=-90}, draw=none, from=1, to=0]
\end{tikzcd}\]
is a strict 2-adjunction.
\end{prop}
\begin{proof}
We first show that $\GC$ can be made 2-functorial.  Let $(F,a), (F',a') \colon (P,J) \rightrightarrows (Q,K)$ be morphisms of doctrinal sites.  We will show that every natural transformation $\alpha \colon F \to F'$ that defines a 2-cell between morphisms of doctrinal sites
\[\begin{tikzcd}
{(P,J)} && {(Q,K)}
\arrow[""{name=0, anchor=center, inner sep=0}, "{(F,a)}", curve={height=-12pt}, from=1-1, to=1-3]
\arrow[""{name=1, anchor=center, inner sep=0}, "{(F',a')}"', curve={height=12pt}, from=1-1, to=1-3]
\arrow["\alpha", shorten <=3pt, shorten >=3pt, Rightarrow, from=0, to=1]
\end{tikzcd}\]
also yields a 2-cell of morphisms of geometric doctrines
\[\begin{tikzcd}
{\GC(P,J)} && {\GC(Q,K).}
\arrow[""{name=0, anchor=center, inner sep=0}, "{(F,\mathfrak{a})}", curve={height=-12pt}, from=1-1, to=1-3]
\arrow[""{name=1, anchor=center, inner sep=0}, "{(F',\mathfrak{a}')}"', curve={height=12pt}, from=1-1, to=1-3]
\arrow["\alpha", shorten <=3pt, shorten >=3pt, Rightarrow, from=0, to=1]
\end{tikzcd}\]
I.e. we must show that, for each $c \in \cat$ and $S \in \GC(P,J)(c)$,
\begin{equation}\label{ineqwewant}
\mathfrak{a}_c(S) \leqslant \GC(Q,K)(\alpha_c)(\mathfrak{a}'_c(S)).
\end{equation}
%We give a direct proof here.  An alternate argument relating 2-cells between morphisms of doctrinal sites and 2-cells of geometric morphisms is given in Remark \ref{rem:cylinder}.

If $S$ is of the form $\eta^{(P,J)}_c(x)$, for some $x \in P(c)$, then the inequality (\ref{ineqwewant}) follows by applying the monotone map $\eta^{(Q,K)}_c$ to both sides of the existing inequality $a_c(x) \leqslant Q(\alpha_c)(a'_c(x))$.

For an arbitrary $S \in \GC(P,J)$, we use the description of $\mathfrak{a}$ and $\mathfrak{a}'$ given in Remark \ref{rem:direct}.  For each $(f,x) \in S$, we have that
\begin{align*}
\eta_{F(d)}^{(Q,K)}(a_c(x)) \leqslant \GC(Q,K)(\alpha_d) \left(\eta_{F'(d)}^{(Q,K)}(a'_c(x))\right) & \implies \exists_{\alpha_d} \eta_{F(d)}^{(Q,K)}(a_c(x)) \leqslant  \eta_{F'(d)}^{(Q,K)}(a'_c(x)), \\
& \implies \exists_{F'(f)} \exists_{\alpha_d} \eta_{F(d)}^{(Q,K)}(a_c(x)) \leqslant  \exists_{F'(f)}\eta_{F'(d)}^{(Q,K)}(a'_c(x)).
\end{align*}
Since $\alpha \colon F \to F'$ is natural, $F'(f) \circ \alpha_d = \alpha_c \circ F(f)$, and so
\[\exists_{\alpha_c} \exists_{F(f)} \eta_{F(d)}^{(Q,K)}(a_c(x)) \leqslant  \exists_{F'(f)}\eta_{F'(d)}^{(Q,K)}(a'_c(x)),\]
from which we deduce that $ \exists_{F(f)} \eta_{F(d)}^{(Q,K)}(a_c(x)) \leqslant \GC(Q,K)(\alpha_c) \left( \exists_{F'(f)}\eta_{F'(d)}^{(Q,K)}(a'_c(x))\right)$.  Hence, we obtain the desired inequality:
\[
\mathfrak{a}_c(S) = \bigvee_{(f,x) \in S}  \exists_{F(f)} \eta_{F(d)}^{(Q,K)}(a_c(x)) \leqslant \bigvee_{(f,x) \in S}\GC(Q,K)(\alpha_c) \left( \exists_{F'(f)}\eta_{F'(d)}^{(Q,K)}(a'_c(x)) \right)= \GC(Q,K)(\alpha_c)(\mathfrak{a}'_c(S)).
\]

It remains to show that, for each doctrinal site $(P,J) \in {\bf DocSites}$ and geometric doctrine $\Lb \in {\bf GeomDoc}$, there is a natural isomorphism of categories
\begin{equation}\label{isoof2cats}
\Hom_{{\bf DocSites}}((P,J),(\Lb,K_\Lb)) \cong \Hom_{{\bf GeomDoc}}(\GC(P,J),\Lb).
\end{equation}
The isomorphism on objects is provided by the universal property of the geometric completion.  Given a pair of morphisms of doctrines $(F,a),(F',a') \colon (P,J) \to (\Lb,K_\Lb)$ and a natural transformation $\alpha \colon F \to F'$, if $a_c(x) \leqslant \Lb(\alpha_c)(a'_c(x))$ for all $c \in \cat$ and $x \in P(c)$, i.e. $\alpha$ defines a 2-cell $\alpha \colon (F,a) \to (F',a')$, then $\alpha$ also defines a 2-cell $\alpha \colon (F,\mathfrak{a}) \to (F',\mathfrak{a}')$ by above.  Conversely, if $\mathfrak{a}_c(S) \leqslant \Lb(\alpha_c)(\mathfrak{a}'_c(S))$ for all $ c \in \cat$ and $S \in \GC(P,J)(c)$, then 
\[a_c(x) = \mathfrak{a}_c(\eta_c^{(P,J)}(x)) \leqslant \Lb(\alpha_c)(\mathfrak{a}'_c(\eta_{c}^{(P,J)}(x))) = \Lb(\alpha_c)(a'_c(x)),\]
and so $\alpha$ defines a 2-cell $\alpha \colon (F,a) \to (F',a')$.  Thus, we obtain the isomorphism 
(\ref{isoof2cats}).    
\end{proof}

%%%%%%%%%%%%%%%%%

%%%%%%%%%%%%%%%%%%%

\begin{rem}\label{rem:fullfaithfulunit}
{\rm
The isomorphism (\ref{isoof2cats}) could also be obtained by the more general observation: whenever $(Q,K)$ is a doctrinal site such that each component $\eta^{(Q,K)}_d \colon Q(d) \to \GC(Q,K)(d)$ of the unit is injective,  for any other doctrinal site $(P,J)$ and a pair of morphisms of doctrinal sites
\[
(F,a), (F,a') \colon (P,J) \rightrightarrows (Q,K),
\]
a natural transformation $\alpha \colon F \to F'$ defines a 2-cell $(F,a) \to (F',a')$ if and only if $\alpha$ defines a 2-cell $(F,\mathfrak{a}) \to (F',\mathfrak{a}')$.  The proof is almost identical to Proposition \ref{prop:strict2adj}.  One direction is shown in Proposition \ref{prop:strict2adj}, while the converse follows by the implication
\begin{align*}
\mathfrak{a}_c\left(\eta^{(P,J)}_c(x)\right) \leqslant \GC(Q,K)(\alpha_c)\left(\mathfrak{a}'_c\left(\eta^{(P,J)}_c(x)\right)\right)
& \implies \eta^{(Q,K)}_c(a_c(x)) \leqslant \eta^{(Q,K)}_c(Q(\alpha_c)(a_c(x))), \\
& \implies a_c(x) \leqslant Q(\alpha_c(x)),
\end{align*}
for all $c \in \cat$ and $x \in P(c)$.

Therefore, the induced functor on hom-categories
\[\Hom_{\bf DocSites}((P,J),(Q,K)) \to \Hom_{\bf GeomDoc}(\GC(P,J),\GC(Q,K))\]
is full and faithful.  The specific isomorphism (\ref{isoof2cats}) is then obtained by noting that $\eta^{(\Lb,K_\Lb)} \colon \Lb \to \GC(\Lb,K_\Lb)$ is an isomorphism for any geometric doctrine $\Lb$.

}
\end{rem}

%%%%%%%%%%%%%%%%%

%%%%%%%%%%%%%%%

It remains to describe the algebras of the geometric completion monad $(\GC,\eta,\mu)$ of the adjunction (\ref{geocompadj}).  Since the geometric completion is an idempotent monad, a simple application of \parencite[Corollary 4.2.4, volume 2]{borvol2} (extended to the 2-categorical setting) yields the following corollary.

\begin{coro}
The algebras for the monad $(\GC,\eta,\mu)$ coincide with geometric doctrines:
\[{\bf DocSites}^\GC \simeq {\bf GeomDoc}.\]
In particular, by restricting the adjunction (\ref{geocompadj}), for each category $\cat$ there is a 2-equivalence
\[({\bf DocSites}/\cat)^\GC \simeq \Loc\left(\sets^{\cat^{op}}\right)^{op},\]
where $({\bf DocSites}/\cat)$ denotes the 1-full 2-subcategory of ${\bf DocSites}$ whose objects are doctrinal sites fibred over the category $\cat$.
\end{coro}

%%%%%%%%%%%%%%%%%%%%%%%%%%

%%%%%%%%%%%%%%%%%%%%%%%%%%%%%%%%%%%%
%%%%%%%%%%%%%%%%%%%%%%%%%%%%%%%%%%%%%%

%%%%%%%%%%%%%%%%%%%%%%%%%%%%%%%%%%%%
%%%%%%%%%%%%%%%%%%%%%%%%%%%%%%%%%%%
%%%%%%%%%%%%%%%%%%%%%%%%%%%%%%%%%%%%%%%

%%%%%%%%%%%%%%%%%%%%%%%%%%%%%%%%%%%%%%%%%%
%%%%%%%%%%%%%%%%%%%%%%%%%%%%%%%%%%%%%%%%%%%%%%

\section{Syntactic sites and the geometric completion}\label{sec:syn}

So far, we have used doctrines to categorically represent a logical theory.  However, it is also common to see theories represented by \emph{syntactic categories} (e.g., in \cite[\S D1.4]{elephant}).  Intuitively, the syntactic category of a first-order theory $\theory$ is the category whose objects are formulae in context and whose arrows are those formulae that express functional relations.  More generally, for any doctrine $P$ expressing regular logic, one can define a `syntactic category' ${\bf Syn}(P)$ of $P$, recalled in Definition \ref{df:syncat}.

It is therefore of interest to study how the geometric completion of a doctrine interacts with taking its syntactic category.  In \S\ref{subsec:geocompforregcat} we will observe that the geometric completion for a doctrinal site yields a geometric completion for \emph{regular sites}, those sites $(\cat,K)$ whose underlying category $\cat$ is regular and where all regular epimorphisms are $K$-covers.

We proceed as follows.  
\begin{itemize}
\item In \S\ref{subsec:existsites}, we recall the theory of \emph{existential sites} from \cite{fibredsites} in the particular context of doctrines in which language we express our development.  We observe that the geometric completion of an existential doctrinal site can be computed `point-wise'.
\item An \emph{existential doctrinal site} whose underlying doctrine is also a primary doctrine has enough expressive power to construct a `syntactic category'.  This construction is recalled in \S\ref{subsec:syncat}, as well as the quasi 2-adjunction ${\bf Syn} \dashv \Sub_{(-)}$ between the syntactic category construction and taking the doctrine of subobjects.
\item In \S\ref{subsec:synsites}, the quasi 2-adjunction ${\bf Syn} \dashv \Sub_{(-)}$ is extended to give an quasi 2-adjunction between existential doctrinal sites and regular sites.  Also in this subsection, we compare the two toposes of sheaves we can now associate with an existential doctrinal site -- $\Sh(\cat \rtimes P,J_\rtimes)$ and $\Sh({\bf Syn}(P),J_{\bf Syn})$.  For an existential doctrine $P$, we exhibit a functor $\zeta^P \colon \cat \rtimes P \to {\bf Syn}(P)$ that yields a dense morphism of sites $\zeta^P \colon (\cat \rtimes P,J_\rtimes) \to ({\bf Syn}(P),J_{\bf Syn})$, and hence an equivalence $\Sh(\cat \rtimes P,J_\rtimes)\simeq\Sh({\bf Syn}(P),J_{\bf Syn})$.
\item Finally, we describe the geometric completion of a regular site in \S\ref{subsec:geocompforregcat}.
\end{itemize}

%%%%%%%%%%%%%%%%%%%%%%%%%%%%%%%%%%%
%%%%%%%%%%%%%%%%%%%%%%%%%%%%%%%%%%%
%%%%%%%%%%%%%%%%%%%%%%%%%%%%%%%%%%%%%

\subsection{Existential doctrinal sites}\label{subsec:existsites}

We saw that, given a geometric doctrine $\Lb \colon \cat^{op} \to \Frm_{\rm open}$, the category $\cat \rtimes \Lb$ can be given a Grothendieck topology $K_\Lb$ whose covering sieves are those sieves $\{\, (c_i,U_i) \xrightarrow{f_i} (d,V) \mid  i \in I\,\}$ such that
\[V = \bigvee_{i \in I} \exists_{f_i} U_i.\]
We now consider variations of this topology for other doctrines $P \colon \cat^{op} \to \PreOrd$ using the language of \emph{existential fibred sites} introduced in \cite{fibredsites}.  We will observe in Proposition \ref{prop:geocompforexsites} that the geometric completion of these doctrinal sites can be computed `point-wise'.

First, we recall some definitions from \parencite[\S 5]{fibredsites}.  Note that the exposition in \cite{fibredsites} exists in the more general framework of indexed categories $F \colon \cat^{op} \to {\bf Cat}$, whereas we have elected to study only indexed preorders (or doctrines in our language).

\begin{df}[Definition 5.1 \cite{fibredsites}]\label{df:exsites}
{\rm
Let $P \colon \cat^{op} \to \PreOrd$ be a doctrine such that, for each arrow $d \xrightarrow{f} c \in \cat$, the map $P(f) \colon P(c) \to P(d)$ has a left adjoint $\exists_f$.  Suppose we are also given, for each $c \in \cat$, a Grothendieck topology $J_c$ on the preorder $P(c)$ for which $\exists_f \colon P(c) \to P(d)$ sends $J_c$-covers to $J_d$-covers.
\begin{enumerate}
\item We say that the pair $(P,(J_c)_{c \in \cat})$ satisfies the \emph{relative Frobenius condition} if for each sieve $S$ on the object $(d,y) \in \cat \rtimes P$ for which the sieve
\[\left\{\,\exists_{f} z \leqslant y \,\middle\vert\,  (c,z) \xrightarrow{f} (d,y) \in S\,\right\}\]
is $J_d$-covering then the sieve
\[\left\{\,\exists_{f} z \leqslant x \,\middle\vert\,  (c,z) \xrightarrow{f} (d,y) \in S, \, z \leqslant P(f)(x)\,\right\}\]
is $J_d$-covering too for any $x \in P(d)$ with $x \leqslant y$.
\item\label{df:exsites:relbc} We say that the pair $(P,(J_c)_{c \in \cat})$ satisfies the \emph{relative Beck-Chevalley condition} if for each sieve $S$ on $(d,y) \in \cat \rtimes P$ for which the sieve
\[\left\{\,\exists_{f} z \leqslant y \,\middle\vert\,  (c,z) \xrightarrow{f} (d,y) \in S\,\right\}\]
is $J_d$-covering, given an arrow $e \xrightarrow{h} d \in \cat$, the sieve
\[\left\{\,\exists_g z \leqslant P(h)(y) \,\middle\vert\,  (c,z) \xrightarrow{f} (d,y) \in S, \, h\circ g = f\,\right\}  \]
is $J_d$-covering too.
\end{enumerate}
The pair $(P,(J_c)_{c \in \cat})$ is said to be a \emph{existential doctrinal site} if the relative Frobenius and relative Beck-Chevalley conditions are both satisfied.
}
\end{df}

Let $P \colon \cat^{op} \to \PreOrd$ be a doctrine such that each fibre $P(c)$ has a Grothendieck topology $J_c$ and, for each arrow $d \xrightarrow{f} c \in \cat$, the map $P(f) \colon P(c) \to P(d)$ has a cover-preserving left adjoint $\exists_f$.  In light of Proposition \ref{prop:exandcohmodels}, we would want to define a Grothendieck topology $J_\rtimes$ on the category $\cat \rtimes P$ for which the $J_\rtimes$-covering sieves are precisely those sieves $\{\,(c_i,x_i) \xrightarrow{f_i}(d,y) \mid i \in I\,\}$ for which $\{\, (d,\exists_{f_i} x_i) \xrightarrow{f_i} (d,y) \mid i \in I\,\}$ is $J_d$-covering.  Such an assignment of sieves to objects is reflexive and transitive, but not necessarily stable.  We observe that the stability of $J_\rtimes$ under arrows of the form $(d,x) \xrightarrow{\id_d} (d,y)$ (respectively, $(e,P(h)(y)) \xrightarrow{h} (d,y)$) is precisely given by the relative Frobenius condition (resp., the relative Beck-Chevalley condition), and since any arrow $(e,x) \xrightarrow{h} (d,y) \in \cat \rtimes P$ can be factored as
\[(e,x) \xrightarrow{\id_e} (e,P(h)(y)) \xrightarrow{h} (d,y)\]
we obtain the following proposition.

\begin{prop}[Theorem 5.1 \cite{fibredsites}]
For the doctrine $P \colon \cat^{op} \to \PreOrd$ above, $J_\rtimes$ defines a Grothendieck topology on the category $\cat \rtimes P$ if and only if the pair $(P,(J_c)_{c \in \cat})$ satisfies both the relative Frobenius and relative Beck-Chevalley conditions.
\end{prop}

\begin{df}[Theorem 5.1 \cite{fibredsites}]\label{df:existentialtopology}
{\rm
Let $(P,(J_c)_{c \in \cat})$ be an existential doctrinal site.  We call the Grothendieck topology $J_\rtimes$ on $\cat \rtimes P$, for which the $J_\rtimes$-covering sieves are precisely those sieves $\{\,(c_i,x_i) \xrightarrow{f_i}(d,y) \mid  i \in I\,\}$ for which $\{\, (d,\exists_{f_i} x_i) \xrightarrow{f_i} (d,y) \mid  i \in I\,\}$ is $J_d$-covering, the \emph{existential topology} for the pair $(P,(J_c)_{c \in \cat})$.  
}
\end{df}

The relative Frobenius and relative Beck-Chevalley conditions are related to the usual Frobenius and Beck-Chevalley conditions by the following proposition.

\begin{prop}[Proposition 5.3 \cite{fibredsites}]\label{prop:normalfrobandbc}
For an existential doctrinal site $(P,(J_c)_{c \in \cat})$,
\begin{enumerate}
\item if $P(c)$ has all finite meets for each $c \in \cat$, then the pair $(P,(J_c)_{c \in \cat})$ satisfies the relative Frobenius condition if and only if $P$ satisfies the Frobenius condition - i.e. $\exists_f z \land x = \exists_f ( z \land P(f)(x))$.
\item if $\cat$ has all pullbacks, then the pair $(P,(J_c)_{c \in \cat})$ satisfies the relative Beck-Chevalley condition if and only if $P$ satisfies the Beck-Chevalley condition - i.e. for each pullback square \[\begin{tikzcd}
c \times_e d \ar{d}{k} \ar{r}{g} & d \ar{d}{h} \\
c \ar{r}{f} & e
\end{tikzcd}\]
of $\cat$, the square
\[\begin{tikzcd}
P({c \times_e d}) \ar{r}{\exists_{g}} & P(d) \\
P(c) \ar{r}{\exists_{f}} \ar{u}{P(k)} & P(e) \ar{u}{P(h)}
\end{tikzcd}\]
commutes.
\end{enumerate}
\end{prop}

Hence we note that, for any existential doctrinal site $(P,(J_c)_{c \in \cat})$, if $P$ is a primary doctrine, then $P$ is automatically an existential doctrine, and the existential topology $J_\rtimes$ contains the topology $J_{\rm Ex}$.  Such an existential doctrinal site should be understood as a doctrine which interprets regular logic, if not further richer syntax.

\paragraph{The geometric completion of an existential doctrinal site.}

Since the arrow $(d,x) \xrightarrow{f} (c,\exists_f x)$ is a $J_\rtimes$-cover for any arrow $d \xrightarrow{f} c$ of $\cat$, an element $S \in \GC(P,J)(c)$ is entirely determined by its elements of the form $(\id_c,x) \in S$.  Hence, by Construction \ref{constrgc}, we obtain the following.

\begin{prop}\label{prop:geocompforexsites}
For an existential doctrinal site $(P,(J_c)_{c \in \cat})$, the geometric completion $\GC(P,J_\rtimes)$ is isomorphic to its point-wise ideal completion, that is: 
\begin{enumerate}
\item for each object $c $ of $ \cat$, $\GC(P,J)(c)$ is the frame $J_c \text{-}{\bf Idl}(P(c))$,
\item for each arrow $d \xrightarrow{f} c$ of $\cat$, $\GC(P,J)(f) \colon \GC(P,J)(c) \to \GC(P,J)(d)$ sends a $J_c$-ideal $I$ to the $J_d$-ideal
\[f^\ast(I) = \{\,y \in P(d) \mid \exists_f y \in I \,\}.\]
\end{enumerate}
\end{prop}

\begin{exs}\label{exmorecoherentdoctrines}
{\rm

Let $P \colon \cat^{op} \to \PreOrd$ be a doctrine such that, for each arrow $d \xrightarrow{f} c \in \cat$, the map $P(f) \colon P(c) \to P(d)$ has a left adjoint $\exists_f$.  Of particular interest is the case where the Grothendieck topology $J_c$ assigned to each fiber $P(c)$ of an existential doctrinal site $(P,(J_c)_{c \in \cat})$ is subcanonical, i.e. a family $\{\,x_i \leqslant y\mid i \in I\,\}$ of inequalities in $P(c)$ is $J_c$-covering only if the join $\bigvee_{i \in I} x_i$ exists and $y \leqslant \bigvee_{i \in I} x_i$ (if $P(c)$ is a poset, $y =\bigvee_{i \in I} x_i$).  For each arrow $d \xrightarrow{f} c$ of $\cat$, being a left adjoint, $\exists_f$ preserves all joins that exist in $P(d)$ and so $\exists_f$ is automatically cover-preserving if $J_d$ and $J_c$ are both subcanonical.  Consideration of certain cases will allow us to generalise existential doctrines and coherent doctrines to non-cartesian settings, as mentioned in Remark \ref{subscriptcart}.

\begin{enumerate}
\item\label{noncartex} There exists an existential doctrinal site $(P,(J_{{\rm triv}})_{c \in \cat})$, where each fibre $P(c)$ has been given the trivial topology $J_{\rm triv}$, if and only if:
\begin{enumerate}
	\item {\it (relative Frobenius condition)} for each arrow $d \xrightarrow{f} c$ of $\cat$, $x, y \in P(c)$ with $x \leqslant y$ and $z \in P(d)$ such that
	\[\exists_f z \leqslant y \text{ and } y \leqslant \exists_f z,\]
	there exists some $w \in P(d)$ with $w \leqslant z$ such that 
	\[\exists_f w \leqslant x \text{ and } x \leqslant \exists_f w;\]
	\item {\it (relative Beck-Chevalley condition)} for each pair of arrows
	\[\begin{tikzcd}
		& e \ar{d}{h} \\
		d \ar{r}{f} & c
	\end{tikzcd}\]
	of $\cat$, $y \in P(c)$ and $z \in P(d)$ such that
	\[\exists_f z \leqslant y \text{ and } y \leqslant \exists_f z,\]
	there exists a commutative square
	\[\begin{tikzcd}
		e' \ar{r}{g} \ar{d}{k}& e \ar{d}{h} \\
		d \ar{r}{f} & c
	\end{tikzcd}\]
	in $\cat$ such that
	\[\exists_g P(k)(z) \leqslant x \text{ and } x \leqslant \exists_g P(k)(z).\]
\end{enumerate}
Note that if $P \colon \cat^{op} \to {\bf MSLat}$ is an existential doctrine, $(P,(J_{\rm triv})_{c \in \cat})$ is an existential doctrinal site.  We denote by $J_{\rm Ex}$ the existential topology on $\cat \rtimes P$ induced by the existential doctrinal site $(P,(J_{\rm triv})_{c \in \cat})$, i.e. the Grothendieck topology generated by covering families of the form
\[\begin{tikzcd}
	(d,x) \ar{r}{g} &(c, \exists_g x),
\end{tikzcd}\]
for $x \in P(d)$ and $d \xrightarrow{g} c \in \cat$, in analogy with Proposition \ref{prop:exandcohmodels}\ref{exmodelsi}.  We define ${\bf RelExDoc} $, the 2-category of \emph{relative existential doctrines}, as the full 2-subcategory of $ {\bf DocSites}$ on objects of the form $(P,J_{\rm Ex})$, where $(P,(J_{\rm triv})_{c \in \cat})$ is an existential doctrinal site.  The category ${\bf ExDoc}$ of existential doctrines is thus now a full 2-subcategory of ${\bf RelExDoc}$.

\item\label{noncartcoh} Suppose that we can endow each fibre $P(c)$ with the Grothendieck topology $J_{\rm Coh}$, where a sieve $S$ on $y \in P(c)$ is $J_{\rm Coh}$-covering precisely if $S$ contains a finite family
\[\{\,x_i \leqslant y \mid i \in I\,\} \subseteq S\]
such that $y \leqslant \bigvee_{i \in I} x_i$.  If $P(c)$ is a lattice, $J_{\rm Coh}$ defines a Grothendieck topology on $P(c)$ if and only if $P(c)$ is a distributive lattice.  The pair $(P,(J_{\rm Coh})_{c \in \cat})$ defines an existential doctrinal site if and only if:
\begin{enumerate}
	\item {\it (relative Frobenius condition)} for each pair $x, y \in P(c)$ with $x \leqslant y$, and each finite collection of pairs $d_i \xrightarrow{f_i} c$ and $z \in P(d_i)$, indexed by $i \in I$, such that
	\[\bigvee_{i \in I} \exists_{f_i} z_i \leqslant y \text{ and } y \leqslant \bigvee_{i \in I} \exists_{f_i} z_i,\]
	for each $i \in I$ there exists some $w_i \in P(d_i)$ with $w_i \leqslant z_i$ such that 
	\[\bigvee_{i \in I} \exists_{f_i} w_i \leqslant x \text{ and } x \leqslant \bigvee_{ i \in I}  \exists_{f_i} w_i;\]
	\item {\it (relative Beck-Chevalley condition)} for each arrow $e \xrightarrow{h} c$ of $\cat$, $y \in P(c)$, and each finite collection of pairs $d_i \xrightarrow{f_i} c$ and $z \in P(d_i)$, indexed by $i \in I$, such that
	\[\bigvee_{i \in I} \exists_{f_i} z_i \leqslant y \text{ and } y \leqslant \bigvee_{i \in I} \exists_{f_i} z_i,\]
	for each $i \in I$, there exists a finite collection of pairs of arrows $g_j$ and $k_j$, indexed by $j \in J_i$, such that there is a commutative square
	\[\begin{tikzcd}
		e_j' \ar{r}{g_j} \ar{d}{k_j}& e \ar{d}{h} \\
		d \ar{r}{f} & c
	\end{tikzcd}\]
	of $\cat$ and secondly
	\[\bigvee_{i \in I} \bigvee_{j \in J_i} \exists_{g_j} P(k_j)(z_i) \leqslant x \text{ and } x \leqslant \bigvee_{i \in I} \bigvee_{j \in J_i} \exists_{g_j} P(k_j)(z_i).\]
\end{enumerate}
Just as with relative existential doctrines, by analogy with Proposition \ref{prop:exandcohmodels}\ref{cohmodelsi}, we denote the existential topology on $\cat \rtimes P$ induced by the existential doctrinal site $(P,(J_{\rm Coh})_{c \in \cat})$ by $J_{\rm Coh}$, i.e. the Grothendieck topology generated by covering families of the form
\[\begin{tikzcd}
	(d,x) \ar{r}{g} & (c, \exists_g x \lor \exists_h y ) & \ar{l}[']{h} (e,y),
\end{tikzcd}\]
for $x \in P(d)$, $y \in P(d)$, and arrows $d \xrightarrow{g} c , e \xrightarrow{h} c \in \cat$.  We call the resultant doctrinal site $(P,J_{\rm Coh})$ a \emph{relative coherent doctrine} and denote the full 2-subcategory of ${\bf DocSites}$ on relative coherent doctrines by ${\bf RelCohDoc}$.  We once again have that ${\bf CohDoc}$ is a full 2-subcategory of ${\bf RelCohDoc}$.

When $P \colon \cat^{op} \to {\bf DLat}$ is a coherent doctrine, and $\cat \rtimes P$ is equipped with the topology $J_{\rm Coh}$, we recognise by Proposition \ref{prop:geocompforexsites} that the fibre of the geometric completion $\GC(P,J_{\rm Coh})(c)$ is the coherent locale associated with the distributive lattice $P(c)$ under the (point-free) Stone duality for distributive lattices (see \cite[\S II.3.3]{stone}).

In particular, if $\mathbb{B} \colon \cat^{op} \to {\bf Bool}$ is a Boolean doctrine, then $\GC(\mathbb{B},J_{\rm Coh}) \colon \cat^{op} \to {\bf StFrm}_{\rm open}$ sends $c \in \cat$ to the Stone frame corresponding to the Boolean algebra $\mathbb{B}(c)$.  If there is an isomorphism $\mathbb{B} \cong F^\theory$ for some single-sorted classical theory $\theory$ over a signature $\Sigma$, then $\GC(\mathbb{B},J_{\rm Coh})(\vec{x})$ (where $\vec{x} \in {\bf Con}_\Sigma$ is a context/tuple of variables of length $n$) coincides with the frame of opens of the familiar $n$th Stone space of the theory $\theory$ (see \parencite[\S 6.3]{hodges}).  Doctrines of this form (or rather, since Stone frames are spatial, the doctrines $\Pt \circ \GC(\mathbb{B},K^{\rm fin}_{\mathbb{B}}) \colon \Con \to {\bf StSpace}$) were dubbed \emph{polyadic spaces} in the note \cite{polyadicnote} and suggested for use in categorically proving standard theorems of classical logic, a desire realised in \cite{VGM}.

\end{enumerate}

}
\end{exs}

%%%%%%%%%%%%%%%%%%%%%
%%%%%%%%%%%%%%%%%%%%

%%%%%%%%%%%%%%%%%%%%%%

\subsection{Syntactic categories}\label{subsec:syncat}

We can always obtain a primary doctrine from a category $\cat$ with all finite limits by taking its \emph{doctrine of subobjects} $\Sub_\cat \colon \cat^{op} \to {\bf MSLat}$, the doctrine where:
\begin{enumerate}
\item for each object $c $ of $ \cat$, $\Sub_\cat(c)$ is the meet-semilattice of subobjects of $c$,
\item for each arrow $d \xrightarrow{f} c$ of $\cat$, $\Sub_\cat(f) \colon \Sub_\cat(c) \to \Sub_\cat(d)$ is the map which sends a subobject $e \rightarrowtail c$ to the pullback
\[
\begin{tikzcd}
f^\ast(e) \ar[tail]{d} \ar{r} & e \ar[tail]{d} \\
d \ar{r}{f} & c.
\end{tikzcd}
\]
\end{enumerate}
Taking the doctrine of subobjects of a cartesian category naturally defines a 2-functor 
\[\Sub_{(-)} \colon {\bf Cart} \to {\bf PrimDoc}.\]
\begin{enumerate}
\item Each cartesian functor $F \colon \cat \to \dcat$ restricts to a meet-semilattice homomorphism 
\[{a^F_c \colon \Sub_\cat(c) \to \Sub_\dcat(F(c))}\]
natural in $c \in \cat$.  Hence, the pair $(F,a^F)$ defines a morphism of primary doctrines
\[(F,a^F) \colon \Sub_\cat \to \Sub_\dcat.\]

\item Each natural transformation $\alpha \colon F \to F'$ defines a 2-cell $\alpha \colon (F,a^F) \to (F',a^{F'})$ of ${\bf PrimDoc}$.  For each $c \in \cat$ and $x \in \Sub_\cat(c)$, the required inequality $a^F(x) \leqslant \Sub_\cat(f)(a^{F'}(x))$ follows by the universal property of the pullback:
\[\begin{tikzcd}
F(x) \ar[tail]{rd} \ar[dashed]{r} \ar[bend left]{rr}{\alpha_x} & \Sub_\cat(\alpha_c)(F'(x)) \ar{r} \ar[tail]{d} & F'(x) \ar[tail]{d} \\
& F(c) \ar{r}{\alpha_c} & F'(c).
\end{tikzcd}\]
\end{enumerate}
Moreover, the 2-functor $\Sub_{(-)}$ can easily be checked to be full and faithful.

From a doctrine $P \colon \cat^{op} \to \PreOrd$, it is possible to construct a \emph{syntactic category} ${\bf Syn}(P)$, and this construction is converse to taking the doctrine of subobjects in the sense that ${\bf Syn}(\Sub_\cat) \simeq \cat$ for a certain subclass of cartesian categories (regular categories, see below).  It is only possible to construct the syntactic category of a doctrine $P$ when $P$ is rich enough to interpret \emph{provably functional relations}.  These are predicates that, according to the internal logic of a doctrine (see \cite[\S 4.3]{jacobs}), encode the graph of a function between two other predicates.  For that, we need at least regular logic.  In this subsection we review material found in \cite{triposphd}, \cite{pittsopen} and \cite{triposretro} regarding the syntactic category construction.  Our exposition is similar to the explanation found in \cite[\S 3]{coumans}.

\paragraph{Building a syntactic category.} 

It seems superfluous to recall that, given two subsets $A \subseteq X$ and $B \subseteq Y$, the graph of a function $f \colon A \to B$ consists of a subset $f \subseteq X \times Y$ such that
\[
(x,y) \in f \implies x \in A, \, y \in B, \ \ (x,y), \, (x,y') \in f \implies y = y', \ \ x \in A \implies \exists \, y \in B \ (x,y) \in f.
\]
The conceit behind provably functional relations is to translate these implications into the internal language of a doctrine, as is done below.  Recall from \cite{lawcompr} that the internal equality predicate of an existential doctrine $P \colon \cat^{op} \to {\bf MSLat}$, for an object $c \in \cat$, is given by $\exists_{\Delta_c} \top_c$, where $\Delta_c \colon c \to c \times c$ is the diagonal.

\begin{df}\label{df:syncat}{\rm
Let $P \colon \cat^{op} \to {\bf MSLat}$ be an existential doctrine.  The \emph{syntactic category} ${\bf Syn}(P)$ of $P$ is the category:
\begin{enumerate}
\item whose objects are pairs $(c,U)$ where $c$ is an object of $\cat$ and $U \in P(c)$,
\item\label{dfii:syncat} and each arrow $(c,U) \to (d,V)$ is given by some $W \in P(c \times d)$ that defines a \emph{provably functional relation}, i.e. the inequalities
\[W \leqslant P(\pr_1)(U) \land P(\pr_2)(V),\ \ P(\pr_{1,2})(W) \land P(\pr_{1,3})(W) \leqslant P(\pr_{2,3}) \exists_{\Delta_d} \top_d, \ \ U \leqslant \exists_{\pr_1} W,\]
are satisfied, where $\pr_1$ and $\pr_2$ are the projections
\[\begin{tikzcd}
	c & \ar{l}[']{\pr_1} c \times d \ar{r}{\pr_2} & d,
\end{tikzcd}\]
$\pr_{1,2}$, $\pr_{1,3}$ and $\pr_{2,3}$ are the projections
\[\begin{tikzcd}
	& c \times d \times d \ar[bend right]{ld}[']{\pr_{1,3}} \ar{d}{\pr_{1,2}} \ar[bend left]{rd}{\pr_{2,3}} & \\
	c \times d & c \times d & d \times d,
\end{tikzcd}\]
and $\Delta_d \colon d \to d \times d$ is the diagonal.
\end{enumerate}
The identity morphism on $(c,U)$ is given by $\exists_{\Delta_c} U \in P(c \times c)$, while the composite of two arrows
\[\begin{tikzcd}
(c,U) \ar{r}{W} & (d,V) \ar{r}{W'} & (e,V')
\end{tikzcd}\]
is given by $\exists_{\pr_{1,3}} (P(\pr_{1,2})( W) \land P(\pr_{2,3})( W')$.
}\end{df}

\begin{rem}
{\rm 
We have presented the syntactic category construction for an existential doctrine.  This is analogous to the \emph{category of maps} construction for an allegory (see \cite[\S 2.132]{allegories}).  Indeed, for each existential doctrine $P$, the two constructions coincide in that ${\bf Syn}(P) \simeq \mathcal{MAP}({\bf A}(P))$, where ${\bf A}(P)$ is the allegory of relations on $P$, i.e. the allegory whose objects are elements $x \in P(c \times d)$ (note that $P$ must be a regular doctrine in order to express the relational composite of $x \in P(c \times d)$ and $y\in P(d \times e)$).
}
\end{rem}

\begin{ex}{\rm
Let $\theory$ be a theory over a signature $\Sigma$ in a fragment of first order logic that contains \emph{regular logic} (see \parencite[Definition D1.1.3]{elephant}).  The syntactic category ${\bf Syn}(F^\theory)$ of the regular doctrine $F^\theory \colon \Sort \to {\bf MSLat}$ from Example \ref{prototypical}, by definition, coincides with the usual syntactic category $\cat_\theory$ for $\theory$ as described in \parencite[\S D1.4]{elephant}.
}\end{ex}

%%%%%%%%%%%%%%%%%
\paragraph{The syntax-subobject adjunction.}  Taking the syntactic category of an existential doctrine yields a left inverse to the restriction of the 2-functor $\Sub_{(-)} \colon {\bf Cart} \to {\bf PrimDoc}$ to a suitable 2-subcategory: the 2-category of regular categories.

\begin{df}[\S A1.3 \cite{elephant}]
{\rm 
By ${\bf Reg}$, we denote the 2-category:
\begin{enumerate}
\item whose objects are \emph{regular categories} -- categories with finite limits and image factorisations that are stable under pullback (we will also require that a regular category is \emph{well-powered}),
\item whose 1-cells are \emph{regular functors} (also called \emph{exact functors} in \cite{barrexact} and \cite{carbonivitale}) -- finite limit preserving functors that also preserve regular epimorphisms,
\item and whose 2-cells are natural transformations between regular functors.
\end{enumerate}
}
\end{df}

For each regular category $\cat$, the subobject doctrine $\Sub_\cat \colon \cat^{op} \to \PreOrd$ is an existential doctrine (see \cite[Theorem 4.4.4]{jacobs}).  The left adjoint to $\Sub_\cat(f)$, for an arrow $f$ of $\cat$, is given by the existence of images (see \cite[Lemma A1.3.1]{elephant}).  Thus, the subobject 2-functor $\Sub_{(-)} \colon {\bf Cart} \to {\bf PrimDoc}$ from above restricts to a 2-functor
\[\Sub_{(-)} \colon {\bf Reg} \to {\bf ExDoc}.\]

The syntactic category construction also induces a 2-functor ${\bf Syn} \colon {\bf ExDoc} \to {\bf Reg}$.  For an existential doctrine $P$, the category ${\bf Syn}(P)$ is regular: it has finite limits and the image factorisation of an arrow $(c,U) \xrightarrow{W} (d,V) \in {\bf Syn}(P)$ is given by
\[\begin{tikzcd}
(c,U) \ar[two heads]{r}{W} & (d,\exists_{\pr_2} W) \ar[tail]{r} & (d,V)
\end{tikzcd}\]
(see \cite[\S2.4 \& 2.5]{triposphd}).  Since a morphism of existential doctrines $(F,a) \colon P \to Q$ preserves the interpretation of regular logic, $(F,a)$ induces a regular functor ${\bf Syn}(F,a) \colon {\bf Syn}(P) \to {\bf Syn}(Q)$.

The two functors form a quasi 2-adjunction
\begin{equation}\label{synsubadj}
\begin{tikzcd}
{{\bf ExDoc}} && {{\bf Reg}}
\arrow[""{name=0, anchor=center, inner sep=0}, "{{\rm Sub}_{(-)}}", shift left=3, from=1-3, to=1-1]
\arrow[""{name=1, anchor=center, inner sep=0}, "{{\bf Syn}}", shift left=3, from=1-1, to=1-3]
\arrow["\dashv"{anchor=center, rotate=-90}, draw=none, from=1, to=0]
\end{tikzcd}
\end{equation}
such that, since $\Sub_{(-)}$ is full and faithful, the counit is a natural equivalence of categories ${\bf Syn}(\Sub_\cat) \simeq \cat$ for each $\cat \in {\bf Reg}$.  The quasi 2-adjunction (\ref{synsubadj}) can be deduced from the analogous quasi 2-adjunction found in \cite[Proposition 1.3]{pittsopen} and \cite[Theorem 3.6]{triposretro}.

%%%%%%%%%%%%%%%%%
%%%%%%%%%%%%%%%%%%

\subsection{Syntactic sites}\label{subsec:synsites}

%%%%%%%%%%%%%%%%%%%%%%
%%%%%%%%%%%%%%%%%%%%%
To apply the geometric completion as currently formulated to regular categories, we will need a version of the quasi 2-adjunction (\ref{synsubadj}) that also incorporates Grothendieck topologies.  We introduce this extension in the first half of this subsection.

We will introduce the 2-category of \emph{regular sites} and a particular 2-category of existential doctrinal sites which we denote by ${\bf ExDocSites}$.  We then extend the 2-functors
\[{\bf Syn} \colon {\bf ExDoc} \to {\bf Reg}, \ \ \Sub_{(-)} \colon {\bf Reg} \to {\bf ExDoc}\]
by showing that for each existential doctrinal site $(P,(J_c)_{c \in \cat})$ (respectively, regular site $(\cat,K)$) there exists a natural choice of Grothendieck topology $J_{\bf Syn}$ on ${\bf Syn}(P)$ making $({\bf Syn}(P),J_{\bf Syn})$ a regular site (resp., a natural choice of Grothendieck topology $K|_{\Sub_\cat(c)}$ on $\Sub_\cat(c)$, for each $c \in \cat$, making $(\Sub_\cat,(K|_{\Sub_\cat(c)})_{c \in \cat})$ an existential site).  Finally, we demonstrate that these extended 2-functors are quasi 2-adjoint.

For each existential doctrinal site $(P,(J_c)_{c \in \cat}) \in {\bf ExDocSites}$, there are now two choices, $\Sh(\cat \rtimes P,J_{\rtimes})$ and $\Sh({\bf Syn}(P),J_{\bf Syn})$, for the toposes we can associate with the doctrinal site.  In particular, if $\theory$ is a theory in a fragment of logic that contains regular logic, then there are two choices for the classifying topos of $\theory$.  The second half of this subsection is devoted to showing that this is a monochotomous choice
%illusory choice$
by demonstrating an equivalence of toposes
\[\Sh(\cat \rtimes P,J_{\rtimes})\simeq \Sh({\bf Syn}(P),J_{\bf Syn}),\]
natural in $(P,(J_c)_{c \in \cat}) \in {\bf ExDocSites}$.

\paragraph{Regular sites.}  

\begin{df}\label{df:regsites}
{\rm
\begin{enumerate}
\item Let ${\bf RegSites}$ be the 2-category whose objects are \emph{regular sites}, by which we mean pairs $(\cat,K)$ consisting of a regular category $\cat$ and a Grothendieck topology $K$ on $\cat$ such that the sieve generated by each regular epimorphism $c \twoheadrightarrow d$ is a $K$-cover.  The 1-cells of ${\bf RegSites}$ are cover preserving regular functors and the 2-cells are all natural transformations between these.

\item By ${\bf ExDocSites}$ we denote the the full 2-subcategory of ${\bf DocSites}$ on objects of the form $(P,J_{\rtimes})$ for an existential doctrinal site $(P,(J_c)_{c \in \cat})$ whose underlying doctrine $P$ is also a primary doctrine (and therefore, by Proposition \ref{prop:normalfrobandbc}, an existential doctrine).
\end{enumerate}
}
\end{df}

\begin{rem}\label{rem:subobjtop}{\rm
Let $(\cat,K) \in {\bf RegSites}$ be a regular site.  As every morphism $d \xrightarrow{f} c \in \cat$ can be factored as a regular epimorphism followed by a monomorphism
\[\begin{tikzcd}
d \ar[two heads]{r} & f(d) \ar[tail]{r} & c,
\end{tikzcd}\]
the Grothendieck topology $K$ on $\cat$ is entirely determined by which families of subobjects
\[\{\,e_i \rightarrowtail c \mid i \in I\,\}\]
are $K$-covering for each $c \in \cat$.
}\end{rem}

For a regular site $(\cat,K)$, by endowing each subobject lattice $\Sub_\cat(c)$ with the Grothendieck topology $K|_{\Sub_\cat(c)}$, we obtain an existential site $(\cat,(K|_{\Sub_\cat(c)})_{c \in \cat})$.  We note that the left adjoint
\[\exists_f \colon \Sub_\cat(c) \to \Sub_\cat(d),\]
for each arrow $c \xrightarrow{f} d \in \cat$, preserves covers because if a family of arrows $\{\,a_i \rightarrowtail b \mid i \in I\,\}$ in $\Sub_\cat(c)$ is $K$-covering then, using the diagram
\[\begin{tikzcd}
a_i \ar[tail]{r} \ar[tail]{rd} \ar[bend left, two heads]{rrr} & b\ar[tail]{d} \ar[bend left, two heads]{rrr}&& f(a_i) \ar[tail]{rd} \ar[tail]{r} & f(b) \ar[tail]{d} \\
& c \ar{rrr}{f} &&& d,
\end{tikzcd}\]
the fact that $a_i \twoheadrightarrow f(a_i)$ and $b \twoheadrightarrow f(b)$ are both $K$-covers, and transitivity of $K$-covers, we observe that $\{\,f(a_i) \rightarrowtail f(b) \mid i \in I\,\}$ is a $K$-covering family too.  We denote the resulting topology on $\cat \rtimes \Sub_\cat$ by $K_{\Sub}$.  It is easily checked that, given regular sites $(\cat,K)$ and $(\dcat,K')$ and a regular functor $F \colon \cat \to \dcat$, $F$ sends $K$-covers to $K'$-covers if and only if the induced morphism of doctrines
\[(F,a^F) \colon \Sub_\cat \to \Sub_\dcat\]
sends $K_\Sub$-covers to $K'_\Sub$-covers.  Thus, we obtain a 2-functor $\Sub_{(-)} \colon {\bf RegSites} \to {\bf ExDocSites}$ which, moreover, is full and faithful since $\Sub_{(-)} \colon {\bf Reg} \to {\bf ExDoc}$ is full and faithful.

Conversely, for each existential site $(P,(J_c)_{c \in \cat}) \in {\bf ExDocSites}$, we can endow the syntactic site ${\bf Syn}(P)$ with the Grothendieck topology $J_{\bf Syn}$ where a family of arrows 
\[\left\{\,(c_i , U_i) \xrightarrow{W_i} (d,V) \,\middle\vert\,  i \in I\,\right\}\]
is $J_{\bf Syn}$-covering if and only if the family
\[\{\,\exists_{\pr_2} W_i \leqslant V \mid i \in I\,\} \]
is $J_d$-covering.  Recall that the image factorisation of an arrow $(c,U) \xrightarrow{W} (d,V) \in {\bf Syn}(P)$ is given by
\[\begin{tikzcd}
(c,U) \ar[two heads]{r}{W} & (d,\exists_{\pr_2} W) \ar[tail]{r} & (d,V),
\end{tikzcd}\]
whose left factor $(c,U) \xrightarrow{W} (d,\exists_{\pr_2}W)$ is trivially a $J_{\bf Syn}$-cover.  Thus, regular epimorphisms are $J_{\rm Syn}$-covers and so $({\bf Syn}(P),J_{\bf Syn})$ is a regular site.  It is equally trivially observed that if $(F,a) \colon (P,J_\rtimes) \to (Q,J'_\rtimes)$ is a morphism of ${\bf ExDocSites}$ then the induced functor
\[{\bf Syn}(F,a) \colon {\bf Syn}(P) \to {\bf Syn}(Q)\]
sends $J_{\bf Syn}$-covers to $J'_{\bf Syn}$-covers, and hence there is a functor ${\bf Syn} \colon {\bf ExDocSites} \to {\bf RegSites}$.

\begin{prop}\label{prop:subsynadjforsites}
The quasi 2-adjunction (\ref{synsubadj}) extends to give a second quasi 2-adjunction and a morphism of adjunctions
\begin{equation*}
\begin{tikzcd}
{{\bf ExDoc}} && {{\bf ExDocSites}} \\
\\
{{\bf Reg}} && {{\bf RegSites},}
\arrow[""{name=0, anchor=center, inner sep=0}, "{{\bf Syn}}"', shift right=3, from=1-3, to=3-3]
\arrow[""{name=1, anchor=center, inner sep=0}, "{{\rm Sub}_{(-)}}"', shift right=3, from=3-3, to=1-3]
\arrow["{U'}"', from=3-3, to=3-1]
\arrow["U"', from=1-3, to=1-1]
\arrow[""{name=2, anchor=center, inner sep=0}, "{{\rm Sub}_{(-)}}"', shift right=3, from=3-1, to=1-1]
\arrow[""{name=3, anchor=center, inner sep=0}, "{{\bf Syn}}"', shift right=3, from=1-1, to=3-1]
\arrow["\dashv"{anchor=center}, draw=none, from=0, to=1]
\arrow["\dashv"{anchor=center}, draw=none, from=3, to=2]
\end{tikzcd}
\end{equation*}
i.e. ${\bf Syn} \circ U = U' \circ {\bf Syn}$ and $\Sub_{(-)} \circ U' = U \circ \Sub_{(-)}$, where $U$ and $U'$ are the forgetful 2-functors.
\end{prop}

Proposition \ref{prop:subsynadjforsites} is obtained by restricting the equivalence
\begin{equation}\label{diag:2adjforregsites}
\begin{tikzcd}[column sep = tiny]
{\textup{ }\textup{ }\textup{ }\textup{ }\textup{ }\textup{ }\textup{ }\textup{ }\textup{ }\textup{ }\textup{ }\textup{ }\textup{ }\textup{ }\textup{ }{\rm Hom}_{\bf Reg}({\bf Syn}(P),\mathcal{C})} & {{\rm Hom}_{\bf ExDoc}(P,{\rm Sub}_\mathcal{C})\textup{ }\textup{ }\textup{ }\textup{ }\textup{ }\textup{ }\textup{ }\textup{ }\textup{ }\textup{ }\textup{ }\textup{ }\textup{ }\textup{ }\textup{ }\textup{ }\textup{ }\textup{ }} \\
{{\rm Hom}_{\bf RegSites}(({\bf Syn}(P),J_{\bf Syn}),(\mathcal{C},K))} & {{\rm Hom}_{{\bf ExDocSites}}((P,J_\rtimes),({\rm Sub}_\mathcal{C},K_{\rm Sub})).}
\arrow["\simeq"{description}, draw=none, from=2-1, to=2-2]
\arrow[hook, from=2-1, to=1-1]
\arrow[hook, from=2-2, to=1-2]
\arrow["\simeq"{description}, draw=none, from=1-1, to=1-2]
\end{tikzcd}
\end{equation}

%%%%%%%%%%%%%%%%%%%%%%%%%%%%

%%%%%%%%%%%%%%%%%%%%%%%%%%%%

%%%%%%%%%%%%%%%%%%%%%%%%%%%%%%%%%%
\paragraph{Comparing the toposes of sheaves.}

%\begin{df}
%{\rm
%An \emph{elementary doctrine} $P \colon \cat^{op} \to {\bf MSLat}$ is an existential doctrine in which, for all $d \in \cat$
%\[ \rho'_{1,2} \exists_{\Delta_d} \top_d \land \rho'_{1,3} %\exists_{\Delta_d} \top_d \leqslant \rho'_{2,3}  \exists_{\Delta_d} \top_d , \]
%   where $\rho'_{1,2}$, $\rho'_{1,3}$ and $\rho'_{2,3}$ are %the projections $d \times d \times d \to d \times d$.}
%\end{df}

We now wish to compare the two toposes $\Sh(\cat \rtimes P,J_\rtimes)$ and $\Sh({\bf Syn}(P),J_{\bf Syn})$ for an existential doctrinal site $(P,(J_c)_{c \in \cat}) \in {\bf ExDocSites}$.  Already, from the natural equivalence (\ref{diag:2adjforregsites}), for a Grothendieck topos $\topos$ we obtain a natural equivalence
\begin{align*}
\Geom(\topos,\Sh({\bf Syn}(P),J_{\bf Syn})) & \simeq J_{\bf Syn}\text{-}{\bf Flat}({\bf Syn}(P),\topos) \\
& \simeq {\rm Hom}_{\bf RegSites}(({\bf Syn}(P),J_{\bf Syn}),(\topos,J_{\rm can})), \\
& \simeq  {\rm Hom}_{{\bf ExDocSites}}((P,J_\rtimes),({\rm Sub}_\topos,K_{\rm Sub_\topos})), \\
& \simeq J_\rtimes \text{-}{\bf Flat}(\cat \rtimes P,\topos), \\
& \simeq \Geom(\topos,\Sh(\cat \rtimes P,J_\rtimes))
\end{align*}
and hence an equivalence of toposes $\Sh(\cat \rtimes P,J_\rtimes)\simeq\Sh({\bf Syn}(P),J_{\bf Syn})$.  However, it is instructive to see where this equivalence comes from.

For each $(P,(J_c)_{c \in \cat}) \in {\bf ExDocSites}$, we will construct a functor $\zeta^P \colon \cat \rtimes P \to {\bf Syn}(P)$ and then demonstrate that 
\[\zeta^P \colon (\cat \rtimes P,J_{\rtimes}) \to ({\bf Syn}(P) ,J_{\bf Syn})\]
is a dense morphism of sites.  Thus, we obtain an equivalence of toposes 
\[\Sh(\cat \rtimes P,J_\rtimes)\simeq\Sh({\bf Syn}(P),J_{\bf Syn}).\]

Since $\cat \rtimes P$ and ${\bf Syn}(P)$ share the same objects, it is obvious how we would wish $\zeta^P$ to act on objects.  Our first task therefore is to conjure a provably functional relation from an arrow $(c,U) \xrightarrow{f} (d,V)$ of $\cat \rtimes P$.  This is simple in the case of the existential doctrine $F^\theory \colon \Con \to {\bf MSLat}$ associated to a regular theory $\theory$.  For each arrow $\form{\varphi}{x} \xrightarrow{\sigma} \form{\psi}{y}$ of $\Con \rtimes F^\theory$, the formula $\varphi \land  \bigwedge_{y_i \in \vec{y}} y_i = \sigma(y_i) $, in context $\vec{x},\vec{y}$, is easily shown to be a provably functional formula
\[\begin{tikzcd}
\form{\varphi}{x} \xrightarrow{\varphi \land  \bigwedge_{y_i \in \vec{y}} y_i = \sigma(y_i)} \form{\psi}{y}.
\end{tikzcd}\]
We demonstrate below that the same intuition holds for arbitrary existential doctrines.

\begin{lem}
Let $P \colon \cat^{op} \to {\bf MSLat}$ be an existential doctrine.  For each arrow $(c,U) \xrightarrow{f} (d,V)$ of $\cat \rtimes P$, i.e. whenever $U \leqslant P(f)(V)$, the proposition
\[\exists_{\id_c \times f}U\in P(c \times d)\]
is a \emph{provably functional relation} $(c,U) \to(d,V) \in {\bf Syn}(P)$.
\end{lem}
\begin{proof}
We must check the three conditions from Definition \ref{df:syncat}\ref{dfii:syncat} are satisfied.  Using 
\[U \leqslant U = P(\pr_1 \circ \id_c \times f)(U)\text{ and }U \leqslant P(f)(V) = P(\pr_2 \circ \id_c \times f)(V),\]
we obtain the first inequality 
\begin{align*}
U \leqslant  P(\pr_1 \circ \id_c \times f)(U), \  P(\pr_2 \circ \id_c \times f)(V) & \implies U \leqslant P(\id_c \times f)P(\pr_1)(U) , \ P(\id_c \times f)P(\pr_2)(V), \\
& \implies \exists_{\id_c \times f}(U) \leqslant P(\pr_1)(U) \land P(\pr_2)(V).
\end{align*}

%%%%%%%%%%%

The second desired inequality,
\[P(\pr_{1,2})\exists_{\id_c \times f} U \land P(\pr_{1,3})\exists_{\id_c \times f} U \leqslant  P(\pr_{2,3})\exists_{\Delta_d}\top_d,\]
is effectively transitivity of the internal equality predicate.  We first note that all the squares in the diagram
% https://q.uiver.app/?q=WzAsMTAsWzAsMCwiZCJdLFswLDIsImRcXHRpbWVzIGQiXSxbMiwwLCJjIl0sWzIsMiwiYyBcXHRpbWVzIGQiXSxbNCwwLCJjIFxcdGltZXMgZCJdLFs0LDIsImMgXFx0aW1lcyBkIFxcdGltZXMgZCJdLFs2LDAsImMiXSxbNiwyLCJjIFxcdGltZXMgZCJdLFsyLDQsImMiXSxbNCw0LCJjIFxcdGltZXMgZCJdLFswLDEsIlxcRGVsdGFfZCIsMl0sWzIsMCwiZiIsMl0sWzIsMywiXFxpZF9jIFxcdGltZXMgZiJdLFszLDEsIihmLFxcaWRfZCkiLDJdLFsyLDQsIlxcaWRfYyBcXHRpbWVzIGYiXSxbNCw1LCIoXFxpZF9jLGYsXFxpZF9kKSJdLFs1LDcsIlxccHJfezEsMn0iXSxbNiw3LCJcXGlkX2MgXFx0aW1lcyBmIl0sWzQsNiwiXFxwcl8xIl0sWzgsOSwiXFxpZF9jIFxcdGltZXMgZiJdLFs1LDksIlxccHJfezEsM30iXSxbMyw4LCJcXHByXzEiLDJdLFszLDUsIihcXGlkX2MsXFxpZF9kLGYpIl0sWzExLDEzLCJcXHRleHR7XFx0ZXh0Y2lyY2xlZHtcXHJhaXNlYm94ey0uNXB0fXsxfX19IiwxLHsic2hvcnRlbiI6eyJzb3VyY2UiOjIwLCJ0YXJnZXQiOjIwfSwic3R5bGUiOnsiYm9keSI6eyJuYW1lIjoibm9uZSJ9LCJoZWFkIjp7Im5hbWUiOiJub25lIn19fV0sWzE0LDIyLCJcXHRleHR7XFx0ZXh0Y2lyY2xlZHtcXHJhaXNlYm94ey0uNXB0fXsyfX19IiwxLHsic2hvcnRlbiI6eyJzb3VyY2UiOjIwLCJ0YXJnZXQiOjIwfSwic3R5bGUiOnsiYm9keSI6eyJuYW1lIjoibm9uZSJ9LCJoZWFkIjp7Im5hbWUiOiJub25lIn19fV0sWzE4LDE2LCJcXHRleHR7XFx0ZXh0Y2lyY2xlZHtcXHJhaXNlYm94ey0uNXB0fXszfX19IiwxLHsic2hvcnRlbiI6eyJzb3VyY2UiOjIwLCJ0YXJnZXQiOjIwfSwic3R5bGUiOnsiYm9keSI6eyJuYW1lIjoibm9uZSJ9LCJoZWFkIjp7Im5hbWUiOiJub25lIn19fV0sWzIyLDE5LCJcXHRleHR7XFx0ZXh0Y2lyY2xlZHtcXHJhaXNlYm94ey0uNXB0fXs0fX19IiwxLHsic2hvcnRlbiI6eyJzb3VyY2UiOjIwLCJ0YXJnZXQiOjIwfSwic3R5bGUiOnsiYm9keSI6eyJuYW1lIjoibm9uZSJ9LCJoZWFkIjp7Im5hbWUiOiJub25lIn19fV1d
\[\begin{tikzcd}
d && c && {c \times d} && c \\
\\
{d\times d} && {c \times d} && {c \times d \times d} && {c \times d} \\
\\
&& c && {c \times d}
\arrow["{\Delta_d}"', from=1-1, to=3-1]
\arrow[""{name=0, anchor=center, inner sep=0}, "f"', from=1-3, to=1-1]
\arrow["{\id_c \times f}", from=1-3, to=3-3]
\arrow[""{name=1, anchor=center, inner sep=0}, "{(f,\id_d)}"', from=3-3, to=3-1]
\arrow[""{name=2, anchor=center, inner sep=0}, "{\id_c \times f}", from=1-3, to=1-5]
\arrow["{(\id_c,f,\id_d)}", from=1-5, to=3-5]
\arrow[""{name=3, anchor=center, inner sep=0}, "{\pr_{1,2}}", from=3-5, to=3-7]
\arrow["{\id_c \times f}", from=1-7, to=3-7]
\arrow[""{name=4, anchor=center, inner sep=0}, "{\pr_1}", from=1-5, to=1-7]
\arrow[""{name=5, anchor=center, inner sep=0}, "{\id_c \times f}", from=5-3, to=5-5]
\arrow["{\pr_{1,3}}", from=3-5, to=5-5]
\arrow["{\pr_1}"', from=3-3, to=5-3]
\arrow[""{name=6, anchor=center, inner sep=0}, "{(\id_c,\id_d,f)}", from=3-3, to=3-5]
\arrow["{\text{\textcircled{\raisebox{-.5pt}{1}}}}"{description}, draw=none, from=0, to=1]
\arrow["{\text{\textcircled{\raisebox{-.5pt}{2}}}}"{description}, draw=none, from=2, to=6]
\arrow["{\text{\textcircled{\raisebox{-.5pt}{3}}}}"{description}, draw=none, from=4, to=3]
\arrow["{\text{\textcircled{\raisebox{-.5pt}{4}}}}"{description}, draw=none, from=6, to=5]
\end{tikzcd}\]
are pullbacks, where $(\id_c,f,\id_d)$ and $(\id_c,\id_d,f)$ denote the universally obtained maps
% https://q.uiver.app/?q=WzAsNyxbMSwwLCJjIFxcdGltZXMgZCJdLFswLDAsImMiXSxbMiwwLCJkIl0sWzEsMSwiYyBcXHRpbWVzIGQgXFx0aW1lcyBkIl0sWzAsMiwiYyJdLFsxLDIsImQiXSxbMiwyLCJkIl0sWzAsMywiKFxcaWRfYyxmLFxcaWRfZCkiLDAseyJzdHlsZSI6eyJib2R5Ijp7Im5hbWUiOiJkYXNoZWQifX19XSxbMyw0LCJcXHByXzEiLDAseyJsYWJlbF9wb3NpdGlvbiI6NjAsImN1cnZlIjoyfV0sWzMsNSwiXFxwcl8yIl0sWzMsNiwiXFxwcl8zIiwyLHsibGFiZWxfcG9zaXRpb24iOjYwLCJjdXJ2ZSI6LTJ9XSxbMCwyLCJcXHByXzIiXSxbMCwxLCJcXHByXzEiLDJdLFsxLDQsIlxcaWRfYyIsMix7ImN1cnZlIjoxfV0sWzEsNSwiZiIsMix7ImxhYmVsX3Bvc2l0aW9uIjo0MCwiY3VydmUiOjF9XSxbMiw2LCJcXGlkX2QiLDAseyJjdXJ2ZSI6LTF9XV0=
\[\begin{tikzcd}
c & {c \times d} & d \\
& {c \times d \times d} \\
c & d & d,
\arrow["{(\id_c,f,\id_d)}", dashed, from=1-2, to=2-2]
\arrow["{\pr_2'}", from=2-2, to=3-2]
\arrow["{\pr_3'}"'{pos=0.6}, curve={height=-12pt}, from=2-2, to=3-3]
\arrow["{\pr_2}", from=1-2, to=1-3]
\arrow["{\pr_1}"', from=1-2, to=1-1]
\arrow["{\id_c}"', curve={height=6pt}, from=1-1, to=3-1]
\arrow["f"'{pos=0.2}, curve={height=6pt}, from=1-1, to=3-2]
\arrow["{\pr_1'}"{pos=0.6}, curve={height=12pt}, from=2-2, to=3-1, crossing over]
\arrow["{\id_d}", curve={height=-6pt}, from=1-3, to=3-3]
\end{tikzcd} \ \ \ \begin{tikzcd}
c & {c \times d} & d \\
& {c \times d \times d} \\
c & d & d.
\arrow["{(\id_c,f,\id_d)}", dashed, from=1-2, to=2-2]
\arrow["{\pr_3'}"'{pos=0.6}, curve={height=-12pt}, from=2-2, to=3-3]
\arrow["{\pr_2}", from=1-2, to=1-3]
\arrow["{\pr_1}"', from=1-2, to=1-1]
\arrow["{\id_c}"', curve={height=6pt}, from=1-1, to=3-1]
\arrow["{\id_d}", curve={height=-6pt}, from=1-3, to=3-3]
\arrow["f"'{pos=0.1}, curve={height=24pt}, from=1-1, to=3-3]
\arrow["{\pr_2'}", from=2-2, to=3-2, crossing over]
\arrow["{\pr_1'}"{pos=0.6}, curve={height=12pt}, from=2-2, to=3-1, crossing over]
\end{tikzcd}\]
We note also that $(f,\id_d) = \pr_{2,3} \circ (\id_c,f,\id_d)$.  Beginning with the identity inequality $\exists_{\id_c \times f} \top_c \leqslant \exists_{\id_c \times f} \top_c$, we conclude that
\begingroup
\renewcommand{\arraystretch}{1.5} % Default value: 1
\[\begin{array}{l@{\hskip 6pt}r@{\hskip 6pt}c@{\hskip 6pt}l@{\hskip -30pt}l}
&\exists_{\id_c \times f} P(\id_c \times f)(\top_{c\times d}) &\leqslant &\exists_{\id_c \times f}P(f)( \top_d), & \\
\implies & P(\id_c,f,\id_d) \exists_{(\id_c,\id_d,f)}\top_{c \times d}& \leqslant & P(f,\id_d)\exists_{\Delta_d}\top_d & \text{ by \textcircled{\raisebox{-.5pt}{1}}, \textcircled{\raisebox{-.5pt}{2}} and B.-C.,} \\
\implies & P(\id_c,f,\id_d) \exists_{(\id_c,\id_d,f)}\top_{c \times d} & \leqslant & P(\id_c,f,\id_d)P(\pr_{2,3})\exists_{\Delta_d}\top_d, & \\
\implies & \exists_{(\id_c,f,\id_d)}( \top_{c \times d} \land P(\id_c,f,\id_d)\exists_{(\id_c,\id_d,f)} \top_{c \times d}) & \leqslant & P(\pr_{2,3})\exists_{\Delta_d}\top_d & 
%\text{ by $\exists_{(\id_c,f,\id_d)} \dashv P({\id_c,f,\id_d})$,}
\\
\implies & \exists_{(\id_c,f,\id_d)} \top_{c \times d} \land \exists_{(\id_c,\id_d,f)} \top_{c \times d} & \leqslant &  P(\pr_{2,3})\exists_{\Delta_d}\top_d & \text{ by Frobenius,}\\
\implies & \exists_{(\id_c,f,\id_d)}P(\pr_1)( \top_c ) \land \exists_{(\id_c,\id_d,f)} P(\pr_1)( \top_c) &\leqslant  & P(\pr_{2,3})\exists_{\Delta_d}\top_d, & \\
\implies & P(\pr_{1,2})\exists_{\id_c \times f} \top_c \land P(\pr_{1,3})\exists_{\id_c \times f} \top_c &\leqslant  & P(\pr_{2,3})\exists_{\Delta_d}\top_d, & \text{ by \textcircled{\raisebox{-.5pt}{3}}, \textcircled{\raisebox{-.5pt}{4}} and B.-C.}
\end{array}\]
\endgroup
We need now only note that $U \leqslant \top_c$ to achieve our desired inequality that
\[P(\pr_{1,2})\exists_{\id_c \times f} U \land P(\pr_{1,3})\exists_{\id_c \times f} U \leqslant P(\pr_{2,3})\exists_{\Delta_d} \top_d.\]

%%%%%%%%%%%%

The final inequality is obtained via 
\[U = \exists_{\pr_1 \circ( \id_c \times f)} U = \exists_{\pr_1} \exists_{\id_c \times f} U.\]
\end{proof}

This assignment of arrows in functorial in the sense that, given composable arrows
\[\begin{tikzcd}
(c,U) \ar{r}{f} & (d,V) \ar{r}{g}& (e,W)
\end{tikzcd}\]
in $\cat \rtimes P$, the composite of 
\[\begin{tikzcd}
(c,U) \ar{r}{\exists_{\id_c \times f} U} & (d,V) \ar{r}{\exists_{\id_d \times g} V} &(e,W)
\end{tikzcd}\]
in ${\bf Syn}(P)$, i.e. the predicate 
\[\exists_{\pr_{1,3}} ( P(\pr_{1,2}) \exists_{\id_c \times f} U \land P(\pr_{2,3}) \exists_{\id_d \times g} V) ,\]
is equal to $\exists_{\id_c \times gf} U $.  To demonstrate this equality, we first note that all the squares in the diagram
% https://q.uiver.app/?q=WzAsOCxbMCwwLCJjIl0sWzIsMCwiYyBcXHRpbWVzIGQiXSxbMCwyLCJjIFxcdGltZXMgZSJdLFsyLDIsImMgXFx0aW1lcyBkIFxcdGltZXMgZSJdLFswLDQsImMiXSxbMiw0LCJjIFxcdGltZXMgZCJdLFs0LDAsImQiXSxbNCwyLCJkIFxcdGltZXMgZSJdLFswLDEsIlxcaWRfYyBcXHRpbWVzIGYiXSxbMCwyLCJcXGlkX2MgXFx0aW1lcyBnZiIsMl0sWzIsNCwiXFxwcl8xIiwyXSxbMiwzLCIoXFxpZF9jIFxcdGltZXMgZixcXGlkX2UpIl0sWzQsNSwiXFxpZF9jIFxcdGltZXMgZiJdLFszLDUsIlxccHJfezEsMn0iXSxbMSwzLCIoXFxpZF9jLFxcaWRfZCBcXHRpbWVzIGcpIl0sWzMsNywiXFxwcl97MiwzfSJdLFsxLDYsIlxccHJfMiJdLFs2LDcsIlxcaWRfZCBcXHRpbWVzIGciXSxbOCwxMSwie1xcdGV4dHtcXHRleHRjaXJjbGVke1xccmFpc2Vib3h7LS41cHR9ezF9fX19IiwxLHsic2hvcnRlbiI6eyJzb3VyY2UiOjIwLCJ0YXJnZXQiOjIwfSwic3R5bGUiOnsiYm9keSI6eyJuYW1lIjoibm9uZSJ9LCJoZWFkIjp7Im5hbWUiOiJub25lIn19fV0sWzExLDEyLCJ7XFx0ZXh0e1xcdGV4dGNpcmNsZWR7XFxyYWlzZWJveHstLjVwdH17Mn19fX0iLDEseyJzaG9ydGVuIjp7InNvdXJjZSI6MjAsInRhcmdldCI6MjB9LCJzdHlsZSI6eyJib2R5Ijp7Im5hbWUiOiJub25lIn0sImhlYWQiOnsibmFtZSI6Im5vbmUifX19XSxbMTYsMTUsIntcXHRleHR7XFx0ZXh0Y2lyY2xlZHtcXHJhaXNlYm94ey0uNXB0fXszfX19fSIsMCx7InNob3J0ZW4iOnsic291cmNlIjoyMCwidGFyZ2V0IjoyMH0sInN0eWxlIjp7ImJvZHkiOnsibmFtZSI6Im5vbmUifSwiaGVhZCI6eyJuYW1lIjoibm9uZSJ9fX1dXQ==
\[\begin{tikzcd}
c && {c \times d} && d \\
\\
{c \times e} && {c \times d \times e} && {d \times e} \\
\\
c && {c \times d}
\arrow[""{name=0, anchor=center, inner sep=0}, "{\id_c \times f}", from=1-1, to=1-3]
\arrow["{\id_c \times gf}"', from=1-1, to=3-1]
\arrow["{\pr_1}"', from=3-1, to=5-1]
\arrow[""{name=1, anchor=center, inner sep=0}, "{(\id_c \times f,\id_e)}", from=3-1, to=3-3]
\arrow[""{name=2, anchor=center, inner sep=0}, "{\id_c \times f}", from=5-1, to=5-3]
\arrow["{\pr_{1,2}}", from=3-3, to=5-3]
\arrow["{(\id_c,\id_d \times g)}", from=1-3, to=3-3]
\arrow[""{name=3, anchor=center, inner sep=0}, "{\pr_{2,3}}", from=3-3, to=3-5]
\arrow[""{name=4, anchor=center, inner sep=0}, "{\pr_2}", from=1-3, to=1-5]
\arrow["{\id_d \times g}", from=1-5, to=3-5]
\arrow["{{\text{\textcircled{\raisebox{-.5pt}{1}}}}}"{description}, draw=none, from=0, to=1]
\arrow["{{\text{\textcircled{\raisebox{-.5pt}{2}}}}}"{description}, draw=none, from=1, to=2]
\arrow["{{\text{\textcircled{\raisebox{-.5pt}{3}}}}}", draw=none, from=4, to=3]
\end{tikzcd}\]
are pullbacks.  Therefore, we have that $\exists_{\pr_{1,3}} ( P(\pr_{1,2}) \exists_{\id_c \times f} U \land P(\pr_{2,3}) \exists_{\id_d \times g} V) $ is equal to
\begingroup
\renewcommand{\arraystretch}{1.5} % Default value: 1
\[\begin{array}{l@{\hskip 4pt}l@{\hskip 4pt}l}
& \exists_{\pr_{1,3}}(\exists_{(\id_c \times f,\id_e)} P(\pr_1) (U) \land \exists_{(\id_c,\id_d \times g)} P(\pr_2 ) (V)), & \text{ by {{\text{\textcircled{\raisebox{-.5pt}{2}}}}}, {{\text{\textcircled{\raisebox{-.5pt}{3}}}}} and B.-C.}, \\
= & \exists_{\pr_{1,3}} \exists_{(\id_c \times f,\id_e)} (P(\pr_1) (U) \land P((\id_c \times f,\id_e)) \exists_{(\id_c,\id_d \times g)} P(\pr_2)( V)), & \text{ by Frobenius}, \\
= & P(\pr_1) (U)\land P((\id_c \times f,\id_e)) \exists_{(\id_c,\id_d \times g)} P(\pr_2)( V)), & \\
= & P(\pr_1) (U) \land \exists_{\id_c \times gf} P(\id_c \times f)P(\pr_2)( V), & \text{ by {{\text{\textcircled{\raisebox{-.5pt}{2}}}}} and B.-C.}, \\
= & \exists_{\id_c \times gf}( P(\id_c \times gf)P(\pr_1) (U) \land P(\id_c \times f)P(\pr_2)( V)), & \text{ by Frobenius}, \\
= &  \exists_{\id_c \times gf} (U \land P(f) (V)), & \\
= & \exists_{\id_c \times gf} U, & \text{ since }U \leqslant P(f)(V).
\end{array}\]
\endgroup
Thus, we achieve the desired equality $\exists_{\pr_{1,3}} ( P(\pr_{1,2}) \exists_{\id_c \times f} U \land P(\pr_{2,3}) \exists_{\id_d \times g} V) =\exists_{\id_c \times gf} U $.  We observe also that the identity arrow $(c,U) \xrightarrow{\id_c} (c,U)$ in $\cat \rtimes P$ gets assigned to the identity arrow $(c,U) \xrightarrow{\exists_{\Delta_c}U}(c,U)$ in ${\bf Syn}(P)$.  Hence, we obtain a functor $\cat \rtimes P \to {\bf Syn}(P)$.

\begin{df}
{\rm
We denote by $\zeta^P \colon \cat \rtimes P \to {\bf Syn}(P)$ the functor that sends an arrow $(c,U) \xrightarrow{f} (d,V)$ of $\cat \rtimes P$ to the arrow $(c,U) \xrightarrow{\exists_{\id_c \times f} U} (d,V)$ of ${\bf Syn}(P)$.
}
\end{df}

Intuitively, the functor $\zeta^P \colon \cat \rtimes P \to {\bf Syn}(P)$ is ‘adjoining those arrows that ought to exist’ (i.e. those for which a provably functional relation
exists) and ‘identifying those arrows that ought to be the same’ (i.e. those for which the internal language of $P$ proves an identity of
arrows).

\begin{prop}\label{prop:zetaisdense}
Let $(P,(J_c)_{c \in \cat})$ be an existential doctrinal site.  The functor $\zeta^P \colon \cat \rtimes P \to {\bf Syn}(P)$ defines a dense morphism of sites
\[\zeta^P \colon (\cat \rtimes P,J_{\rtimes}) \to ({\bf Syn}(P) ,J_{\bf Syn}),\]
and hence there is an equivalence of toposes $\Sh(\cat \rtimes P,J_{\rtimes}) \simeq \Sh({\bf Syn}(P) ,J_{\bf Syn})$.
\end{prop}
\begin{proof}
We check the four conditions of Definition \ref{densemorph} one by one.
\begin{enumerate}
\item The first condition, Definition \ref{densemorph}\ref{enumdesnsemorph1}, is immediate once we recall that a family of morphisms
\[ \{\,(c_i,U_i) \xrightarrow{f_i} (d,V) \mid i \in I\,\} \text{ in } \cat \rtimes P\]
is $J_\rtimes$-covering if and only if $\{\,\exists_{f_i} U_i \to V \mid i \in I\,\}$ is $J_d$-covering, while simultaneously the family of morphisms
\[\{\,(c_i,U_i) \xrightarrow{\exists_{\id_{c_i} \times f_i} U_i} (d,V) \mid i \in I\,\} \text{ in }{\bf Syn}(P)\]
is $J_{\bf Syn}$-covering if and only if $\{\,\exists_{f_i} U_i = \exists_{\pr_2} \exists_{\id_{c_i} \times f_i} U_i \to V\mid i \in I\,\}$ is $J_d$-covering.
\item Condition \ref{enumdensemorph2} follows since the functor $\zeta^P$ is surjective on objects.
\item Let $(c,U) \xrightarrow{W} (d,V)$ be a provably functional relation, i.e. an arrow of ${\bf Syn}(P)$.  As $W \leqslant P(\pr_1)(U)$, there is an arrow $(c \times d,W) \xrightarrow{\pr_1} (c,U)$ of $\cat \rtimes P$.  Consider the diagram
\begin{equation}\label{diag:composite1}
\begin{tikzcd}
	(c \times d ,W) \ar{rrrr}{\zeta^P(\pr_1) = \exists_{\id_{c \times d} \times \pr_1}W} &&&& (c,U) \ar{r}{W} & (d,V)
\end{tikzcd}
\end{equation}
of ${\bf Syn}(P)$.  To satisfy condition \ref{enumdensemorph3}, it suffices to show that the arrow $(c \times d ,W) \xrightarrow{\zeta^P(\pr_1)}  (c,U)$ is $J_{\bf Syn}$-covering and the composite of (\ref{diag:composite1}) is in the image of $\zeta^P$.  The former follows from the inequality
\[U \leqslant \exists_{\pr_1} W = \exists_{\pr_2} \exists_{\id_{c \times d} \times \pr_1} W.\]

The composite of (\ref{diag:composite1}) is given by $\exists_{\pr_{1,2,4}} (\pr_{1,2,3} \exists_{\id_{c \times d} \times \pr_1} (W) \land \pr_{2,4} (W))$.  We first gather the necessary observations: that both the squares
\begin{equation}\label{diag:uglypb1}
\begin{tikzcd}
	c \times d \times d \ar{rr}{\pr_{1,2}} \ar{d}[']{(\id_{c \times d} \times \pr_1,\id_d)} && c\times d \ar{d}{\id_{c \times d} \times \pr_1} \\
	c \times d \times c \times d \ar{rr}{\pr_{1,2,3}} && c \times d \times c
\end{tikzcd}
\end{equation}
and
\begin{equation}\label{diag:uglypb2}
\begin{tikzcd}
	c \times d \ar{rr}{\id_{ c \times d} \times \pr_2}\ar{d}[']{\pr_2} && c \times d \times d \ar{d}{\pr_{2,3}} \\
	d \ar{rr}{\Delta_d} && d \times d
\end{tikzcd}
\end{equation}
are both pullbacks, and that there is the inequality
\begin{equation}\label{provfuncrel:cond2}
P(\pr_{1,2})(W) \land P(\pr_{1,3})(W) \leqslant P(\pr_{2,3}) \exists_{\Delta_d} \top_d.
\end{equation}
Therefore, we can demonstrate that $\exists_{\pr_{1,2,4}} (P(\pr_{1,2,3}) \exists_{\id_{c \times d} \times \pr_1} (W) \land P( \pr_{2,4}) (W))$ is equal to
\begingroup
\renewcommand{\arraystretch}{1.5} % Default value: 1
\[\begin{array}{l@{\hskip 4pt}l@{\hskip 4pt}l}
& \exists_{\pr_{1,2,4}} (\exists_{(\id_{c \times d} \times \pr_1,\id_d)} P(\pr_{1,2} ) (W) \land P(\pr_{2,4})( W)), & \text{ by (\ref{diag:uglypb1}) and B.-C.}, \\
= & \exists_{\pr_{1,2,4}} \exists_{(\id_{c \times d} \times \pr_1,\id_d)} (P(\pr_{1,2} ) (W) \land P(\id_{c \times d} \times \pr_1,\id_d)P(\pr_{2,4}) (W))& \text{ by Frobenius}, \\
= & P(\pr_{1,2})(W) \land P(\pr_{1,3})(W), & \\
= & P(\pr_{1,2})(W) \land P(\pr_{1,3})(W) \land P(\pr_{2,3}) \exists_{\Delta_d} \top_d, & \text{ using (\ref{provfuncrel:cond2}),} \\
= & P(\pr_{1,2})(W) \land P(\pr_{1,3})(W) \land   \exists_{\id_{ c \times d} \times \pr_2} P(\pr_2)(\top_{d}), & \text{ by (\ref{diag:uglypb2}) and B.-C.,} \\
= & P(\pr_{1,2})(W) \land P(\pr_{1,3})(W) \land   \exists_{\id_{ c \times d} \times \pr_2} \top_{c \times d}, & \\
= & \exists_{\id_{ c \times d} \times \pr_2} (P(\id_{ c \times d} \times \pr_2)P(\pr_{1,2})(W )\land P(\id_{ c \times d} \times \pr_2)P(\pr_{1,3})(W) \land \top_{c \times d} ), & \text{ by Frobenius}, \\
= &  \exists_{\id_{ c \times d} \times \pr_2} ( P(\id_{c \times d})(W) \land P(\id_{c \times d})(W)), & \\
= &  \exists_{\id_{ c \times d} \times \pr_2} W. &
\end{array}\]
\endgroup
Hence, since $\exists_{\pr_{1,2,4}} (P(\pr_{1,2,3}) \exists_{\id_{c \times d} \times \pr_1} (W) \land P( \pr_{2,4}) (W)) = \exists_{\id_{ c \times d} \times \pr_2} W$, we conclude that the composite of (\ref{diag:composite1}) is the image under $\zeta^P$ of the arrow
\[(c \times d,W) \xrightarrow{ \pr_2} (d,V) \in \cat \rtimes P.\]

\item Let
\[\begin{tikzcd}
(c,U) \ar[shift left]{r}{f} \ar[shift right]{r}[']{g} & (d,V)
\end{tikzcd}\]
be a pair of parallel arrows of $\cat \rtimes P$ that are identified in the image of $\zeta^P$, i.e. $\exists_{\id_c \times f} U = \exists_{\id_c \times g} U$.  To satisfy condition \ref{enumdesnemorph4}, we aim to find a $J_\rtimes$ cover $S$ of $(c,U)$ such that, for all $h\in S$, $f \circ h = g \circ h$.  Let
\[\begin{tikzcd}
e \ar{r}{h} & c \ar[shift left]{r}{f} \ar[shift right]{r}[']{g} & d
\end{tikzcd}\]
be an equalizer diagram in $\cat$, and hence the square
\[\begin{tikzcd}
e \ar{r}{h} \ar{d}{h} & c \ar{d}{\id_c \times f} \\
c \ar{r}{\id_c \times g} & c \times d
\end{tikzcd}\]
is a pullback.  The fork
\[\begin{tikzcd}
(c,P(h)(U)) \ar{r}{h} &(c,U) \ar[shift left]{r}{f} \ar[shift right]{r}[']{g} & (d,V)
\end{tikzcd}\]
in $\cat \rtimes P$ commutes, and the arrow $(c,P(h)(U)) \xrightarrow{h}(c,U)$ is $J_\rtimes$-covering since
\begin{align*}
\exists_{\id_c \times f} U = \exists_{\id_c \times g} U & \implies U \leqslant P(\id_c \times f)\exists_{\id_c \times g} U, \\
& \implies U \leqslant \exists_h P(h)(U).
\end{align*}
\end{enumerate}
\end{proof}

We also observe that the choice of functor $\zeta^P$ is suitably natural in the following sense.  Let ${\bf Sites}$ denote the 2-category of sites, morphisms of sites and natural transformations between these.  By the above, there exist two ways of assigning a site to an existential doctrinal site $(P,(J_c)_{c \in \cat}) \in {\bf ExDocSites}$:
\[(P,(J_c)_{c \in \cat}) \mapsto (\cat \rtimes P,J_\rtimes) \text{ and } (P,(J_c)_{c \in \cat}) \mapsto ({\bf Syn}(P),J_{\bf Syn}),\]
and moreover these assignments yield two 2-functors ${\bf ExDocSites} \rightrightarrows {\bf Sites}$, which we respectively denote as
\[\rtimes \colon {\bf ExDocSites} \to {\bf Sites}, \ \ {\bf Syn} \colon {\bf ExDocSites} \to {\bf RegSites} \subseteq {\bf Sites}.\]
It is easily checked that the morphisms of sites $\zeta^P \colon (\cat \rtimes P,J_\rtimes) \to ({\bf Syn}(P),J_{\bf Syn})$, for each existential doctrinal site $(P,(J_c)_{c \in \cat}) \in {\bf ExDocSites}$, consitute the components of a natural transformation 
\[\begin{tikzcd}
{{\bf ExDocSites}} && {{\bf Sites}.}
\arrow[""{name=0, anchor=center, inner sep=0}, "{{\bf Syn}}"', shift right=4, from=1-1, to=1-3]
\arrow[""{name=1, anchor=center, inner sep=0}, "\rtimes", shift left=4, from=1-1, to=1-3]
\arrow["\zeta", shorten <=2pt, shorten >=2pt, Rightarrow, from=1, to=0]
\end{tikzcd}\]

\begin{rem}
{\rm

Let $\theory$ be a geometric theory over a signature $\Sigma$.  The textbook construction of the classifying topos $\topos_\theory$ of $\theory$, as can be found in \cite[\S D1.4]{elephant}, \cite[\S X]{SGL} or \cite[\S 1.4]{TST}, uses the \emph{syntactic site} $(\cat_\theory,J_\theory)$ of the theory.  Recall that
\begin{enumerate}
\item the syntactic category $\cat_\theory$ of $\theory$ is the category
\begin{enumerate}
	\item whose objects are the $\theory$-provable equivalence classes of formulae $\form{\varphi}{x}$ over $\Sigma$,
	\item and whose arrows $[\theta] \colon \form{\varphi}{x} \to \form{\psi}{y}$ are $\theory$-provable equivalence classes of $\theory$-provably functional formulae, that is formulae $\theta$ in the context $\vec{x},\vec{y}$ such that $\theory$ proves the sequents
	\[\theta \vdash_{\vec{x},\vec{y}} \varphi\land \psi, \ \varphi \vdash_{\vec{x}} \exists \vec{y} \, \theta , \ \theta \land \theta[\vec{z}/\vec{y}] \vdash_{\vec{x},\vec{y},\vec{z} } \vec{y} = \vec{z};\]
\end{enumerate}
\item in the syntactic topology $J_\theory$ on $\cat_\theory$, a family of arrows $\{\,[\theta_i] \colon \form{\varphi_i}{x_i} \to \form{\psi}{y}\mid i \in I \,\}$ is $J_\theory$-covering if and only if $\theory$ proves the sequent
\[\psi \vdash_{\vec{y}} \bigvee_{i \in I} \exists \vec{x}_i \, \theta_i.\]
\end{enumerate}
We recognise the category $\cat_\theory$ as the syntactic category construction ${\bf Syn}(F^\theory)$ for the geometric doctrine $F^\theory \colon \Con \to \Frm_{\rm open}$ considered in Example \ref{prototypical}.  Similarly, the syntactic topology $J_\theory$ is precisely the topology $J_{\bf Syn}$ obtained from the existential doctrinal site $(F^\theory,(J_{\vec{x}})_{\vec{x} \in \Con})$, where each fibre $F^\theory(\vec{x})$ has been endowed with the canonical topology (the induced topology $J_\rtimes$ on $\Con \rtimes F^\theory$ is precisely the topology $K_{F^\theory}$ that we have already encountered).

Thus, by Proposition \ref{prop:zetaisdense}, we conclude that both $(\Con \rtimes F^\theory,K_{F^\theory})$ and $(\cat_\theory,J_\theory)$ are sites of definition for the classifying topos $\topos_\theory$, as visualised in the `bridge' diagram:
\[\begin{tikzcd}
&& {\begin{matrix} \mathcal{E}_\mathbb{T} \\ \textit{the classifying} \\ \textit{topos of } \mathbb{T},\end{matrix}} \\
\\
{\begin{matrix} F^\mathbb{T} \colon {\bf Con}_\Sigma \to {\bf Frm}_{\rm open} \\ \textit{the geometric doctrine},\end{matrix}} &&&& {\textup{ }\textup{ }\textup{ }\begin{matrix} (\mathcal{C}_\mathbb{T},J_{\mathbb{T}}) \\ \textit{the syntactic site.} \end{matrix}}
\arrow[curve={height=-30pt}, dashed, tail reversed, from=1-3, to=3-5]
\arrow[curve={height=30pt}, dashed, tail reversed, from=1-3, to=3-1]
\end{tikzcd}\]
Recall that every internal locale of $\sets^\Con$ is of the form $F^\theory$ for some theory over the same sorts as the signature $\Sigma$.  Thus, we recover the bijection between the internal locales of $\sets^\Con$ and geometric theories (up to Morita-equivalence) with the same sorts as $\Sigma$ that was observed in the single-sorted case (when $\Con \simeq \Finsets$) in \cite[Theorem D3.2.5]{elephant}.  This equivalence can also be deduced via the theory of \emph{localic expansions} (see \cite[\S 7.1]{TST}), since every theory $\theory$ with the same sorts as $\Sigma$ is a localic expansion of $\mathbb{O}_\Sigma$ -- the empty theory over the sorts of $\Sigma$, whose classifying topos is $\sets^\Con$.

The site $(\Con \rtimes F^\theory,K_{F^\theory})$ was dubbed the \emph{alternative syntactic site} of the theory $\theory$ in \cite{myselfpresentation}.  The site $(\Con \rtimes F^\theory,K_{F^\theory})$ can be more amenable than the standard syntactic site for some calculations.  Notably, every arrow $\formm{\varphi}{x} \xrightarrow{\sigma} \formm{\psi}{y}$ is a restriction of the arrow
\[
\begin{tikzcd}
\formm{\varphi}{x} \ar{r}{\sigma} \ar[tail]{d} & \formm{\psi}{y} \ar[tail]{d} \\
\formm{\top}{x} \ar{r}{\sigma} & \formm{\top}{y},
\end{tikzcd}
\]
and moreover, since $\formm{\top}{z}$ is the product $\prod_{z_i \in \vec{z}} (\,z_i,\top\,)$ in $\Con \rtimes F^\theory$, the arrow $\formm{\top}{x} \xrightarrow{\sigma} \, \formm{\top}{y}$, labelled by a relabelling $\sigma \colon \vec{y} \to \vec{x}$ -- i.e. a sort-preserving map between finite sets, is induced universally as in the diagram
\[
\begin{tikzcd}
\formm{\top}{x} \ar[dashed]{rr}{\sigma} \ar{d}{\pr_{\sigma(y_i)}} && \formm{\top}{y} \ar{d}{\pr_{y_i}} \\
(\,\sigma(y_i),\top\,) \ar{rr}{\id_{\sigma(y_i)} = \id_{y_i}} && (\,y_i,\top\,).
\end{tikzcd}
\]
There are, of course, desirable properties of the syntactic site $(\cat_\theory,J_\theory)$ that are not shared by the site $(\Con \rtimes F^\theory,K_{F^\theory})$.  For example, the topology $J_\theory$ is subcanonical (see \cite[Lemma X.4.5]{SGL}) while the topology $K_{F^\theory}$ on $\Con \rtimes F^\theory$ is not (see \cite[Remark 3.11]{myselfintlocmorph}).

}\end{rem}

%%%%%%%%%%%%%%%%%%%%%%%%%%%%%%%%%%%%%%%%%%%%%%%%%

\subsection{The geometric completion of a regular site}\label{subsec:geocompforregcat}

Finally, we are able to combine the above to define the geometric completion of a regular site, which sends a regular site to a geometric category.  Unsurprisingly, this amounts to assigning to each regular site $(\cat,K)$ the full subcategory of $\Sh(\cat,K)$ spanned by subobjects of representables.

\begin{df}
{\rm 
We denote by ${\bf GeomCat}$ the 2-category of \emph{geometric categories}, the 2-category
\begin{enumerate}
\item whose objects are geometric categories -- regular categories whose subobject lattices have arbitrary joins that are preserved by pullback,
\item whose 1-cells are \emph{geometric functors} -- regular functors that also preserve joins of subobjects,
\item and whose 2-cells are natural transformations between these.
\end{enumerate}
}
\end{df}

Each geometric category $\mathcal{G}$ can be equipped with the \emph{geometric topology} $J_{\rm Geom}$, the Grothendieck topology whose covering families are the jointly epimorphic ones, to obtain a regular site $(\mathcal{G},J_{\rm Geom})$.  In light of Remark \ref{rem:subobjtop}, this is the topology whose restriction $J_{\rm Geom}|_{\Sub_\mathcal{G}(g)}$ to the subobject lattice $\Sub_\mathcal{G}(g)$, for $g \in \mathcal{G}$, is the topology where
\[\{\,e_i \rightarrowtail h \mid i \in I\,\} \text{ is a  $J_{\rm Geom}|_{\Sub_\mathcal{G}(g)}$-cover} \iff h = \bigvee_{i \in I} e_i.  \]
This assignment of a regular site to a geometric category $\mathcal{G}$ is easily observed to determine a full and faithful 2-embedding ${\bf GeomCat} \hookrightarrow {\bf RegSites}$.

\begin{thm}
There is a quasi 2-adjunction
\[\begin{tikzcd}
{{\bf RegSites}} && {{\bf GeomCat}}
\arrow[""{name=0, anchor=center, inner sep=0}, shift left=3, hook', from=1-3, to=1-1]
\arrow[""{name=1, anchor=center, inner sep=0}, "{\GC^{\rm Cat}}", shift left=3, from=1-1, to=1-3]
\arrow["\dashv"{anchor=center, rotate=-90}, draw=none, from=1, to=0]
\end{tikzcd}\]
for which each square in the diagram
\begin{equation}\label{diag:squareof2adjs}\begin{tikzcd}
{{\bf ExDocSites}} && {{\bf GeomDoc}_{\rm cart}} \\
\\
{{\bf RegSites}} && {{\bf GeomCat}}
\arrow[""{name=0, anchor=center, inner sep=0}, shift left=3, hook', from=3-3, to=3-1]
\arrow[""{name=1, anchor=center, inner sep=0}, "{\GC^{\rm Cat}}", shift left=3, from=3-1, to=3-3]
\arrow[""{name=2, anchor=center, inner sep=0}, "{{\bf Syn}}"', shift right=3, from=1-3, to=3-3]
\arrow[""{name=3, anchor=center, inner sep=0}, "{\Sub_{(-)}}"', shift right=3, from=3-3, to=1-3]
\arrow[""{name=4, anchor=center, inner sep=0}, "{\Sub_{(-)}}"', shift right=3, from=3-1, to=1-1]
\arrow[""{name=5, anchor=center, inner sep=0}, "{{\bf Syn}}"', shift right=3, from=1-1, to=3-1]
\arrow[""{name=6, anchor=center, inner sep=0}, shift left=3, hook', from=1-3, to=1-1]
\arrow[""{name=7, anchor=center, inner sep=0}, "\GC", shift left=3, from=1-1, to=1-3]
\arrow["\dashv"{anchor=center, rotate=-90}, draw=none, from=1, to=0]
\arrow["\dashv"{anchor=center}, draw=none, from=2, to=3]
\arrow["\dashv"{anchor=center}, draw=none, from=5, to=4]
\arrow["\dashv"{anchor=center, rotate=-90}, draw=none, from=7, to=6]
\end{tikzcd}\end{equation}
commutes.
\end{thm}

\begin{proof}
In order to obtain the commutativity of (\ref{diag:squareof2adjs}), we define $\GC^{\rm Cat} \colon {\bf RegSites} \to {\bf GeomCat}$ as the composite ${\bf Syn} \circ \GC \circ \Sub_{(-)}$.  The quasi 2-adjunction is then obtained by the natural equivalences, for each regular site $(\cat,K) \in {\bf RegSites}$ and geometric category $\mathcal{G} \in {\bf GeomCat}$,
\begin{align*}
\Hom_{\bf GeomCat}({\bf Syn}( \GC(\Sub_{\cat},K_{\Sub})),\mathcal{G}) &\simeq \Hom_{\bf GeomDoc_{\rm cart}}(\GC(\Sub_{\cat},K_{\Sub}),\Sub_{\mathcal{G}}), \\
& \simeq \Hom_{\bf ExDocSites}((\Sub_\cat,K_\Sub),(\Sub_\mathcal{G},K_{\Sub_{\mathcal{G}}})), \\
& \simeq \Hom_{\bf RegSites}((\cat,K),(\mathcal{G},J_{\rm Geom})),
\end{align*}
where in the last equivalence we have used that $\Sub_{(-)} \colon {\bf RegSites} \to {\bf ExDocSites}$ is full and faithful.
\end{proof}

%%%%%%%%%%%%%%%%%%%%%%%%%%%
%%%%%%%%%%%%%%%%%%%%%%%%%

%%%%%%%%%%%%%%%%%%%%%%%%%%%%%%%%%%%%%%%%%%%%%%
%%%%%%%%%%%%%%%%%%%%%%%%%%%%%%%%%%%%%%%%%%%%%%
%%%%%%%%%%%%%%%%%%%%%%%%%%%%%%%%%%%%%%%%%%%%%%%%%%%%

%%%%%%%%%%%%%%%%%%%%%%%%%%%%%%%%%%%%%%%%%%%%%%%%%%%
%%%%%%%%%%%%%%%%%%%%%%%%%%%%%%%%%%%%%%%%%%%%%%%%%%%

%%%%%%%%%%%%%%%%%%%%%%%%%%%%%%%%%%%%%%%%%%
%%%%%%%%%%%%%%%%%%%%%%%%%%%%%%%%%%%%%%%%%%%%%%

%%%%%%%%%%%%%%%%%%%%%%%%%%%
%%%%%%%%%%%%%%%%%%%%%%%%%

%%%%%%%%%%%%%%%%%%%%%%%%%%%%%
\section{Coarse geometric completions}\label{sec:coarse}

Because the geometric completion takes a Grothendieck topology as a second argument, it is an idempotent completion (see Theorem \ref{thm:univprop}).  This is in contrast to many of the other completions of doctrines considered in the literature (e.g. Trotta's existential completion \cite{trotta}).  The geometric completion would not be idempotent if we did not have the ability to take suitable topologies as a second argument.

Consider the 2-element frame $\2$.  Being a frame, there is an isomorphism $J_{\rm can} \text{-}{\bf Idl}(\2) \cong \2$, but one can easily calculate that $J_{\rm triv}\text{-}{\bf Idl}(\2)$ is the 3-element frame ${\bf 3}$ (i.e. the opens of the Sierpinski space).  We can interpret this behaviour as a `loss of information' by taking a \emph{coarser} Grothendieck topology $J_{\rm triv} \subseteq J_{\rm can}$ on $\2$.  In order to relate the geometric completion to other completions of doctrines considered in the literature, we consider in this section the behaviour of the geometric completion for doctrines when, for each geometric doctrine $\Lb$, we deliberately choose a coarser Grothendieck topology $J_\Lb^A \subseteq K_\Lb$ on the category $\cat \rtimes \Lb$ (or indeed forget the Grothendieck topology entirely by assigning the trivial topology $J_{\rm triv}$ to $\cat\rtimes \Lb$).  

We thus arrive at the notion of a \emph{coarse geometric completion} - a monad $\GC_A$ acting on a 2-full 2-subcategory of ${\bf DocSites}$.  As evidenced by the example given above, this monad $\GC_A$ is no longer idempotent (unless each $J_\Lb^A$ is chosen to be $K_\Lb$), unlike the geometric completion monad $\GC$.  We will observe that each coarse geometric completion is instead \emph{lax-idempotent}.  After establishing the lax-idempotency of a coarse geometric completion in Corollary \ref{coro:coarseislax}, we demonstrate in Corollary \ref{coro:freegcforcat} how this yields a lax-idempotent geometric completion for cartesian, regular and coherent categories.

%%%%%%%%%%%%%%%%%%%%%%%%%%%%%%%%%%%%%%%%%%
%%%%%%%%%%%%%%%%%%%%%%%%%%%%%%%%%%%%%%%%%%%%

% We have already seen that $\GC(\2,J_{\rm can}) \not \cong \GC(\2,J_{\rm triv})$.  Similarly, every geometric doctrine $\Lb \colon \cat^{op} \to \Frm_{\rm open}$ is a relative coherent doctrine, but we do not (in general) have an isomorphism between $\Lb \cong \GC(\Lb,K_\Lb)$ and $\GC(\Lb,J_{\rm Coh})$.  This is due to the fact that the Grothendieck topology $K_\Lb$ (in general) strictly contains $J_{\rm Coh}$ -- in essence, by considering the coarser topology $J_{\rm Coh}$, we have `forgotten' some of the geometric data of $\Lb$.

% It may still be of interest to consider the behaviour of the geometric completion when geometric doctrines are endowed with coarser topologies.  To that end, in this subsection we develop an abstract framework for \emph{coarse geometric completions}.  We then prove that all coarse geometric completions are lax-idempotent.

\begin{df}\label{df:coarsecompl}
{\rm  
A \emph{coarse geometric completion} consists of the following data.
\begin{enumerate}
\item We are given a 2-full 2-subcategory $A\text{-}{\bf Doc} \subseteq {\bf DocSites}$.  The objects of $A\text{-}{\bf Doc}$ we call $A$-doctrines and their morphisms we call $A$-doctrine morphisms.

\item There is a 2-subcategory ${\bf GeomDoc}_A \subseteq {\bf GeomDoc}$ which is full on 1-cells and 2-cells satisfying the following conditions.
\begin{enumerate}
	\item For each $\Lb \in {\bf GeomDoc}_A$, there is a choice of Grothendieck topology $J^A_\Lb$ on the category $\cat \rtimes \Lb$ which is coarser than the topology $K_\Lb$, i.e. $J^A_\Lb \subseteq K_\Lb$, such that $(\Lb,J^A_\Lb)$ is an object of $A\text{-}{\bf Doc}$.  Moreover, the choice of topology $J^A_\Lb$ is functorial in the sense that, for each morphism of geometric doctrines $(F,a) \colon \Lb \to \Lb'$, there is a morphism of $A$-doctrines
	\[(F,a) \colon (\Lb,J^A_\Lb) \to (\Lb',J^A_{\Lb'}).\]
	In other words, there is a 2-full 2-embedding ${\bf GeomDoc}_A \hookrightarrow A\text{-}{\bf Doc}$.
	
	\item For each object $(P,J) \in A\text{-}{\bf Doc}$, the geometric completion $\GC(P,J) \colon \cat^{op} \to \Frm_{\rm open}$ is contained in ${\bf GeomDoc}_A$ and the unit $\eta^{(P,J)} \colon P \to \GC(P,J)$ defines a morphism of $A$-doctrines:
	\[\eta^{(P,J)} \colon (P,J) \to (\GC(P,J),J^A_{\GC(P,J)}).\]
\end{enumerate}
\end{enumerate}
}
\end{df}

\begin{thm}\label{thm:coarsecompl}
Let $A\text{-}{\bf Doc} \subseteq {\bf DocSites}$ and ${\bf GeomDoc}_A \subseteq {\bf GeomDoc}$ define a coarse geometric completion.  There is a strict 2-adjunction 
\[\begin{tikzcd}
{A\text{-}{\bf Doc}} && {{\bf GeomDoc}_A}
\arrow[""{name=0, anchor=center, inner sep=0}, shift left=2, hook', from=1-3, to=1-1]
\arrow[""{name=1, anchor=center, inner sep=0}, "\GC_A", shift left=2, from=1-1, to=1-3]
\arrow["\dashv"{anchor=center, rotate=-90}, draw=none, from=1, to=0]
\end{tikzcd}\]
where $\GC_A$ is the 2-functor
\[\begin{tikzcd}
{A\text{-}{\bf Doc}} \ar[hook]{r} & {\bf DocSites} \ar{r}{\GC} & {\bf GeomDoc}.
\end{tikzcd}\]
\end{thm}
\begin{proof}
For each $(P,J) \in A\text{-}{\bf Doc}$ and $\Lb \in {\bf GeomDoc}_A$, the natural isomorphism on objects of the categories
\begin{equation}\label{strictadj2} \Hom_{A\text{-}{\bf Doc}}((P,J),(\Lb,J^A_\Lb)) \cong \Hom_{{\bf GeomDoc}_A}(\GC_A(P,J),\Lb)\end{equation}
acts by sending a morphism of geometric doctrines $(F,a) \colon \GC_A(P,J) \to \Lb$ to the composite
\[\begin{tikzcd}
(P,J) \ar{rr}{\etaPJ} && (\GC_A(P,J),J^A_{\GC_A(P,J)}) \ar{rr}{(F,a)} && (\Lb,J^A_\Lb)
\end{tikzcd}\]
and, vice versa, sending an arrow $(F,a) \colon (P,J) \to (\Lb,J^A_\Lb)$ to the morphism of geometric doctrines $(F,\mathfrak{a})$ as induced by the diagram
\[\begin{tikzcd}
(P,J) \ar{rrd}[']{(F,a)} \ar{rr}{\eta^{(P,J)}} & &(\GC(P,J),J^A_{\GC(P,J)}) \ar{r}\ar[dashed]{d}{(F,\mathfrak{a})} & (\GC(P,J),K_{\GC(P,J)}) \ar[dashed]{d}{(F,\mathfrak{a})} \\
&& (\Lb,J^A_\Lb) \ar{r} & (\Lb,K_\Lb)
\end{tikzcd}\]
and Theorem \ref{thm:univprop}.  That this extends to an isomorphism on arrows, and hence the isomorphism of categories (\ref{strictadj2}), follows from Proposition \ref{prop:strict2adj} and the fact that $A\text{-}{\bf Doc} \subseteq {\bf DocSites}$ and ${\bf GeomDoc}_A \subseteq {\bf GeomDoc}$ are both full on 2-cells.
\end{proof}

Of course, $\2$ is a quotient frame (or \emph{sublocale}) of ${\bf 3}$.  Similarly, the coarse geometric completion  of a geometric doctrine $\GC(\Lb,J^A_\Lb)$ is related to the geometric doctrine $\Lb$ by a point-wise surjective morphism of geometric doctrines $\GC(\Lb,J^A_\Lb) \to \Lb$ (or \emph{internal sublocale embedding}) corresponding to the inclusion
\[\Sh(\cat \rtimes \Lb,K_\Lb) \rightarrowtail \Sh(\cat \rtimes \Lb,J^A_\Lb)\]
(see \cite[Theorem 5.7]{myselfintlocmorph} -- if $J^A_\Lb$ is the trivial topology, the morphism $\GC(\Lb,J_{\rm triv}) \to \Lb$ is precisely the $K_\Lb$-closure operation from Definition \ref{df:closure}).

%%%%%%%%%%%%%%%%
%%%%%%%%%%%%%%%%%

\paragraph{Lax idempotency for coarse geometric completions.}  As previously mentioned, the strict 2-adjunction 
\[\begin{tikzcd}
{A\text{-}{\bf Doc}} && {{\bf GeomDoc}_A}
\arrow[""{name=0, anchor=center, inner sep=0}, shift left=2, hook', from=1-3, to=1-1]
\arrow[""{name=1, anchor=center, inner sep=0}, "\GC_A", shift left=2, from=1-1, to=1-3]
\arrow["\dashv"{anchor=center, rotate=-90}, draw=none, from=1, to=0]
\end{tikzcd}\]
of a coarse geometric completion is not necessarily idempotent.  We dedicate the remainder of this section to showing that $\GC_A$ satisfies a weaker form of idempotency: lax-idempotency.

Let $T \colon \cat\to \cat$ be a 2-monad with unit $\eta \colon \id_\cat \to T$ and multiplication $\mu \colon T^2 \to T$.  The 2-monad $T$ is \emph{lax-idempotent}\footnote{In \cite{kockmonads} lax-idempotent monads are called \emph{KZ-doctrines}.  Using this terminology would be confusing in the context of doctrines in the sense of Lawvere.} if the composites of the diagram
\[\begin{tikzcd}
TA \ar[shift right]{r}[']{\eta_{TA}} & \ar[shift right]{l}[']{\mu_A} T^2 A,
\end{tikzcd}\]
although perhaps not strictly equal to the identities $\id_{TA}$ and $\id_{T^2A}$, can be related by canonical 2-cells such that there is an adjunction $\mu_A \dashv \eta_{TA}$ (see \parencite[Proposition 1.2]{kockmonads}).  Specifically, we require a require a 2-cell
\[\begin{tikzcd}
{TA} && {T^2A,}
\arrow[""{name=0, anchor=center, inner sep=0}, "{T\eta_A}", curve={height=-12pt}, from=1-1, to=1-3]
\arrow[""{name=1, anchor=center, inner sep=0}, "{\eta_{TA}}"', curve={height=12pt}, from=1-1, to=1-3]
\arrow["{\lambda_A}", shorten <=3pt, shorten >=3pt, Rightarrow, from=0, to=1]
\end{tikzcd}\]
natural in $A$, such that the composites
\[\begin{tikzcd}
A & TA && {T^2A,}
\arrow[""{name=0, anchor=center, inner sep=0}, "{T\eta_A}", curve={height=-12pt}, from=1-2, to=1-4]
\arrow[""{name=1, anchor=center, inner sep=0}, "{\eta_{TA}}"', curve={height=12pt}, from=1-2, to=1-4]
\arrow["{\eta_A}", from=1-1, to=1-2]
\arrow["{\lambda_A}", shorten <=3pt, shorten >=3pt, Rightarrow, from=0, to=1]
\end{tikzcd}
\begin{tikzcd}
TA && {T^2A} & A
\arrow[""{name=0, anchor=center, inner sep=0}, "{T\eta_A}", curve={height=-12pt}, from=1-1, to=1-3]
\arrow[""{name=1, anchor=center, inner sep=0}, "{\eta_{TA}}"', curve={height=12pt}, from=1-1, to=1-3]
\arrow["{\mu_A}", from=1-3, to=1-4]
\arrow["{\lambda_A}", shorten <=3pt, shorten >=3pt, Rightarrow, from=0, to=1]
\end{tikzcd}\]
are both identity 2-cells (see \cite[Definition 1.1]{kockmonads}).

Often it can be more tractable, if circumlocutory, to demonstrate lax-idempotency by an equivalent condition regarding the algebras of the monad.  Recall from \cite{lackkelly} that, given a pair of (strict) $T$-alebras $(A,a)$ and $(B,a)$, a \emph{lax morphism} of $T$-algebras is a pair $(f,\alpha)$ where $f \colon A \to B$ is an arrow of $\cat$ while $\alpha$ is a 2-cell $\alpha \colon b \circ Tf \to f \circ a$ that fills the square
\[\begin{tikzcd}
TA & TB \\
A & B
\arrow[""{name=0, anchor=center, inner sep=0}, "Tf", from=1-1, to=1-2]
\arrow["a"', from=1-1, to=2-1]
\arrow["b", from=1-2, to=2-2]
\arrow[""{name=1, anchor=center, inner sep=0}, "f"', from=2-1, to=2-2]
\arrow["\alpha", shorten <=4pt, shorten >=4pt, Rightarrow, from=0, to=1]
\end{tikzcd}\]
and satisfies the coherence conditions
\begin{equation}\label{cohdiag1}\begin{tikzcd}
{T^2A} & {T^2B} && {T^2A} & {T^2B} \\
TA & TB & {=} & TA & TB \\
A & B && A & B
\arrow[""{name=0, anchor=center, inner sep=0}, "f"', from=3-1, to=3-2]
\arrow["a"', from=2-1, to=3-1]
\arrow["b", from=2-2, to=3-2]
\arrow[""{name=1, anchor=center, inner sep=0}, "Tf", from=2-1, to=2-2]
\arrow["{T^2f}", from=1-1, to=1-2]
\arrow["{\mu_A}"', from=1-1, to=2-1]
\arrow["{\mu_B}", from=1-2, to=2-2]
\arrow[""{name=2, anchor=center, inner sep=0}, "f"', from=3-4, to=3-5]
\arrow["a"', from=2-4, to=3-4]
\arrow["b", from=2-5, to=3-5]
\arrow[""{name=3, anchor=center, inner sep=0}, "Tf", from=2-4, to=2-5]
\arrow["Ta"', from=1-4, to=2-4]
\arrow["TB", from=1-5, to=2-5]
\arrow[""{name=4, anchor=center, inner sep=0}, "{T^2f}", from=1-4, to=1-5]
\arrow["\alpha", shorten <=4pt, shorten >=4pt, Rightarrow, from=1, to=0]
\arrow["\alpha", shorten <=4pt, shorten >=4pt, Rightarrow, from=3, to=2]
\arrow["T\alpha"{pos=0.4}, shorten <=4pt, shorten >=12pt, Rightarrow, from=4, to=3]
\end{tikzcd}\end{equation}
and
\begin{equation}\label{cohdiag2}\begin{tikzcd}
A & B && A & B \\
TA & TB & {=} \\
A & B && A & B.
\arrow["{\eta_A}"', from=1-1, to=2-1]
\arrow["{\eta_B}", from=1-2, to=2-2]
\arrow["f", from=1-1, to=1-2]
\arrow[""{name=0, anchor=center, inner sep=0}, "Tf", from=2-1, to=2-2]
\arrow[""{name=1, anchor=center, inner sep=0}, "f"', from=3-1, to=3-2]
\arrow["a"', from=2-1, to=3-1]
\arrow["b", from=2-2, to=3-2]
\arrow["f", from=3-4, to=3-5]
\arrow["f", from=1-4, to=1-5]
\arrow["{\id_A}"', from=1-4, to=3-4]
\arrow["{\id_B}", from=1-5, to=3-5]
\arrow["\alpha", shorten <=4pt, shorten >=4pt, Rightarrow, from=0, to=1]
\end{tikzcd}\end{equation}
It is shown in \parencite[Theorem 6.2]{lackkelly} that $T$ is lax-idempotent if and only if for each pair $(A,a)$ and $(B,b)$ of (strict) $T$-algebras and a morphism $f \colon A \to B$ there is a unique 2-cell $\alpha \colon b \circ Tf \to f \circ a$ such that $(f,\alpha) \colon (A,a) \to (B,b)$ is a lax morphism of $T$-algebras.

We require two lemmas concerning the algebras of $\GC_A$ to demonstrate that a coarse geometric completion $\GC_A \colon A\text{-}{\bf Doc} \to A\text{-}{\bf Doc}$ is lax-idempotent.  Since in each example of coarse geometric completions we will consider in Examples \ref{coarseexs}, the Grothendieck topology $J$ given on $\cat \rtimes P$ for an $A$-doctrine $(P,J) \in A\text{-}{\bf Doc}$ is chosen for us, in what follows we simplify notation and denote the object $(P,J)$ of $A\text{-}{\bf Doc}$ by simply $P$.  Also in aid of legibility, if $(G,b) = \xi \colon P \to Q$ is a morphism of $A$-doctrines, we will abuse notation and write $\xi$ for the natural transformation $b \colon P \to Q\circ G^{op}$.

\begin{lem}\label{GCAalglem1}
Let $P \colon \cat^{op} \to \PreOrd$ and $Q \colon \dcat^{op} \to \PreOrd$ be a pair of $A$-doctrines and let $\xi \colon \GC_A(P) \to P$ and $\zeta \colon \GC_A(Q) \to Q$ be natural transformations such that the triangles
\[\begin{tikzcd}
P \ar{r}{\eta^P} \ar[equal]{rd} & \GC_A(P) \ar{d}{\xi}& Q \ar{r}{\eta^Q} \ar[equal]{rd} & \GC_A(Q) \ar{d}{\zeta} \\
& P, & & Q
\end{tikzcd}\]
commute.  Given a functor $F \colon \cat \to \dcat$ and a natural transformation $a \colon P \to Q \circ F^{op}$, for each arrow $d \xrightarrow{f} c \in \cat$ and $x \in P(d)$, there is an inequality
\[\zeta_c \circ \exists_{\GC_A(Q)(F(f))} \circ \eta^Q_d \circ a_d(x)  \leqslant a_c \circ \xi_c\circ\exists_{\GC_A(P)(f)} \circ \eta^P_d(x).\]
\end{lem}
\begin{proof}
Firstly, using the inequality $\eta^P_d(x) \leqslant \GC_A(P)(f) \circ \exists_{\GC_A(P)(f)} \circ \eta^P_d(x)$, we deduce that
\begin{align*}
\eta^Q_d \circ a_d(x) = \eta^Q_d \circ a_d \circ \xi_d \circ \eta^P_d (x) &\leqslant \eta^Q_d \circ a_d \circ \xi_d \circ \GC_A(P)(f)\circ  \exists_{\GC_A(P)(f)} \circ \eta^P_d(x), \\
& =  \GC_A(Q)(F(f))\circ \eta^Q_c \circ a_c \circ \xi_c \circ  \exists_{\GC_A(P)(f)} \circ \eta^P_d(x).
\end{align*}
Thus, by the adjunction $\exists_{\GC(Q)(F(f))} \dashv \GC(Q)(F(f))$, we have that
\[ \exists_{\GC_A(Q)(F(f))} \circ\eta^Q_d \circ a_d(x) \leqslant \eta^Q_c \circ a_c \circ \xi_c \circ  \exists_{\GC_A(P)(f)} \circ \eta^P_d(x),\]
and we therefore obtain the desired inequality
\begin{align*}
\zeta_c \circ \exists_{\GC_A(Q)(F(f))} \circ \eta^Q_d \circ a_d(x) & \leqslant \zeta_c \circ \eta^Q_c \circ a_c \circ \xi_c\circ\exists_{\GC_A(P)(f)} \circ \eta^P_d(x), \\
& = a_c \circ \xi_c\circ\exists_{\GC_A(P)(f)} \circ \eta^P_d(x).
\end{align*}
\end{proof}

\begin{lem}\label{GCAalglem2}
If $(P,\xi)$ is a $\GC_A$-algebra, then $\xi_c \colon \GC_A(P)(c) \to P(c)$ preserves joins for all $c \in \cat$.
\end{lem}
\begin{proof}
Let us first show that $P(c)$ must have all joins.  For a subset $\{\,x_i \mid i \in I\,\}$ of $P(c)$, we claim that the join $\bigvee_{i \in I} x_i$ is given by $\xi_c \left(\bigvee_{i \in I} \eta^P_c(x) \right)$.  For each $i \in I$,
\[x_i = \xi_c \circ \eta^P_c(x) \leqslant \xi_c \left(\bigvee_{i \in I} \eta^P_c(x) \right)\]
while, if given $y \in P(c)$ with $x_i \leqslant y$ for all $i \in I$,
\[\xi_c \left(\bigvee_{i \in I} \eta^P_c(x) \right) \leqslant \xi_c \circ \eta^P_c(y) = y.\]

To show that $\xi_c$ preserves these joins, we first observe that the diagrams
\[\begin{tikzcd}
& \GC_A\GC_A(P) \ar{r}{\GC_A(\xi)} \ar{d}{\mu^P} & \GC_A(P)  \ar{d}{\xi}  &  \GC_A(P) \ar{r}{\xi} \ar{d}{\eta^{\GC_A(P)}} & P \ar{d}{\eta^P} \\
\GC_A(P) \ar{ru}{\eta^{\GC_A(P)}} \ar[equal]{r} & \GC_A(P) \ar{r}{\xi} & P, & \GC_A\GC_A(P) \ar{r}{\GC_A(\xi)} & \GC_A(P)
\end{tikzcd}\]
both commute -- the right-hand diagram commutes since $\GC_A$ is a monad and $(P,\xi)$ is a $\GC_A$-algebra, while the left-hand square commutes as $\eta$ is natural.  Note also that $\mu^P$ and $\GC_A(\xi)$ are morphisms of geometric doctrines.  In particular, for each $c \in \cat$, both $\mu^P_c$ and $\GC_A(\xi)_c$ commute with all joins.

Therefore, given a subset $\{\,S_i \mid i \in I\,\} \subseteq \GC_A(P)(c)$, we observe that
\begingroup
\renewcommand{\arraystretch}{1.5} % Default value: 1
\begin{align*}
\bigvee_{i \in I} \xi_c(S_i) = \xi_c\left(\bigvee_{i \in I} \eta^P_c \circ \xi_c(S_i)\right) & = \xi_c\left(\bigvee_{i \in I} \GC_A(\xi)_c \circ \eta_{c}^{\GC_A(P)}(S_i) \right) , \\
& = \xi_c \circ  \GC_A(\xi)_c \left(\bigvee_{i \in I} \eta_c^{\GC_A(P)} (S_i)\right), \\
& = \xi_c \circ \mu^P_c \left(\bigvee_{i \in I} \eta_c^{\GC_A(P)} (S_i)\right), \\
& = \xi_c  \left(\bigvee_{i \in I} \mu^P_c \circ \eta_c^{\GC_A(P)} (S_i)\right) = \xi_c\left(\bigvee_{i \in I} S_i \right)
\end{align*}
\endgroup
and hence joins are indeed preserved.
\end{proof}

Finally we complete the proof that $\GC_A$ is lax-idempotent.  The argument will be reminiscent of that found in \parencite[Theorem 5.6]{trotta}.

\begin{coro}\label{coro:coarseislax}
Each coarse geometric completion $\GC_A \colon A\text{-}{\bf Doc} \to A\text{-}{\bf Doc}$ is lax-idempotent.
\end{coro}
\begin{proof}
Let $(P,\xi)$ and $(Q,\zeta)$ be algebras of the 2-monad $\GC_A$.  For each morphism of $A$-doctrines $(F,a)$, we first demonstrate that the identity transformation $\id_F$ defines a 2-cell that fills the square
\[\begin{tikzcd}
{\GC_A(P)} && {\GC_A(Q)} \\
P && Q.
\arrow[""{name=0, anchor=center, inner sep=0}, "{(F,\mathfrak{a})}", from=1-1, to=1-3]
\arrow["\xi"', from=1-1, to=2-1]
\arrow["\zeta", from=1-3, to=2-3]
\arrow[""{name=1, anchor=center, inner sep=0}, "{(F,a)}"', from=2-1, to=2-3]
\arrow["\id_F", shorten <=4pt, shorten >=4pt, Rightarrow, from=0, to=1]
\end{tikzcd}\]
We thus need to demonstrate, for all $S \in \GC_A(P)(c)$, the inequality
\[ \zeta_c \circ \mathfrak{a}_c(S) \leqslant a_c \circ \xi_c (S).\]
By combining Remark \ref{rem:direct}, Lemma \ref{GCAalglem1} and Lemma \ref{GCAalglem2}, we obtain the desired inequality:
\begingroup
\renewcommand{\arraystretch}{1.5}
\begin{align*}
\zeta_c \circ \mathfrak{a}_c (S) & = \zeta_c \left(\bigvee_{(f,x) \in S} \exists_{\GC_A(Q)(F(f))} \circ \eta^Q_d \circ a_d(x) \right), \\
& = \bigvee_{(f,x) \in S} \zeta_c \circ \exists_{\GC_A(Q)(F(f))} \circ \eta^Q_d \circ a_d(x) , \\
&\leqslant \bigvee_{(f,x) \in S} a_c \circ \xi_c\circ\exists_{\GC_A(P)(f)} \circ \eta^P_d(x) , \\
& \leqslant a_c \circ \xi_c \left(\bigvee_{(f,x) \in S} \exists_{\GC_A(P)(f)} \circ \eta^P_d(x)\right) = a_c \circ \xi_c(S).
\end{align*}
\endgroup
Its trivially shown that $((F,a),\id_F)$ satisfies the coherence conditions (\ref{cohdiag1}) and (\ref{cohdiag2}).

For any other 2-cell $\alpha \colon \zeta \circ (F,\mathfrak{a}) \to (F,a) \circ \xi $ satisfying the coherence condition
\[\begin{tikzcd}
P & Q && P & Q \\
{\GC_A(P)} & {\GC_A(Q)} & {=} \\
A & B && P & Q,
\arrow["{\eta^P}"', from=1-1, to=2-1]
\arrow["{\eta^Q}", from=1-2, to=2-2]
\arrow["{(F,a)}", from=1-1, to=1-2]
\arrow[""{name=0, anchor=center, inner sep=0}, "{(F,\mathfrak{a})}", from=2-1, to=2-2]
\arrow[""{name=1, anchor=center, inner sep=0}, "{(F,a)}"', from=3-1, to=3-2]
\arrow["\xi"', from=2-1, to=3-1]
\arrow["\zeta", from=2-2, to=3-2]
\arrow["{(F,a)}", from=3-4, to=3-5]
\arrow["{(F,a)}", from=1-4, to=1-5]
\arrow["{\id_P}"', from=1-4, to=3-4]
\arrow["{\id_Q}", from=1-5, to=3-5]
\arrow["\alpha", shorten <=4pt, shorten >=4pt, Rightarrow, from=0, to=1]
\end{tikzcd}\]
the equality $\alpha = \id_F$ is forced, and so $((F,a),\id_F)$ is the unique such lax $\GC_A$-algebra morphism.

\end{proof}

%%%%%%%%%%%%%%%

\begin{exs}\label{coarseexs}
{\rm
We obtain, as an immediate application of Theorem \ref{thm:coarsecompl} and Proposition \ref{coro:coarseislax}, the following lax-idempotent 2-monads.
\begin{enumerate}
\item\label{coarseexfree} By ${\bf Doc}_{\rm flat}$ denote the full 2-subcategory of ${\bf DocSites}$ on objects of the form $(P,J_{\rm triv})$.  Equivalently ${\bf Doc}_{\rm flat}$ is the 2-subcategory of ${\bf Doc}$ on doctrines and \emph{flat} morphisms of doctrines.  The assignment of the trivial topology $J_{\rm triv}$ to each geometric doctrine $\Lb \in {\bf GeomDoc}$ induces a 2-embedding ${\bf GeomDoc} \hookrightarrow {\bf Doc}$ that satisfies the conditions of Definition \ref{df:coarsecompl}.  Thus, we obtain a coarse geometric completion that we will call the \emph{free} geometric completion:
\[\begin{tikzcd}
	{{\bf Doc}_{\rm flat}} && {{\bf GeomDoc}.}
	\arrow[""{name=0, anchor=center, inner sep=0}, shift left=3, hook', from=1-3, to=1-1]
	\arrow[""{name=1, anchor=center, inner sep=0}, "{\GC_{\Free}}", shift left=3, from=1-1, to=1-3]
	\arrow["\dashv"{anchor=center, rotate=-90}, draw=none, from=1, to=0]
\end{tikzcd}\]
In particular, this restricts to a strict 2-adjunction:
\[\begin{tikzcd}
	{{\bf PrimDoc}} && {{\bf GeomDoc}_{\rm cart},}
	\arrow[""{name=0, anchor=center, inner sep=0}, shift left=3, hook', from=1-3, to=1-1]
	\arrow[""{name=1, anchor=center, inner sep=0}, "{\GC_{\Free}}", shift left=3, from=1-1, to=1-3]
	\arrow["\dashv"{anchor=center, rotate=-90}, draw=none, from=1, to=0]
\end{tikzcd}\]
between the 2-category of primary doctrines and the 2-category of geometric doctrines indexed over cartesian base categories.

\item\label{coarseexreg} There is a 2-embedding of ${\bf GeomDoc} $ into the 2-subcategory ${\bf RelExDoc} \subseteq {\bf DocSites}$ of relative existential doctrines, given by sending a geometric doctrine $\Lb \in {\bf GeomDoc}$ to $(\Lb,J_{\rm Ex}) \in {\bf RelExDoc}$, satisfying the conditions of Definition \ref{df:coarsecompl}.  Hence, we obtain a coarse geometric completion that we call the \emph{existential} geometric completion:
\[\begin{tikzcd}
	{{\bf RelExDoc}} && {{\bf GeomDoc},}
	\arrow[""{name=0, anchor=center, inner sep=0}, shift left=3, hook', from=1-3, to=1-1]
	\arrow[""{name=1, anchor=center, inner sep=0}, "{\GC_{\rm Ex}}", shift left=3, from=1-1, to=1-3]
	\arrow["\dashv"{anchor=center, rotate=-90}, draw=none, from=1, to=0]
\end{tikzcd}\]
This strict 2-adjunction restricts to the 2-subcategories of existential doctrines and geometric doctrines over a cartesian base category:
\[\begin{tikzcd}
	{{\bf ExDoc}} && {{\bf GeomDoc}_{\rm cart}.}
	\arrow[""{name=0, anchor=center, inner sep=0}, shift left=3, hook', from=1-3, to=1-1]
	\arrow[""{name=1, anchor=center, inner sep=0}, "{\GC_{\rm Ex}}", shift left=3, from=1-1, to=1-3]
	\arrow["\dashv"{anchor=center, rotate=-90}, draw=none, from=1, to=0]
\end{tikzcd}\]

\item Similarly, we obtain a coarse coherent completion for relative coherent doctrines, the \emph{coherent} geometric completion:
\[\begin{tikzcd}
	{{\bf RelCohDoc}} && {{\bf GeomDoc},}
	\arrow[""{name=0, anchor=center, inner sep=0}, shift left=3, hook', from=1-3, to=1-1]
	\arrow[""{name=1, anchor=center, inner sep=0}, "{\GC_{\rm Coh}}", shift left=3, from=1-1, to=1-3]
	\arrow["\dashv"{anchor=center, rotate=-90}, draw=none, from=1, to=0]
\end{tikzcd}\]
where ${\bf GeomDoc} \hookrightarrow {\bf RelCohDoc}$ is the 2-embedding that sends a geometric doctrine $\Lb \in {\bf GeomDoc}$ to $(\Lb,J_{\rm Coh}) \in {\bf RelCohDoc}$.  Once again, this restricts to a strict 2-adjunction:
\[\begin{tikzcd}
	{{\bf CohDoc}} && {{\bf GeomDoc}_{\rm cart}.}
	\arrow[""{name=0, anchor=center, inner sep=0}, shift left=3, hook', from=1-3, to=1-1]
	\arrow[""{name=1, anchor=center, inner sep=0}, "{\GC_{\rm Coh}}", shift left=3, from=1-1, to=1-3]
	\arrow["\dashv"{anchor=center, rotate=-90}, draw=none, from=1, to=0]
\end{tikzcd}\]

\end{enumerate}

}
\end{exs}

\paragraph{Coarse geometric completions for categories.}

We now relate how the coarse geometric completions we've considered in Examples \ref{coarseexs} interact with the syntactic category construction from \S \ref{sec:syn}.  We will obtain (coarse) geometric completions for cartesian categories, regular categories and coherent categories.

A \emph{coherent category} (see \cite[\S A1.4]{elephant}, also called a \emph{logical category} in \cite{MR}) is a regular category whose subobject lattices $\Sub_\cat(c)$ have finite joins and, for each arrow $d \xrightarrow{f} c$ of $\cat$, $\Sub_\cat(f)$ preserves these finite joins. A \emph{coherent functor} $F \colon \cat \to \dcat$, between coherent categories, is a regular functor that preserves finite joins as well.  We denote by ${\bf Coh}$ the 2-category of coherent cateogries, coherent functors and natural transformations between these.

The 2-functors $\GC^{\rm Cat}_{\rm Ex} \colon {\bf Reg} \to {\bf GeomCat}$ and $\GC^{\rm Cat}_{\rm Coh} \colon {\bf Coh} \to {\bf GeomCat}$ constructed below in Corollary \ref{coro:freegcforcat} are evidently given by the composites
\[
\begin{tikzcd}
{\bf Reg} \ar[hook]{r} & {\bf RegSites} \ar{r}{\GC^{\rm Cat}} & {\bf GeomCat}
\end{tikzcd}\text{ and }
\begin{tikzcd}
{\bf Coh} \ar[hook]{r} & {\bf RegSites} \ar{r}{\GC^{\rm Cat}} & {\bf GeomCat},
\end{tikzcd}
\]
where ${\bf Reg} \hookrightarrow {\bf RegSites}$ (respectively, ${\bf Coh} \hookrightarrow {\bf RegSites}$) is the 2-embedding that sends a regular (resp., coherent) category $\cat$ to the regular site $(\cat,J_{\rm Reg})$ (resp., $(\cat,J_{\rm Coh})$).  Here $J_{\rm Reg}$ denotes the \emph{regular topology} and $J_{\rm Coh}$ denotes the \emph{coherent topology} (see \cite[Examples A2.1.11]{elephant}).

\begin{coro}\label{coro:freegcforcat}
There are lax-idempotent quasi 2-adjunctions:
\begin{enumerate}
\item\label{coro:freefcforcat:enum:cart} \begin{center}
\begin{tikzcd}
	{{\bf Cart}} && {{\bf GeomCat},}
	\arrow[""{name=0, anchor=center, inner sep=0}, shift left=3, hook', from=1-3, to=1-1]
	\arrow[""{name=1, anchor=center, inner sep=0}, "{\GC^{\rm Cat}_{\Free}}", shift left=3, from=1-1, to=1-3]
	\arrow["\dashv"{anchor=center, rotate=-90}, draw=none, from=1, to=0]
\end{tikzcd}\end{center}
\item \begin{center}
\begin{tikzcd}
	{{\bf Reg}} && {{\bf GeomCat},}
	\arrow[""{name=0, anchor=center, inner sep=0}, shift left=3, hook', from=1-3, to=1-1]
	\arrow[""{name=1, anchor=center, inner sep=0}, "{\GC^{\rm Cat}_{\rm Ex}}", shift left=3, from=1-1, to=1-3]
	\arrow["\dashv"{anchor=center, rotate=-90}, draw=none, from=1, to=0]
\end{tikzcd}
\end{center}
\item \begin{center}
\begin{tikzcd}
	{{\bf Coh}} && {{\bf GeomCat}.}
	\arrow[""{name=0, anchor=center, inner sep=0}, shift left=3, hook', from=1-3, to=1-1]
	\arrow[""{name=1, anchor=center, inner sep=0}, "{\GC^{\rm Cat}_{\rm Coh}}", shift left=3, from=1-1, to=1-3]
	\arrow["\dashv"{anchor=center, rotate=-90}, draw=none, from=1, to=0]
\end{tikzcd}
\end{center}
\end{enumerate}
\end{coro}

\begin{proof}
We will only spell out the proof for \ref{coro:freefcforcat:enum:cart}, the other quasi 2-adjoints being constructed in a similar fashion.  We define $\GC_{\Free}^{\rm Cat} \colon {\bf Cart} \to {\bf GeomCat}$ as the composite ${\bf Syn} \circ \GC_{\Free} \circ \Sub_{(-)}$, as in the diagram 
\[\begin{tikzcd}
{{\bf PrimDoc}} && {{\bf GeomDoc}_{\rm cart}} \\
\\
{{\bf Cart}} && {{\bf GeomCat}.}
\arrow[""{name=0, anchor=center, inner sep=0}, shift left=3, hook', from=3-3, to=3-1]
\arrow[""{name=1, anchor=center, inner sep=0}, "{\GC^{\rm Cat}_{\Free}}", shift left=3, from=3-1, to=3-3]
\arrow[""{name=2, anchor=center, inner sep=0}, shift left=3, hook', from=1-3, to=1-1]
\arrow[""{name=3, anchor=center, inner sep=0}, "{\GC_{\Free}}", shift left=3, from=1-1, to=1-3]
\arrow["{{\rm Sub}_{(-)}}", from=3-1, to=1-1]
\arrow[""{name=4, anchor=center, inner sep=0}, "{{\rm Sub}_{(-)}}"', shift right=3, from=3-3, to=1-3]
\arrow[""{name=5, anchor=center, inner sep=0}, "{{\bf Syn}}"', shift right=3, from=1-3, to=3-3]
\arrow["\dashv"{anchor=center, rotate=-90}, draw=none, from=1, to=0]
\arrow["\dashv"{anchor=center}, draw=none, from=5, to=4]
\arrow["\dashv"{anchor=center, rotate=-90}, draw=none, from=3, to=2]
\end{tikzcd}\]
The required natural equivalence of categories, for $\cat \in {\bf Cart}$ and $\mathcal{G} \in {\bf GeomCat}$,
\[\Hom_{\bf Cart}(\cat,\mathcal{G}) \simeq \Hom_{\bf GeomCat}({\bf Syn} (\GC_{\Free} (\Sub_{\cat})),\mathcal{G})\]
follows by the chain of equivalences
\begin{align*}
\Hom_{\bf Cart}(\cat,\mathcal{G}) & \simeq \Hom_{\bf PrimDoc}(\Sub_\cat,\Sub_{\mathcal{G}}), \\
& \simeq \Hom_{\bf GeomDoc_{\rm cart}}(\GC_{\Free}(\Sub_\cat),\Sub_{\mathcal{G}}), \\
& \simeq \Hom_{\bf GeomCat}({\Syn}(\GC_{\Free}(\Sub_\cat)),{\mathcal{G}}),
\end{align*}
where we have used that $\Sub_{(-)}$ is full and faithful.
\end{proof}

%%%%%%%%%%%%%%%%%%%%%%%%%%%%%%%%
\section{Sub-geometric completions}\label{sec:subgeo}

Having hinted at the existence of sub-geometric completions throughout, we finally turn to their systematic treatment.  The term sub-geometric completion is intended to convey the following vague sense: a completion $TP$ of a doctrine $P$ is `sub-geometric' if the data added by $T$ can be `seen' by a certain Grothendieck topology $J^T$ on the category $\cat \rtimes TP$, and has the property that $\GC(P,J_{\rm triv}) \cong \GC(TP,J^T)$ -- i.e. freely geometrically completing $P$ is the same as completing $P$ according to $T$, keeping track of this new information by $J^T$, and then geometrically completing.  We have already observed this phenomenon in \S \ref{subsec:altcont} with the free top completion, and we will see further examples below.  It is this vague notion of `sub-geometricity' that we seek to formalise in this section.

We proceed as follows.  Immediately below in \S \ref{freeandex} we present another motivating example for the theory of sub-geometric completions: we demonstrate that the existential completion of a primary doctrine due to Trotta \cite{trotta} satisfies our vague understanding of sub-geometricity as stated above.  We use this, and our study of the free top completion in \S \ref{subsec:altcont}, as intuition when introducing the formal definition of a sub-geometric completion in \S \ref{subsec:gensubgeo}.  We also discuss sufficient conditions under which a sub-geometric completion can automatically be deduced to be lax-idempotent.

In the remaining two subsections \S \ref{subsec:compatible} and \S \ref{subsec:exsubgeo}, we discuss several examples of sub-geometric completions.  In the former, we discuss sub-geometric completions obtained by considering special `compatible' subdoctrines of the free geometric completion.  In this way, we recover the existential completion as well as the coherent completion for primary doctrines.  We also relate these completions of primary doctrines to the corresponding \emph{regular completion} and \emph{coherent completion} of cartesian categories (see \cite{carboni}).  Finally, in the latter subsection, we give examples of `point-wise' sub-geometric completions.

%%%%%%%%%%%%%%%%%
%%%%%%%%%%%%%%%%%%%%
%%%%%%%%%%%%%%%%%%%%%%%

\subsection{The free geometric completion and the existential completion}\label{freeandex}

We begin by explicitly describing the free geometric completion $\GC_{\Free}(P)$ of a primary doctrine $P \in {\bf PrimDoc}$ as defined in Example \ref{coarseexs}\ref{coarseexfree}.  This is the geometric doctrine $\GC(P,J_{\rm triv}) \colon \cat^{op} \to \Frm_{\rm open}$ and thus, by Construction \ref{constrgc}, can be described in the following way.
\begin{enumerate}
\item For each object $c$ of $\cat$, an element $S$ of $\GC_{\Free}(P)(c)$ is a set of pairs $(f,x)$, where $d \xrightarrow{f} c$ is an arrow of $\cat$ and $x \in P(d)$, such that if $(f,x) \in S$, for each arrow $e \xrightarrow{g} d$ of $\cat$ and $y \in P(e)$, if $y \leqslant P(g)(x)$ then $(f \circ g,y) \in S$ too.  We order $\GC_{\Free}(P)(c)$ by inclusion.
\item For each arrow $d \xrightarrow{f} c$ of $\cat$, $\GC_{\Free}(P)(f) \colon\GC_{\Free}(P)(c) \to\GC_{\Free}(P)(d)$ sends $S \in \GC_{\Free}(P)(c)$ to
\[f^\ast(S) = \{\,(g,y) \mid (f \circ g,y) \in S\,\} \in \GC_{\Free}(P)(d).\]
\end{enumerate}
The description of the free geometric completion $\GC_{\Free}(P)$ given above is markedly similar to the \emph{existential completion} of a primary doctrine established in \parencite[\S 4]{trotta}, which we recall below.  We will be able to relate the two: the free geometric completion of a primary doctrine can be computed as the existential completion followed by the existential geometric completion (see Example \ref{coarseexs}\ref{coarseexreg}).

%%%%%%%%%%%%%%%%%%%%%%%%%%%%%%
%%%%%%%%%%%%%%%%%%%%%%%
%%%%%%%%%%%%%%%%%%%%%%%%%%%%%%%%%

%%%%%%%%%%%%%%%%%%%%%%%%%%%%%%%%%%%%%%%%%%%%
%%%%%%%%%%%%%%%%%%%%%%%%%%%%%%%%%%%%%%%%%%%
%%%%%%%%%%%%%%%%%%%%%%%%%%%%%%%%%%%%%%%%%%%

\paragraph{The existential completion.}

Recall from \cite{trotta} that the existential completion of a primary doctrine $P \colon \cat^{op} \to {\bf MSLat}$ is the functor $P^\exists \colon \cat^{op} \to {\bf MSLat}$ defined as follows.
\begin{enumerate}
\item Let $c$ be an object of $\cat$.  Consider the set whose elements are pairs $(f,x)$ where $d \xrightarrow{f} c$ is an arrow of $\cat$ and $x \in P(d)$.  We order this set by setting $(g,y) \leqslant (f,x)$ if there is an arrow $e \xrightarrow{h} d$, making the triangle commute
\[\begin{tikzcd}
e \ar{d}{h} \ar{rd}{g} & \\
d \ar{r}{f} & c,
\end{tikzcd}\]
such that $y \leqslant P(h)(x)$.  Define $P^\exists(c)$ as the poset obtained when we identify two elements such that $(f,x) \leqslant (g,y)$ and $(g,y) \leqslant (f,x)$.  Just as in \cite{trotta}, we will abuse notation and not differentiate between the pair $(f,x)$ and its equivalence class.
\item Given an arrow $e \xrightarrow{g} c$ of $\cat$, the map $P^\exists(g) \colon P^\exists(c) \to P^\exists(e)$ sends an element $(f,x) \in P^\exists(c)$ to $(k, P(h)(x)) \in P^\exists(e)$, where
\[\begin{tikzcd}
e \times_c d \ar{r}{k} \ar{d}{h} & e \ar{d}{g} \\
d \ar{r}{f} & c
\end{tikzcd}\]
is a pullback square in $\cat$.
\end{enumerate}
This is the `existential completion' of $P$ in following sense.
\begin{enumerate}
\item For each arrow $e \xrightarrow{g} c$ of $\cat$, the map $P^\exists(g) \colon P^\exists(c) \to P^\exists(e)$ has a left adjoint $\exists_{P^\exists(g)}$ that sends $(f,x) \in P^\exists(e)$ to $(g \circ f,x) \in P^\exists(d)$.  With these left adjoints, the doctrine $P^\exists$ satisfies the Frobenius and Beck-Chevalley conditions (see \parencite[Proposition 4.2 \& Theorem 4.3]{trotta}).
\item There is a natural transformation $\iota_P \colon P \to P^\exists$ given by sending $x \in P(c)$ to $(\id_c,x) \in P^\exists(c)$ (see \parencite[Proposition 4.10]{trotta}).
\item Given an existential doctrine $Q \colon \dcat^{op} \to {\bf MSLat}$, for each left exact functor $F \colon \cat \to \dcat$ and natural transformation $\alpha \colon P \to Q \circ F^{op}$ there exists a unique natural transformation $\alpha^\exists \colon P^\exists \to Q$ such that:
\begin{enumerate}
\item the triangle 
\[\begin{tikzcd}
P \ar{rd}[']{\alpha} \ar{r}{{\iota_P}}& P^\exists  \ar{d}{{\alpha^\exists}} \\
& Q \circ F^{op}
\end{tikzcd}\]
commutes,
\item for each arrow $e \xrightarrow{g} c$ of $\cat$, the square
\[\begin{tikzcd}
P^\exists(c) \ar{d}{{\alpha^\exists}_c} & P^\exists(e) \ar{l}[']{\exists_{P^\exists(g)}} \ar{d}{{\alpha^\exists}_e} \\
Q(F(c)) & \ar{l}[']{\exists_{Q(F(g))}} Q(F(e))
\end{tikzcd}\]
commutes (see \parencite[Theorem 4.14]{trotta}).
\end{enumerate}

\end{enumerate}
The construction presented here is slightly simplified to that found in \cite{trotta}.  Namely, we have added a left adjoint $\exists_{P^\exists(g)}$ to $P^\exists(g)$ for all arrows $e \xrightarrow{g} c$ of $\cat$, whereas in \cite{trotta} a generalised construction is given that freely adds a left adjoint $\exists_{P^\exists(g)}$ to $P^\exists(g)$ for arrows in a chosen class $\Lambda$ of morphisms of $\cat$ closed under pullbacks and compositions and containing all identities.  In \parencite[Proposition 4.9]{trotta}, it is shown that the existential completion defines a 2-functor
\[(-)^\exists \colon {\bf PrimDoc} \to {\bf ExDoc}.\]

We can now observe that the existential completion satisfies our loose notion of `sub-geometricity'.
\begin{prop}\label{prop:existsissubgeo}
For each primary doctrine $P \colon \cat^{op} \to {\bf MSLat}$, there is a natural isomorphism
\[\GC_{\Free}(P) \cong \GC_{\rm Ex}(P^\exists) = \GC(P^\exists,J_{\rm Ex}).\]
\end{prop}
\begin{proof}
This is immediate since the data of a down-set of $P^\exists(c)$, i.e. an element of $\GC_{\rm Ex}(P^\exists)$ by Proposition \ref{prop:geocompforexsites}, is precisely the data of an element $S \in \GC_{\Free}(P)(c)$.
\end{proof}

%%%%%%%%%%%%%%%%%%%%%%%%%%
%%%%%%%%%%%%%%%%%%%%%%%
%%%%%%%%%%%%%%%%%%%%%%%%%%

%%%%%%%%%%%%%%%%%%%%%%%%%%%%
%%%%%%%%%%%%%%%%%%%%%%%%%%%%%

%%%%%%%%%%%%%
%%%%%%%%%%%%%

%%%%%%%%%%%%%%
%%%%%%%%%%%%%%%%

%%%%%%%%%%%%%%
%%%%%%%%%%%%%%%%%%%
%%%%%%%%%%%%%%%%%

%%%%%%%%%%%%%%%%%%%
%%%%%%%%%%%%%%%%%
%%%%%%%%%%%%%%%%%%%%%%

\subsection{Generalised sub-geometric completions}\label{subsec:gensubgeo}

We now develop an abstract framework which captures the notion of a \emph{sub-geometric completion}.  We also give sufficient conditions under which a sub-geometric completion is automatically lax-idempotent.  In the latter subsections \S \ref{subsec:compatible} and \ref{subsec:exsubgeo}, we will demonstrate that the examples of sub-geometric completions we have encountered so far satisfy this generalised definition.

\begin{df}\label{df:subgeo}{\rm
Let $A\text{-}{\bf Doc}$ be a 2-full 2-subcategory of ${\bf Doc}$ (an object of $A\text{-}{\bf Doc}$ will be called an $A$-doctrine, and an arrow of $A\text{-}{\bf Doc}$ a morphism of $A$-doctrines) such that the image of the 2-functor
\[\begin{tikzcd}
A\text{-}{\bf Doc}_{\rm flat} \ar[hook]{r} & {\bf Doc}_{\rm flat} \ar{r}{\GC_{\Free}} & {\bf GeomDoc} \subseteq {\bf Doc}_{\rm flat}
\end{tikzcd}\]
is contained in $A\text{-}{\bf Doc}$, as is the unit $\eta^{(P,J_{\rm triv})} \colon P \to \GC_{\Free}(P)$ for each $A$-doctrine $P \in A\text{-}{\bf Doc}$, where here $A\text{-}{\bf Doc}_{\rm flat}$ represents the 2-full 2-subcategory of $\ADoc$ whose objects are $A$-doctrines and whose 1-cells are $A$-doctrine morphisms that are also flat.  A 2-monad $(T,\varepsilon,\nu)$ on the category $A\text{-}{\bf Doc}$, thought of as a completion of $A$-doctrines, is said to be \emph{sub-geometric} if it satisfies the following conditions.
\begin{enumerate}
%\item for each $A$-doctrine $P \colon \cat^{op} \to \PreOrd$ in $ A\text{-}{\bf Doc}$, $TP$ is still fibred over the same category $\cat$;
\item\label{df:subgeo:enum:algebra} For each $A$-doctrine $P \colon \cat^{op} \to \PreOrd$ in $ A\text{-}{\bf Doc}$, there exists a choice of $A$-doctrine morphism
\[\xi_P \colon T\GC_{\Free}(P) \to \GC_{\Free}(P) \]
such that $(\GC_{\Free}(P),\xi_P)$ defines a $T$-algebra.
%, and, moreover, this choice is natural in the sense that, for each morphism $h \colon P \to Q$ of $A \text{-}{\bf Doc}$, the morphism $\GC_{\Free}(h) \colon \GC_{\Free}(P) \to \GC_{\Free}(Q) $ yields a morphism of $T$-algebras $\GC_{\Free}(h) \colon (\GC_{\Free}(P),\xi_P) \to (\GC_{\Free}(Q),\xi_Q)$.
\item\label{df:subgeo:topologystuff} For each $A$-doctrine $P \colon \cat^{op} \to \PreOrd$, there exists a Grothendieck topology $J_{P}^T$ on the category $\dcat \rtimes TP$ such that:
\begin{enumerate}

	\item\label{subgeo:unitismorph} the unit $\varepsilon^P \colon P \to TP$ of the monad yields a morphism of doctrinal sites
	\[ \varepsilon^P \colon (P,J_{\rm triv}) \to (TP, J^T_P),\]

	\item\label{subgeo:algismorph} for each $A$-doctrine $P \colon \cat^{op} \to \PreOrd$, the $A$-doctrine morphism 
	\[\xi_P \colon T\GC_{\Free}(P) \to \GC_{\Free}(P)\]
	from above yields a morphism of doctrinal sites
	\[ \xi_P \colon (T\GC_{\Free}(P),J_{\GC_{\Free}(P)}^T) \to (\GC_{\Free}(P),K_{\GC_{\Free}(P)}), \]
	
	\item\label{subgeo:morphismorph} and the mapping that sends an $A$-doctrine $P \colon \cat^{op} \to \PreOrd$ to the doctrinal site $(TP,J^T)$ can be made functorial, i.e. each morphism of $A$-doctrines $\theta \colon P \to Q$ yields a morphism of doctrinal sites
	\[ T\theta \colon (TP,J_P^T) \to (TQ,J_Q^T).\]
	Thus there exists a functor $A\text{-}{\bf Doc} \to {\bf DocSites}$ that acts on objects by $P \mapsto (TP,J_P^T)$ (we label this functor by $J^T$).  In fact, since two morphisms $T\theta, T\theta' \colon TP\rightrightarrows TQ $ share the same 2-cells in both ${\bf Doc}$ and ${\bf DocSites}$, $J^T$ can be taken as a 2-functor.
\end{enumerate}
\end{enumerate}
}\end{df}

%%%%%%%%%%%%%%%%%

Condition \ref{df:subgeo:enum:algebra} of Definition \ref{df:subgeo} expresses that the completion $T$ of $A$-doctrines completes an $A$-doctrine $P$ to some fragment of the data of a geometric doctrine.  Evidently, if $\GC_{\Free}(P)$ already possesses the structure which $T$ is freely adding, then $\GC_{\Free}(P)$ is a $T$-algebra.  Condition \ref{df:subgeo:topologystuff} expresses that the added data can be `seen' by a choice of Grothendieck topology.

In Definition \ref{df:subgeo}, we also made the distinction between the category $A\text{-}{\bf Doc}$, on which the monad of the sub-geometric completion $(T,\varepsilon,\nu)$ acts, and the category $A\text{-}{\bf Doc}_{\rm flat}$.  This pedantry is necessary to include as examples all the completions we would expect to be sub-geometric.  For example, the free top completion does not induce a monad on ${\bf Doc}_{\rm flat}$: for a preorder $P$, the inclusion $P \hookrightarrow P \oplus \top = P^\top $ of $P$ into its free top completion, i.e. the unit of the completion, does not induce a morphism of doctrinal sites $(P,J_{\rm triv}) \to (P^\top,J_{\rm triv})$, but it does induce a morphism of doctrinal sites $(P,J_{\rm triv}) \to (P^\top,J^\top_{\rm triv})$ (see Lemma \ref{addingtopeq}).

%%%%%%%%%%%%%%%%%%%%

\begin{thm}\label{thm:subgeo}
For each sub-geometric completion $(T,\varepsilon,\nu)$ on a 2-subcategory $A\text{-}{\bf Doc} \subseteq {\bf Doc}$, the square
\[\begin{tikzcd}
A\text{-}{\bf Doc}_{\rm flat} \ar[hook]{r} \ar{d}[']{J^T} & {\bf Doc}_{\rm flat} \ar{d}{\GC_{\Free}} \\
{\bf DocSites} \ar{r}{\GC} & {\bf GeomDoc}
\end{tikzcd}\]
commutes up to 2-natural isomorphism.  In particular, for each $A$-doctrine $P \colon \cat^{op} \to \PreOrd$, there is an isomorphism $\GC_{\Free}(P) \cong \GC(T(P),J_P^T)$.
\end{thm}
\begin{proof}
The 2-natural isomorphism $\GC \circ J^T \cong \GC_{\Free}$ has components given by $\GC(\varepsilon^P)$ -- that is the arrow
\[\begin{tikzcd}
(P,J_{\rm triv}) \ar{rr}{\eta^{(P,J_{\rm triv})}} \ar{d}{\varepsilon^P} && (\GC_{\Free}(P),K_{\GC_{\Free}(P)}) \ar[dashed]{d}{\GC(\varepsilon^P)} \\
(TP,J_P^T) \ar{rr}{\eta^{(TP,J_P^T)}} && (\GC(TP,J_P^T),K_{\GC(TP,J_P^T)})
\end{tikzcd}\]
as induced by the universal property of the geometric completion.  By the 2-naturality of $\varepsilon$, it is trivial to see that the arrows $\GC(\varepsilon^P)$ are the components of a 2-natural transformation.

It remains to show that $\GC(\varepsilon^P)$ is an isomorphism for each $A$-doctrine $P \in A\text{-}{\bf Doc}$.  We exploit the universal property of the geometric completion to construct an inverse.  Consider the diagram
\begin{equation}\label{bigdiag:subgeo}
\begin{tikzcd}
&& {(\GC_{\Free}(P),K_{\GC_{\Free}(P)})} \\
{(P,J_{\rm triv})} && {(TP,J_P^T)} && {(\GC(TP,J_P^T),K_{\GC(TP,J_P^T)})} \\
\\
{(TP,J_P^T)} && {(T\GC_{\Free}(P),J_{\GC_{\Free}(P)}^T)} \\
\\
{(\GC(TP,J_P^T),K_{\GC(TP,J_P^T)})} && {(\GC_{\Free}(P),K_{\GC_{\Free}(P)})},
\arrow["{\varepsilon^P}", from=2-1, to=2-3]
\arrow["{\varepsilon^P}"', from=2-1, to=4-1]
\arrow["{T\eta^{(P,J_{\rm triv})}}", from=2-3, to=4-3]
\arrow["{\xi_P}", from=4-3, to=6-3]
\arrow["{\eta^{(TP,J_P^T)}}", from=2-3, to=2-5]
\arrow["{\Xi_P}", curve={height=-30pt}, dashed, from=2-5, to=6-3]
\arrow["{\eta^{(TP,J_P^T)}}"', from=4-1, to=6-1]
\arrow["{\GC(\varepsilon^P)}"', from=6-3, to=6-1]
\arrow["{\eta^{(P,J_{\rm triv})}}", curve={height=-12pt}, from=2-1, to=1-3]
\arrow["{\GC(\varepsilon^P)}", curve={height=-18pt}, from=1-3, to=2-5]
\arrow["{\eta^{(P,J_{\rm triv})}}"', from=2-1, to=6-3]
\end{tikzcd}
\end{equation}
where the arrow $\Xi_P \colon {(\GC(TP,J_P^T),K_{\GC(TP,J_P^T)})} \to {(\GC_{\Free}(P),K_{\GC_{\Free}(P)})}$ is induced by the universal property of the geometric completion.  We claim that the diagram (\ref{bigdiag:subgeo}) commutes -- it suffices to only check that the triangle
\begin{equation}\label{eq:bigtirangle}\begin{tikzcd}
{(P,J_{\rm triv})} && {(TP,J_P^T)} \\
\\
&& {(T\GC_{\Free}(P),J_{\GC_{\Free}(P)}^T)} \\
\\
&& {(\GC_{\Free}(P),K_{\GC_{\Free}(P)})}
\arrow["{\varepsilon^P}", from=1-1, to=1-3]
\arrow["{T\eta^{(P,J_{\rm triv})}}", from=1-3, to=3-3]
\arrow["{\xi_P}", from=3-3, to=5-3]
\arrow["{\eta^{(P,J_{\rm triv})}}"', from=1-1, to=5-3]
\end{tikzcd}\end{equation}
commutes (the other subdiagrams follow by definition).  The triangle (\ref{eq:bigtirangle}) commutes since
\begin{align*}
\xi_P \circ T\eta^{(P,J_{\rm triv})} \circ \varepsilon^P & = \xi_P \circ \varepsilon^{\GC_{\Free}(P)} \circ \eta^{(P,J_{\rm triv})} & \text{since $\varepsilon$ is natural,} \\
& = \eta^{(P,J_{\rm triv})} & \text{since $(\GC_{\Free}(P),\xi_P)$ is a $T$-algebra.}
\end{align*}
Therefore, by the universal property of the geometric completion, we obtain the desired equations:
\[\GC(\varepsilon^P) \circ \Xi_P = \id_{\GC(TP,J_P^T)}, \ \ \ \Xi_P \circ \GC(\varepsilon^P) = \id_{\GC_{\Free}(P)}.\]
\end{proof}

\begin{rem}\label{rem:subgeoalt}
{\rm 
We saw in \S \ref{subsec:altcont} thatfor each doctrine $P \colon \cat^{op} \to \PreOrd$ and Grothendieck topology $J$ on $\cat \rtimes P$, there exists a Grothendieck topology $J^\top$ on the category $\cat \rtimes P^\top$, where $P^\top$ is the free (preserved) top completion, such that $\GC(P,J) \cong \GC(P^\top,J^\top)$.  It is not hard to see that the notion of sub-geometric completion and the result of Theorem \ref{thm:subgeo} can be extended to encompass 2-subcategories $A\text{-}{\bf Doc} \subseteq {\bf DocSites}$ in addition to 2-subcategories $A\text{-}{\bf Doc} \subseteq {\bf Doc}$ as currently presented.  We present the modified result below.

Let $A \text{-}{\bf Doc}$ be a 2-full 2-subcategory of ${\bf DocSites}$ endowed with a 2-monad $(T,\varepsilon,\nu)$.  By ${\bf GeomDoc}_A$ denote the image of the composite
\[\begin{tikzcd}
A\text{-}{\bf Doc} \ar[hook]{r} & {\bf DocSites} \ar{r}{\GC} & {\bf GeomDoc}.
\end{tikzcd}\]
Suppose that, for each $(P,J) \in A \text{-}{\bf Doc}$, there exists a choice $J^A_{(P,J)}$ of Grothendieck topology on the category $\cat \rtimes \GC(P,J)$ such that $(\GC(P,J),J^A_{(P,J)}) \in A\text{-}{\bf Doc}$ and $\GC(P,J)$ also satisfies the following properties.
\begin{enumerate}
\item The choice of topology $J^T_{(P,J)}$ is 2-functorial, i.e. the function that sends $\GC(P,J) \in {\bf GeomDoc}_A$ to $(\GC(P,J),J^A_{(P,J)}) \in A\text{-}{\bf Doc}$ can be extended to a 2-functor $J^T \colon {\bf GeomDoc}_A \to A\text{-}{\bf Doc}$.
\item For each $(P,J) \in A\text{-}{\bf Doc}$, there is a morphism
\[\xi_{(P,J)} \colon T(\GC(P,J),J^T_{(P,J)}) \to (\GC(P,J),J^T_{(P,J)})\]
of $A\text{-}{\bf Doc}$ for which $((\GC(P,J),J^T_{(P,J)}),\xi_{(P,J)})$ is a $T$-algebra and, moreover, the underlying functor and natural transformation pair of $\xi_{(P,J)}$ define a morphism of doctrines
\[\xi_{(P,J)} \colon T(\GC(P,J),J^T_{(P,J)}) \to (\GC(P,J),K_{\GC(P,J)}) .\]
\end{enumerate}
Then the square
\[\begin{tikzcd}
A \text{-}{\bf Doc} \ar{d}{T} \ar[hook]{rr} && {\bf DocSites} \ar{d}{\GC} \\
A \text{-}{\bf Doc} \ar[hook]{r} & {\bf DocSites} \ar{r}{\GC} & {\bf GeomDoc}
\end{tikzcd}\]
commutes up to natural isomorphism.
}
\end{rem}

%%%%%%%%%%%%%%%%%%

\paragraph{When are sub-geometric completions lax-idempotent?}

We saw in Corollary \ref{coro:coarseislax} that the free geometric completion $\GC_{\Free}$ is lax-idempotent.  We may wonder if this infers that any sub-geometric completion is also lax-idempotent.  The inference holds, under some further assumptions.

\begin{prop}\label{prop:subgeoislax}
Let $(T,\varepsilon,\nu)$ be a sub-geometric completion acting on $A\text{-}{\bf Doc}$, such that:
\begin{enumerate}
\item for each $A$-doctrine $P$, the natural transformation $\eta^{(P,J^T_P)} \colon TP \to \GC(TP,J_P^T)$ is point-wise injective,
\item and, for each $A$-doctrine $P$, the multiplication of the free geometric completion
\[\mu^P \colon \GC_{\Free}\GC_{\Free}(P) \to \GC_{\Free}(P)\]
yields a morphism $(\GC_{\Free}\GC_{\Free}(P),\xi_{\GC_{\Free}(P)}) \to (\GC_{\Free}(P),\xi_P)$ of $T$-algebras,
\end{enumerate}
then $(T,\varepsilon,\nu)$ is lax-idempotent.
\end{prop}

\begin{proof}

Recall that a 2-monad $(\tau,e,m)$ on $\dcat$ is lax-idempotent if, for each $d \in \dcat$, there is a 2-cell 
\[\lambda_d \colon \tau e_d \to e_{\tau d},\] natural in $d$, such that the two composites
\[\begin{tikzcd}
d & {\tau d} && {\tau\tau d,} && {\tau d} && {\tau \tau d} & {\tau_c}
\arrow["{e_d}", from=1-1, to=1-2]
\arrow[""{name=0, anchor=center, inner sep=0}, "{\tau e_d}", curve={height=-12pt}, from=1-2, to=1-4]
\arrow[""{name=1, anchor=center, inner sep=0}, "{e_{\tau d}}"', curve={height=12pt}, from=1-2, to=1-4]
\arrow[""{name=2, anchor=center, inner sep=0}, "{\tau e_d}", curve={height=-12pt}, from=1-6, to=1-8]
\arrow[""{name=3, anchor=center, inner sep=0}, "{e_{\tau d}}"', curve={height=12pt}, from=1-6, to=1-8]
\arrow["{m_d}", from=1-8, to=1-9]
\arrow["{\lambda_d}", shorten <=3pt, shorten >=3pt, Rightarrow, from=0, to=1]
\arrow["{\lambda_d}", shorten <=3pt, shorten >=3pt, Rightarrow, from=2, to=3]
\end{tikzcd}\]
are the identity 2-cells.

Since for each $A$-doctrine $P$, the natural transformation $\eta^{(P,J^T_P)} \colon TP \to \GC(TP,J_P^T)$ is point-wise injective, by Remark \ref{rem:fullfaithfulunit} the functor
\[
\Hom_{\bf DocSites}((TP,J^T_P),(TTP,J^T_{TP})) \to \Hom_{{\bf GeomDoc}}(\GC(TP,J^T_P),\GC(\GC(TP,J^T_P),J^T_{\GC(TP,J^T_P)}))
\]
induced by $\eta^{(TP,J^T_P)}$ is full and faithful.  Hence, so too is the functor
\begin{equation}\label{diag:inducedff}
\Hom_{{\bf DocSites}}((TP,J^T_P),(TTP,J^T_{TP})) \to \Hom_{\bf GeomDoc}(\GC_{\Free}(P),\GC_{\Free}\GC_{\Free}(P))
\end{equation}
induced by the composite $\Xi_P \circ \eta^{(TP,J^T_P)}$ (where $\Xi_P$ is the inverse to $\GC(\varepsilon^P)$ constructed in Theorem \ref{thm:subgeo}).  We will write $\Theta_P$ for the composite $\Xi_P \circ \eta^{(TP,J^T_P)}$.  Therefore, since $\GC_{\Free}$ is lax-idempotent, the corresponding 2-cell 
\[\lambda_P \colon \GC_{\Free}(\eta^{(P,J_{\rm triv})}) \to \eta^{(\GC_{\Free}(P),J_{\rm triv})}\]
lifts, as (\ref{diag:inducedff}) is full, to a 2-cell $\lambda_P \colon T\varepsilon^P \to \varepsilon^{TP}$ (this 2-cell is of course labelled by $\id_\cat$, where $P$ is fibred over the category $\cat$ -- but that should not be confused with it being the identity 2-cell).

Note that, by definition, for each $A$-doctrine $P$ the diagram
\[
\begin{tikzcd}
& TP \ar{d}{\eta^{(TP,J^T_P)}} \\
P \ar[bend left]{ru}{\varepsilon^P} \ar[bend right]{rd}[']{\eta^{(P,J_{\rm triv})}} & \GC(TP,J^T_P) \ar[shift left]{d}{\Xi_P} \\
& \GC_{\Free}(P) \ar[shift left]{u}{\GC_{\Free}(\varepsilon^P)}
\end{tikzcd}
\]
commutes.  Therefore, the diagram
% https://q.uiver.app/?q=WzAsNixbMCwwLCJQIl0sWzIsMCwiVFAiXSxbNCwwLCJUVFAiXSxbNCwyLCJcXEdDX3tcXHJtIGZyZWV9XFxHQ197XFxybSBmcmVlfShQKSJdLFsyLDIsIlxcR0Nfe1xccm0gZnJlZX0oUCkiXSxbMCwyLCJQIl0sWzAsMSwiXFx2YXJlcHNpbG9uX1AiXSxbMSwyLCJUXFx2YXJlcHNpbG9uX1AiLDAseyJjdXJ2ZSI6LTJ9XSxbMSwyLCJcXHZhcmVwc2lsb25fe1RQfSIsMix7ImN1cnZlIjoyfV0sWzUsNCwiXFxldGFeeyhQLEpfe1xccm0gdHJpdn0pfSJdLFswLDUsImVxdWFscyIsMl0sWzEsNCwiXFxYaV9QIFxcY2lyYyBcXGV0YV57KFRQLEpeVF9QKX0iLDJdLFsyLDMsIlxcWGlfe1xcR0Nfe1xccm0gZnJlZX0oUCl9XFxjaXJjIFxcZXRhXnsoVFxcR0Nfe1xccm0gZnJlZX0oUCksSl5UX3tcXEdDX3tcXHJtIGZyZWV9KFApfSl9IFxcY2lyYyBUXFxYaV9QIFxcY2lyYyBUXFxldGFeeyhUUCxKXlRfUCl9Il0sWzQsMywiXFxHQ197XFxybSBmcmVlfVxcbGVmdCggXFxldGFeeyhQLEpfe1xccm0gdHJpdn0pfVxccmlnaHQpIiwwLHsiY3VydmUiOi0yfV0sWzQsMywiXFxldGFeeyhcXEdDX3tcXHJtIGZyZWV9KFApLEpfe1xccm0gdHJpdn0pfSIsMix7ImN1cnZlIjoyfV0sWzcsOCwiXFxsYW1iZGFfUCIsMCx7InNob3J0ZW4iOnsic291cmNlIjoyMCwidGFyZ2V0IjoyMH19XSxbMTMsMTQsIlxcbGFtYmRhX1AiLDAseyJzaG9ydGVuIjp7InNvdXJjZSI6MjAsInRhcmdldCI6MjB9fV1d
\[\begin{tikzcd}
P && TP && TTP \\
\\
P && {\GC_{\Free}(P)} && {\GC_{\Free}\GC_{\Free}(P),}
\arrow["{\varepsilon^P}", from=1-1, to=1-3]
\arrow[""{name=0, anchor=center, inner sep=0}, "{T\varepsilon^P}", curve={height=-12pt}, from=1-3, to=1-5]
\arrow[""{name=1, anchor=center, inner sep=0}, "{\varepsilon^{TP}}"', curve={height=12pt}, from=1-3, to=1-5]
\arrow["{\eta^{(P,J_{\rm triv})}}"', from=3-1, to=3-3]
\arrow[equals, from=1-1, to=3-1]
\arrow["\Theta_P"', from=1-3, to=3-3]
\arrow["\Theta_{\GC_{\Free}(P)} \circ T\Theta_P", from=1-5, to=3-5]
\arrow[""{name=2, anchor=center, inner sep=0}, "{\GC_{\Free}\left( \eta^{(P,J_{\rm triv})}\right)}", curve={height=-12pt}, from=3-3, to=3-5]
\arrow[""{name=3, anchor=center, inner sep=0}, "{\eta^{(\GC_{\Free}(P),J_{\rm triv})}}"', curve={height=12pt}, from=3-3, to=3-5]
\arrow["{\lambda_P}", shorten <=3pt, shorten >=3pt, Rightarrow, from=0, to=1]
\arrow["{\lambda_P}", shorten <=3pt, shorten >=3pt, Rightarrow, from=2, to=3]
\end{tikzcd}\]
also commutes.  Since 
\[\Hom_{\bf DocSites}((P,J_{\rm triv}),(TTP,J^T_{TP})) \to \Hom_{\bf DocSites}((P,J_{\rm triv}),(\GC_{\Free}\GC_{\Free}(P),K_{\GC_{\Free}\GC_{\Free}(P)}))\]
is faithful, again by Remark \ref{rem:fullfaithfulunit}, we conclude that $\lambda_P \ast \varepsilon^P $ is indeed the identity 2-cell.

We exploit a symmteric argument to conclude that $\nu^P \ast \lambda_P$ is also the idenitity 2-cell.  We claim that the diagram
% https://q.uiver.app/?q=WzAsNixbMCwwLCJUUCJdLFsyLDAsIlRUUCJdLFsyLDIsIlxcR0Nfe1xccm0gZnJlZX1cXEdDX3tcXHJtIGZyZWV9KFApIl0sWzAsMiwiXFxHQ197XFxybSBmcmVlfShQKSJdLFs0LDAsIlRQIl0sWzQsMiwiXFxHQ197XFxybSBmcmVlfShQKSJdLFswLDEsIlRcXHZhcmVwc2lsb25fUCIsMCx7ImN1cnZlIjotMn1dLFswLDEsIlxcdmFyZXBzaWxvbl97VFB9IiwyLHsiY3VydmUiOjJ9XSxbMCwzLCJcXFRoZXRhX1AiLDJdLFsxLDIsIlxcVGhldGFfe1xcR0Nfe1xccm0gZnJlZX0oUCl9IFxcY2lyYyBUXFxUaGV0YV9QIl0sWzMsMiwiXFxHQ197XFxybSBmcmVlfVxcbGVmdCggXFxldGFeeyhQLEpfe1xccm0gdHJpdn0pfVxccmlnaHQpIiwwLHsiY3VydmUiOi0yfV0sWzMsMiwiXFxldGFeeyhcXEdDX3tcXHJtIGZyZWV9KFApLEpfe1xccm0gdHJpdn0pfSIsMix7ImN1cnZlIjoyfV0sWzEsNCwiXFxudV5QIl0sWzQsNSwiXFxUaGV0YV9QIl0sWzIsNSwiXFxtdV5QIl0sWzYsNywiXFxsYW1iZGFfUCIsMCx7InNob3J0ZW4iOnsic291cmNlIjoyMCwidGFyZ2V0IjoyMH19XSxbMTAsMTEsIlxcbGFtYmRhX1AiLDAseyJzaG9ydGVuIjp7InNvdXJjZSI6MjAsInRhcmdldCI6MjB9fV1d
\begin{equation}\label{eq:2ndliftingdiag}\begin{tikzcd}
TP && TTP && TP \\
\\
{\GC_{\Free}(P)} && {\GC_{\Free}\GC_{\Free}(P)} && {\GC_{\Free}(P)}
\arrow[""{name=0, anchor=center, inner sep=0}, "{T\varepsilon^P}", curve={height=-12pt}, from=1-1, to=1-3]
\arrow[""{name=1, anchor=center, inner sep=0}, "{\varepsilon^{TP}}"', curve={height=12pt}, from=1-1, to=1-3]
\arrow["{\Theta_P}"', from=1-1, to=3-1]
\arrow["{\Theta_{\GC_{\Free}(P)} \circ T\Theta_P}", from=1-3, to=3-3]
\arrow[""{name=2, anchor=center, inner sep=0}, "{\GC_{\Free}\left( \eta^{(P,J_{\rm triv})}\right)}", curve={height=-12pt}, from=3-1, to=3-3]
\arrow[""{name=3, anchor=center, inner sep=0}, "{\eta^{(\GC_{\Free}(P),J_{\rm triv})}}"', curve={height=12pt}, from=3-1, to=3-3]
\arrow["{\nu^P}", from=1-3, to=1-5]
\arrow["{\Theta_P}", from=1-5, to=3-5]
\arrow["{\mu^P}", from=3-3, to=3-5]
\arrow["{\lambda_P}", shorten <=3pt, shorten >=3pt, Rightarrow, from=0, to=1]
\arrow["{\lambda_P}", shorten <=3pt, shorten >=3pt, Rightarrow, from=2, to=3]
\end{tikzcd}\end{equation}
also commutes.  Using that
\[\Hom_{\bf DocSites}((TP,J^T_P),(TP,J^T_P)) \to \Hom_{\bf GeomDoc}(\GC_{\Free}(P),\GC_{\Free}(P))\]
is faithful, again by Remark \ref{rem:fullfaithfulunit}, we conclude that $\nu^P \ast \lambda_P$ is the identity 2-cell as desired.

However, demonstrating the commutativity of (\ref{eq:2ndliftingdiag}), that is the commutativity of the required square
\begin{equation}\label{square:mult}
\begin{tikzcd}
TTP \ar{r}{\nu^P} \ar{d}[']{{\Theta_{\GC_{\Free}(P)} \circ T\Theta_P}} & TP \ar{d}{\Theta_P} \\
\GC_{\Free} \GC_{\Free}(P) \ar{r}{\mu^P} & \GC_{\Free}(P),
\end{tikzcd}
\end{equation}
is more involved.  First, recall from Theorem \ref{thm:subgeo} that, for each $A$-doctrine $P$, the square
\begin{equation}\label{square:Xi}
\begin{tikzcd}
TP \ar{rr}{\eta^{(TP,J^T_P)}} \ar{d}{T\eta^{(P,J_{\rm triv})}} && \GC(TP,J^T_P) \ar{d}{\Xi_P} \\
T\GC_{\Free} \ar{rr}{\xi_P} && \GC_{\Free}(P)
\end{tikzcd}
\end{equation}
commutes.  We can show that the square (\ref{square:mult}) commutes by decomposing it as
\[\begin{tikzcd}[ampersand replacement=\&]
TTP \&\&\&\&\&\& TP \\
\\
{T\GC(TP,J^T_P)} \&\& {TT\GC_{\Free}(P)} \&\& {T\GC_{\Free}(P)} \&\& {\GC(TP,J^T_P)} \\
\\
{T\GC_{\Free}(P)} \&\&\&\&\&\& {\GC_{\Free}(P)} \\
\\
{\GC(T\GC_{\Free}(P),J_{\GC_{\Free}(P)}^T)} \&\& {T\GC_{\Free}\GC_{\Free}(P)} \&\& {T\GC_{\Free}(P)} \\
\\
{\GC_{\Free}\GC_{\Free}(P)} \&\&\&\&\&\& {\GC_{\Free}(P).}
\arrow[""{name=0, anchor=center, inner sep=0}, "{\nu^P}", from=1-1, to=1-7]
\arrow[""{name=1, anchor=center, inner sep=0}, "{\mu^P}", from=9-1, to=9-7]
\arrow["{TT\eta^{(P,J_{\rm triv})}}", from=1-1, to=3-3]
\arrow[""{name=2, anchor=center, inner sep=0}, "{\nu^{\GC_{\Free}(P)}}", from=3-3, to=3-5]
\arrow["{T\eta^{(P,J_{\rm triv})}}"', from=1-7, to=3-5]
\arrow["{T\eta^{(TP,J^T_P)}}"', from=1-1, to=3-1]
\arrow["{T\Xi_p}"', from=3-1, to=5-1]
\arrow["{T\xi_P}", from=3-3, to=5-1]
\arrow[""{name=3, anchor=center, inner sep=0}, "{\xi_P}", from=5-1, to=5-7]
\arrow["{\xi_P}"', from=3-5, to=5-7]
\arrow["{\eta^{(TP,J^T_P)}}", from=1-7, to=3-7]
\arrow["{\Xi_P}", from=3-7, to=5-7]
\arrow[""{name=4, anchor=center, inner sep=0}, equals, from=5-7, to=9-7]
%\arrow["{\varepsilon_{\GC_{\Free}(P)}}"', from=5-7, to=7-5]
\arrow["{\xi_P}"', from=7-5, to=9-7]
\arrow[""{name=5, anchor=center, inner sep=0}, "{T\mu^P}", from=7-3, to=7-5]
\arrow["{\xi_{\GC_{\Free}(P)}}", from=7-3, to=9-1]
\arrow["{T\eta^{(\GC_{\Free}(P),J_{\rm triv})}}", from=5-1, to=7-3]
\arrow["{\eta^{(T\GC_{\Free}(P),J^T_{\GC_{\Free}(P)})}}"', from=5-1, to=7-1]
\arrow["{\Xi_{\GC_{\Free}(P)}}"', from=7-1, to=9-1]
\arrow["{{\text{\textcircled{\raisebox{-.5pt}{1}}}}}"{description}, draw=none, from=3-5, to=3-7]
\arrow["{{\text{\textcircled{\raisebox{-.5pt}{2}}}}}"{description}, draw=none, from=7-1, to=7-3]
\arrow["{{\text{\textcircled{\raisebox{-.5pt}{3}}}}}"{description}, draw=none, from=3-1, to=3-3]
\arrow["{{\text{\textcircled{\raisebox{-.5pt}{4}}}}}", draw=none, from=0, to=2]
\arrow["{{\text{\textcircled{\raisebox{-.5pt}{5}}}}}", draw=none, from=2, to=3]
%\arrow["{{\text{\textcircled{\raisebox{-.5pt}{7}}}}}", draw=none, from=3, to=5]
%\arrow["{{\text{\textcircled{\raisebox{-.5pt}{6}}}}}"{description}, draw=none, from=7-5, to=4]
\arrow["{{\text{\textcircled{\raisebox{-.5pt}{6}}}}}", draw=none, from=5, to=1]
\end{tikzcd}\]
The squares ${\text{\textcircled{\raisebox{-.5pt}{1}}}}$ and ${\text{\textcircled{\raisebox{-.5pt}{2}}}}$ commute by (\ref{square:Xi}), and the square ${\text{\textcircled{\raisebox{-.5pt}{3}}}}$ is just $T$ applied to (\ref{square:Xi}).  The square ${\text{\textcircled{\raisebox{-.5pt}{4}}}}$ commutes by the naturality of $\nu \colon TT \to T$.  The square ${\text{\textcircled{\raisebox{-.5pt}{5}}}}$ commutes since $(\GC_{\Free}(P),\xi_P)$ is a $T$-algebra.  The square ${\text{\textcircled{\raisebox{-.5pt}{6}}}}$ commutes by the assumption that $\mu^P$ yields a morphism $(\GC_{\Free}\GC_{\Free}(P),\xi_{\GC_{\Free}(P)}) \to (\GC_{\Free}(P),\xi_P)$ of $T$-algebras.  Finally, the remaining equation to check, that $\xi_P \circ T\mu^P \circ T \eta^{(\GC_{\Free}(P),J_{\rm triv})} = \xi_P$, follows from the unit law for $(\GC_{\Free},\eta,\mu)$:
\[\mu^P \circ \eta^{(\GC_{\Free}(P),J_{\rm triv})} = \id_{\GC_{\Free}(P)}.\]

\end{proof}

%%%%%%%%%%%%%
%%%%%%%%%%%%%%

\subsection{Sub-geometric completions via subdoctrines}\label{subsec:compatible}

Let $P \colon \cat^{op} \to {\bf MSLat}$ be a primary doctrine.  Another formulation of Proposition \ref{prop:existsissubgeo} is that $P^\exists(c)$ can be recovered as the subset of $\GC_{\Free}(P)(c)$ of elements of the form $\exists_f \eta^{(P,J_{\rm triv})}_d(x)$, where $d \xrightarrow{f} c$ is an arrow of $\cat$ and $x \in P(d)$.  Furthermore, the elements $\exists_f \eta^{(P,J_{\rm triv})}_d(x) \in \GC_{\Free}(P)(c)$ can be characterised as the \emph{supercompact} objects of the site $(\cat \rtimes \GC_{\Free}(P),K_{\GC_{\Free}(P)})$.

\begin{lem}\label{lem:supercompactelements}
An element $S \in \GC_{\Free}(P)(c)$ is of the form $\exists_f \eta^{(P,J_{\rm triv})}_d(x)$ if and only if $(c,S)$ is \emph{supercompact}, i.e. every $K_{\GC_{\Free}(P)}$-cover of $(c,S)$ contains a singleton subcover.
\end{lem}
\begin{proof}
Firstly, if $S$ is supercompact then, since $(c,S)$ admits the $K_{\GC_{\Free}(P)}$-cover
\[\left\{\,(d,\eta^{(P,J_{\rm triv})}_d(x)) \xrightarrow{f} (c,S) \,\middle\vert\, (f,x) \in S\,\right\},\]
$S$ must be equal to $\exists_f \eta^{(P,J_{\rm triv})}_d(x)$ for some $(f,x) \in S$.

The object $(c,\exists_f \eta^{(P,J_{\rm triv})}_d(x))$ is supercompact since if 
\[\left\{\,(e_i,T_i) \xrightarrow{g_i} (c,\exists_f \eta^{(P,J_{\rm triv})}_d(x)) \,\middle\vert\, i \in I\,\right\}\]
is a $K_{\GC_{\Free}(P)}$-cover, then $\bigcup \exists_{g_i} T_i = \exists_f \eta^{(P,J_{\rm triv})}_d(x)$, and so $(f,x) \in \exists_{g_{i'}} T_{i'} $ for some $i' \in I$.  Therefore, ${\exists_{g_{i'}} T_{i'} = \exists_f \eta^{(P,J_{\rm triv})}_d(x)}$ and so the singleton arrow $\{\,(e_{i'},T_{i'}) \xrightarrow{g_{i'}} (c, \exists_f \eta^{(P,J_{\rm triv})}_d(x))\,\}$ is a $K_{\GC_{\Free}(P)}$-cover.
\end{proof}

In this subsection we study completions of doctrines obtained in an analogous fashion by taking certain subdoctrines of the free geometric completion.  We will formulate a general theory for such completions obtained via subdoctrines, and demonstrate that they are sub-geometric in the sense of Definition \ref{df:subgeo}, thus providing a broad class of examples of sub-geometric completions.  Moreover, we show that the induced 2-monads are all lax-idempotent.

Among the examples of sub-geometric completions we obtain in this way is the existential completion $T^\exists \colon {\bf PrimDoc} \to {\bf PrimDoc}$ established in \cite{trotta}.  We will also obtain a lax-idempotent \emph{free coherent completion} for primary doctrines.  Finally, we will relate the existential and coherent completions thus obtained to the regular and coherent completions of cartesian categories.

%%%%%%%%%%%%%%%%%%%%%%%%%%%

\paragraph{Compatible subcompletions}

We first develop our general theory for completions of doctrines obtained via subdoctrines of the free geometric completion.  We call these \emph{compatible subcompletions} in analogy with the terminology `compatible properties' used in the topos-theoretic study of Stone type dualities given in \cite[\S 3]{stonetype}.  Given a doctrine $Q \colon \cat^{op} \to \PreOrd$, by a \emph{subdoctrine} of $Q$ we mean a doctrine $Q' \colon \cat^{op} \to \PreOrd$, also indexed over $\cat$, and a natural transformation $Q' \hookrightarrow Q$ for which every component is a subset inclusion $Q'(c) \subseteq Q(c)$.

For this subsection, in every doctrinal site $(P,J)$ we encounter, the topology $J$ is taken to be the trivial topology $J_{\rm triv}$.  Therefore, we abbreviate our notation and write $\eta^P$ for $\eta^{(P,J_{\rm triv})}$, $\mu^P$ for $\mu^{(P,J_{\rm triv})}$, etc.

\begin{df}
{\rm
Let $A\text{-}{\bf Doc}$ be a 2-full 2-subcategory of ${\bf Doc}_{\rm flat}$ that contains the image of the functor
\[\begin{tikzcd}
A\text{-}{\bf Doc} \ar[hook]{r} & {\bf Doc}_{\rm flat} \ar{r}{\GC_{\Free}} & {\bf GeomDoc} \subseteq {\bf Doc}_{\rm flat},
\end{tikzcd}\]
as well as the unit $\eta^P \colon P \to \GC_{\Free}(P)$ for each $A$-doctrine $P \in A\text{-}{\bf Doc}$.  A choice of a subdoctrine 
\[H^P \colon TP \hookrightarrow \GC_{\Free}(P),\]
for each $A$-doctrine $P$, is said to be $A$-\emph{compatible} if the following conditions are satisfied.

\begin{enumerate}
\item For each $A$-doctrine $P \colon \cat^{op} \to \PreOrd$, $TP$ is a subdoctrine $TP \hookrightarrow \GC_{\Free}(P)$ that contains the image of the unit $\eta^P$, i.e. there is a factorisation
\[
\begin{tikzcd}
	P \ar[bend left]{rr}[pos = 0.55]{\eta^P} \ar{r}{\varepsilon^P} & TP \ar[hook]{r}{H^P} & \GC_{\Free}(P),
\end{tikzcd}
\]
such that the factoring morphism $\varepsilon^P \colon P \to TP$ is a morphism of $A$-doctrines.
\item The choice is natural in the sense that, for each morphism of $A$-doctrines $(F,a) \colon P \to Q$, the induced morphism of geometric doctrines $\GC_{\Free}(F,a) \colon \GC_{\Free}(P) \to \GC_{\Free}(Q)$ restricts to $TP \to TQ$, as in the diagram
\begin{equation}\label{eq:compatibleonarrows}
	\begin{tikzcd}
		P \ar{d}{(F,a)} \ar{r}{\varepsilon^P} & TP \ar[hook]{r}{H^P} \ar[dashed]{d} & \GC_{\Free}(P) \ar{d}{\GC_{\Free}(F,a)}\\
		Q \ar{r}{\varepsilon^Q} & TQ \ar[hook]{r}{H^Q} & \GC_{\Free}(Q),
	\end{tikzcd}
\end{equation}
and moreover the restriction $TP \to TQ$ a morphism of $A$-doctrines (i.e. we obtain an (1-)endofunctor $T \colon A\text{-}{\bf Doc} \to A\text{-}{\bf Doc}$).
\item For each $P \in A\text{-}{\bf Doc}$, the subdoctrine $H^P \colon TP \hookrightarrow \GC_{\Free}(P)$ is `compatible' with the multiplication $\mu^P \colon \GC_{\Free}\GC_{\Free}(P) \to \GC_{\Free}(P)$ in the sense that the composite
\begin{equation}\label{eq:combatiblewithmult}
	\begin{tikzcd}
		TTP \ar{r}{TH^P} & T\GC_{\Free}(P) \ar{rr}{H^{\GC_{\Free}(P)}}  & & \GC_{\Free}\GC_{\Free}(P) \ar{r}{\mu^P} & \GC_{\Free}(P)
	\end{tikzcd}
\end{equation}
factors through the subdoctrine $H^P \colon TP \hookrightarrow \GC_{\Free}(P)$, and this factorisation $\nu^P \colon TTP \to TP$ is a morphism of $A$-doctrines.
\end{enumerate}

}
\end{df}

\begin{ex}\label{ex:compatiblesubcomp}
{\rm  
There are two basic examples to keep in mind for motivating our development.  In both cases, the 2-category $A\text{-}{\bf Doc}$ is taken to be the 2-category of primary doctrines ${\bf PrimDoc}$.
\begin{enumerate}
\item The first example has been encountered already.  For each primary doctrine $P$, taking $T^\exists P \hookrightarrow \GC_{\Free}(P)$ as the subdoctrine on supercompact elements is ${\bf PrimDoc}$-compatible.  While not every morphism of geometric doctrines $(G,b) \colon \GC_{\Free}(P) \to \GC_{\Free}(Q)$ sends a supercompact element $S \in \GC_{\Free}(P)(c)$ to a supercompact element $b_c(S) \in \GC_{\Free}(Q)(G(c))$, this is however true for morphisms of the form $\GC_{\Free}(F,a)$, where $(F,a) \colon P \to Q$ is a morphism of primary doctrines.

An element of $TTP(c)$ is of the form 
\[\exists_{\GC_{\Free}\GC_{\Free}(P)(g)} \eta_e^{\GC_{\Free}(P)}\left( \exists_{\GC_{\Free}(P)(f)} \eta_d^{P}(x)\right),\]
for a composable pair of arrows $d \xrightarrow{f} e, e \xrightarrow{g} c \in \cat$ and an element $x \in P(d)$.  One can calculate that
\[\mu_e^P\left(\exists_{\GC_{\Free}\GC_{\Free}(P)(g)} \eta_e^{\GC_{\Free}(P)}\left( \exists_{\GC_{\Free}(P)(f)} \eta_d^{P}(x)\right)\right) = \exists_{\GC_{\Free}(P)(g \circ f)} \eta_d^P(x).\]
Thus, $\mu^P$ restricts to a morphism $\nu^P \colon TTP \to TP$.  The other required conditions on $T^\exists$ are easily checked.

\item Now consider taking $T^{\rm Coh}P$ to the be the subdoctrine of $\GC_{\Free}(P)$ on \emph{compact elements}, i.e. $T^{\rm Coh}P(c)$ are those elements $S \in \GC_{\Free}(P)(c)$ such that every $K_{\GC_{\Free}(P)}$-cover of $(c,S)$ has a finite subcover.  Checking that this choice of subdoctrine of $\GC_{\Free}(P)$ satisfies the required conditions is analogous to the case for $T^\exists$.

\end{enumerate}

}
\end{ex}

The (1-)endofunctor $T \colon \ADoc \to \ADoc$ is 2-functorial.  Every 2-cell $\alpha \colon (F,a) \to (F',a')$ (i.e. a suitable natural transformation $\alpha \colon F \to F'$) between $A$-doctrine morphisms $(F,a), (F',a') \colon P \rightrightarrows Q$ yields a 2-cell
\[\begin{tikzcd}
{\GC_{\Free}(P)} && {\GC_{\Free}(Q)}
\arrow[""{name=0, anchor=center, inner sep=0}, "{\GC_{\Free}(F,a)}", curve={height=-18pt}, from=1-1, to=1-3]
\arrow[""{name=1, anchor=center, inner sep=0}, "{\GC_{\Free}(F',a')}"', curve={height=18pt}, from=1-1, to=1-3]
\arrow["\alpha", shorten <=5pt, shorten >=5pt, Rightarrow, from=0, to=1]
\end{tikzcd}\]
since $\GC_{\Free}$ is 2-functorial.  Therefore, $\alpha$ also defines a 2-cell between the restrictions to the subdoctrines
\[\begin{tikzcd}
TP && TQ.
\arrow[""{name=0, anchor=center, inner sep=0}, "{T(F,a)}", curve={height=-18pt}, from=1-1, to=1-3]
\arrow[""{name=1, anchor=center, inner sep=0}, "{T(F',a')}"', curve={height=18pt}, from=1-1, to=1-3]
\arrow["\alpha", shorten <=5pt, shorten >=5pt, Rightarrow, from=0, to=1]
\end{tikzcd}\]
We note also that, since the triple $(T,\varepsilon,\nu)$ is a restriction of the 2-monad $(\GC_{\Free},\eta,\mu)$, the monad equations for $(T,\varepsilon,\nu)$ follow automatically.

\begin{lem}
The triple $(T,\varepsilon,\nu)$ is a 2-monad on $A\text{-}{\bf Doc}$.
\end{lem}

We call this 2-monad the \emph{compatible subcompletion}.

\begin{prop}
Every compatible subcompletion $T \colon A\text{-}{\bf Doc} \to A\text{-}{\bf Doc}$ is sub-geometric.
\end{prop}
\begin{proof}
For each $A$-doctrine $P$, the morphism
\[
\begin{tikzcd}
T\GC_{\Free}(P) \ar{rr}{H^{\GC_{\Free}(P)}} & & \GC_{\Free}\GC_{\Free}(P) \ar{r}{\mu^P} & \GC_{\Free}(P)		
\end{tikzcd}
\]
is a natural way of endowing $\GC_{\Free}(P)$ with the structure of a $T$-algebra.  The unit condition, i.e. the commutativity of the triangle
\[
\begin{tikzcd}
{\GC_{\Free}(P)} \ar[equal]{rrdd} \ar{rr}{\varepsilon^{\GC_{\Free}(P)} } && T{\GC_{\Free}(P)} \ar{d}{H^{\GC_{\Free}(P)} } \\
&& \GC_{\Free}\GC_{\Free}(P) \ar{d}{\mu^P}\\\
&& {\GC_{\Free}(P)},
\end{tikzcd}
\]
is satisfied since $H^{\GC_{\Free}(P)} \circ \varepsilon^{\GC_{\Free}(P)} = \eta^{\GC_{\Free}(P)} $ and $\mu^P \circ \eta^{\GC_{\Free}(P)} = \id_{{\GC_{\Free}(P)}}$.  The action property, i.e. that 
\[\mu^P \circ H^{\GC_{\Free}(P)} \circ T(\mu^P \circ H^{\GC_{\Free}(P)}) = \mu^P \circ H^{\GC_{\Free}(P)} \circ \nu^{\GC_{\Free}(P)},\]
follows from the commutativity of the diagram
% https://q.uiver.app/?q=WzAsOCxbMCwwLCJUVFxcR0Nfe1xccm0gZnJlZX0oUCkiXSxbMCw0LCJUXFxHQ197XFxybSBmcmVlfShQKSJdLFsyLDAsIlRcXEdDX3tcXHJtIGZyZWV9XFxHQ197XFxybSBmcmVlfShQKSJdLFsyLDIsIlxcR0Nfe1xccm0gZnJlZX1cXEdDX3tcXHJtIGZyZWV9XFxHQ197XFxybSBmcmVlfShQKSJdLFsyLDQsIlxcR0Nfe1xccm0gZnJlZX1cXEdDX3tcXHJtIGZyZWV9KFApIl0sWzQsNCwiXFxHQ197XFxybSBmcmVlfShQKSJdLFs0LDIsIlxcR0Nfe1xccm0gZnJlZX1cXEdDX3tcXHJtIGZyZWV9KFApIl0sWzQsMCwiVFxcR0Nfe1xccm0gZnJlZX0oUCkiXSxbMCwxLCJcXG51Il0sWzIsMywiSF57XFxHQ197XFxybSBmcmVlfVxcR0Nfe1xccm0gZnJlZX0oUCl9Il0sWzEsNCwiSF57XFxHQ197XFxybSBmcmVlfShQKX0iXSxbMCwyLCJUSF57XFxHQ197XFxybSBmcmVlfShQKX0iXSxbMyw0LCJcXG11XntcXEdDX3tcXHJtIGZyZWV9KFApfSJdLFs0LDUsIlxcbXVeUCJdLFsyLDcsIlRcXG11XlAiXSxbNyw2LCJIXntcXEdDX3tcXHJtIGZyZWV9KFApfSJdLFszLDYsIlxcR0Nfe1xccm0gZnJlZX1cXG11XlAiXSxbNiw1LCJcXG11XlAiXSxbMCw3LCJUKFxcbXVeUCBcXGNpcmMgSF57XFxHQ197XFxybSBmcmVlfShQKX0pIiwwLHsiY3VydmUiOi00fV0sWzE0LDE2LCJcXHRleHRjaXJjbGVke1xccmFpc2Vib3h7LS41cHR9ezJ9fSIsMSx7InNob3J0ZW4iOnsic291cmNlIjoyMCwidGFyZ2V0IjoyMH0sInN0eWxlIjp7ImJvZHkiOnsibmFtZSI6Im5vbmUifSwiaGVhZCI6eyJuYW1lIjoibm9uZSJ9fX1dLFsxNiwxMywiXFx0ZXh0Y2lyY2xlZHtcXHJhaXNlYm94ey0uNXB0fXsxfX0iLDEseyJzaG9ydGVuIjp7InNvdXJjZSI6MjAsInRhcmdldCI6MjB9LCJzdHlsZSI6eyJib2R5Ijp7Im5hbWUiOiJub25lIn0sImhlYWQiOnsibmFtZSI6Im5vbmUifX19XSxbMTEsMTAsIlxcdGV4dGNpcmNsZWR7XFxyYWlzZWJveHstLjVwdH17M319IiwxLHsic2hvcnRlbiI6eyJzb3VyY2UiOjIwLCJ0YXJnZXQiOjIwfSwic3R5bGUiOnsiYm9keSI6eyJuYW1lIjoibm9uZSJ9LCJoZWFkIjp7Im5hbWUiOiJub25lIn19fV1d
\begin{equation}\label{square:compatible}
\begin{tikzcd}
{TT\GC_{\Free}(P)} && {T\GC_{\Free}\GC_{\Free}(P)} && {T\GC_{\Free}(P)} \\
\\
&& {\GC_{\Free}\GC_{\Free}\GC_{\Free}(P)} && {\GC_{\Free}\GC_{\Free}(P)} \\
\\
{T\GC_{\Free}(P)} && {\GC_{\Free}\GC_{\Free}(P)} && {\GC_{\Free}(P).}
\arrow["\nu^{\GC_{\Free}(P)}", from=1-1, to=5-1]
\arrow["{H^{\GC_{\Free}\GC_{\Free}(P)}}", from=1-3, to=3-3]
\arrow[""{name=0, anchor=center, inner sep=0}, "{H^{\GC_{\Free}(P)}}", from=5-1, to=5-3]
\arrow[""{name=1, anchor=center, inner sep=0}, "{TH^{\GC_{\Free}(P)}}", from=1-1, to=1-3]
\arrow["{\mu^{\GC_{\Free}(P)}}", from=3-3, to=5-3]
\arrow[""{name=2, anchor=center, inner sep=0}, "{\mu^P}", from=5-3, to=5-5]
\arrow[""{name=3, anchor=center, inner sep=0}, "{T\mu^P}", from=1-3, to=1-5]
\arrow["{H^{\GC_{\Free}(P)}}", from=1-5, to=3-5]
\arrow[""{name=4, anchor=center, inner sep=0}, "{\GC_{\Free}\mu^P}", from=3-3, to=3-5]
\arrow["{\mu^P}", from=3-5, to=5-5]
%\arrow["{T(\mu^P \circ H^{\GC_{\Free}(P)})}", curve={height=-24pt}, from=1-1, to=1-5]
\arrow["{\text{\textcircled{\raisebox{-.5pt}{2}}}}"{description}, draw=none, from=3, to=4]
\arrow["{\text{\textcircled{\raisebox{-.5pt}{1}}}}"{description}, draw=none, from=4, to=2]
\arrow["{\text{\textcircled{\raisebox{-.5pt}{3}}}}"{description}, draw=none, from=1, to=0]
\end{tikzcd}
\end{equation}
The commutativity of the square {\textcircled{\raisebox{-.5pt}{1}}} is assured since $(\GC_{\Free},\eta,\mu)$ is a 2-monad, while the squares {\textcircled{\raisebox{-.5pt}{2}}} and {\textcircled{\raisebox{-.5pt}{3}}} commute by definition (see the equations (\ref{eq:compatibleonarrows}) and (\ref{eq:combatiblewithmult})).

We now seek to find a Grothendieck topology $J^T_P$ on $\cat \rtimes TP$ satisfying the required conditions of Definition \ref{df:subgeo}.  We take the obvious choice: since $TP$ is a subdoctrine of $\GC_{\Free}(P)$, $\cat \rtimes TP$ is a subcategory of $\cat \rtimes \GC_{\Free}(P)$, and so we define $J^T_P$ as the restriction of $K_{\GC_{\Free}(P)}$ to $\cat \rtimes P$.  We check that the three conditions of Definition \ref{df:subgeo}\ref{df:subgeo:topologystuff} are satisfied.
\begin{enumerate}[label = (\alph*)]

\item Recall that that the unit of the free geometric completion yields a dense morphism of sites
\[\id_\cat \rtimes \eta^P \colon (\cat \rtimes P,J_{\rm triv}) \to (\cat \rtimes \GC_{\Free}(P),K_{\GC_{\Free}(P)}).\]
The functor $\id_\cat \rtimes \eta^P$ factorises as 
\[
\begin{tikzcd}
(\cat \rtimes P,J_{\rm triv}) \ar{r}{\id_\cat \rtimes \varepsilon^P} & (\cat \rtimes TP, J^T_P) \ar{r}{\id_\cat \rtimes H^P} & (\cat \rtimes \GC_{\Free}(P),K_{\GC_{\Free}(P)}).
\end{tikzcd}
\]
The right factor, $\id_\cat \rtimes H^P$, is the inclusion of a dense subcategory, and hence also a dense morphism of sites.  Therefore, by \cite[Corollary 11.6]{shulman}, $\id_\cat \rtimes \varepsilon^P$ is a morphism of sites, and so \[\varepsilon^P \colon (P,J_{\rm triv}) \to (TP,J^T_P)\]
is a morphism of doctrinal sites as desired.

\item Firstly, the functor $\mu^P$ defines a morphism of doctrinal sites 
\[\mu^P \colon (\GC_{\Free}\GC_{\Free}(P),K_{\GC_{\Free}\GC_{\Free}(P)}) \to (\GC_{\Free}(P),K_{\GC_{\Free}(P)}).\]
Secondly, $\id_\cat \rtimes H^{\GC_{\Free}(P)} \colon (\cat \rtimes \GC_{\Free}(P),J^T_{\GC_{\Free}(P)}) \to (\cat \rtimes \GC_{\Free}\GC_{\Free}(P),K_{\GC_{\Free}\GC_{\Free}(P)})$ is the inclusion of a dense subcategory, thus a morphism of sites and therefore
\[
H^{\GC_{\Free}(P)} \colon ( \GC_{\Free}(P),J^T_{\GC_{\Free}(P)}) \to (\GC_{\Free}\GC_{\Free}(P),K_{\GC_{\Free}\GC_{\Free}(P)})
\]
is a morphism of doctrinal sites.  Hence, the composite $\mu^P \circ H^{\GC_{\Free}(P)}$ defines a morphism of doctrinal sites
\[\mu^P \circ H^{\GC_{\Free}(P)} \colon (T{\GC_{\Free}(P)},J^T_{\GC_{\Free}(P)}) \to (\GC_{\Free}(P),K_{\GC_{\Free}(P)}).\]

\item Finally, we wish to show that $T\theta \colon (TP,J^T_P) \to (TQ,J^T_Q)$ is a morphism of doctrinal sites, for each morphism of $A$-doctrines $\theta \colon P \to Q$.  By assumption, $T\theta$ is already flat by virtue of being a morphism of $A$-doctrines.  That $T\theta$ sends $J^T_P$-covers to $J^T_Q$ covers follows since $\GC_{\Free}(\theta) \colon \GC_{\Free}(P) \to \GC_{\Free}(Q)$ sends $K_{\GC_{\Free}(P)}$-covers to $K_{\GC_{\Free}(Q)}$-covers.
\end{enumerate}
\end{proof}

%%%%%%%%%%%%%%%%%%%%%%%%%%%%%%%%%%%%%%%

%\paragraph{Compatible subcompletions are lax-idempotent.}

An application of Proposition \ref{prop:subgeoislax} now shows that every compatible subcompletion is lax-idempotent.  The two conditions of Proposition \ref{prop:subgeoislax} are clearly satisfied.  Each component $\eta_c^{(TP,J^T_P)} \colon TP(c) \to \GC(TP,J^T_P)(c)$ is injective -- indeed it is isomorphic to the inclusion $H^P_c \colon TP(c) \subseteq \GC_{\Free}(P)(c) \cong \GC(TP,J^T_P)(c)$, while 
\[\mu^P \colon (\GC_{\Free}\GC_{\Free}(P), \mu^{\GC_{\Free}(P)} \circ H^{\GC_{\Free}\GC_{\Free}(P)}) \to (\GC_{\Free}(P),\mu^P \circ H^{\GC_{\Free}(P)})\]
is a morphism of $T$-algebras by the commutativity of the right hand side of (\ref{square:compatible}).

\begin{coro}
Every compatible subcompletion $(T,\varepsilon,\nu)$ is a lax-idempotent 2-monad.
\end{coro}

\paragraph{The regular and coherent completions.}

Let us revisit the examples of compatible subcompletions given in Example \ref{ex:compatiblesubcomp}.  As remarked in Lemma \ref{lem:supercompactelements}, we have recovered the existential completion established in \cite{trotta}, the lax-idempotent 2-monad $T^\exists \colon {\bf PrimDoc} \to {\bf PrimDoc}$, as a compatible subcompletion.

The 2-category of algebras for the 2-monad $T^\exists$ is precisely the the 2-category ${\bf ExDoc}$ of existential doctrines (see \cite[Corollary 5.5]{trotta}).  In a similar fashion, we recognise the 2-category of algebras for the lax-idempotent 2-monad $T^{\rm Coh} \colon {\bf PrimDoc} \to {\bf PrimDoc}$ as the 2-category ${\bf CohDoc}$ of coherent doctrines.  Using the inherent 2-adjunction
\[\begin{tikzcd}
{\tau\text{-}{\bf Alg}} && {\mathcal{C}}
\arrow[""{name=0, anchor=center, inner sep=0}, shift right=3, hook, from=1-1, to=1-3]
\arrow[""{name=1, anchor=center, inner sep=0}, "\tau"', shift right=3, from=1-3, to=1-1]
\arrow["\dashv"{anchor=center, rotate=-90}, draw=none, from=1, to=0]
\end{tikzcd}\]
for a 2-monad $(\tau,e,m)$ on a  2-category $\cat$, we recover the following completions of doctrines.

\begin{coro}[\S 5 \cite{trotta}]
\begin{enumerate}
\item The 2-embedding ${\bf ExDoc} \hookrightarrow {\bf PrimDoc}$ has a lax-idempotent left 2-adjoint.
\item The 2-embedding ${\bf CohDoc} \hookrightarrow {\bf PrimDoc}$ has a lax-idempotent left 2-adjoint.
\end{enumerate}
\end{coro}

Following the example of \cite{uniexactcomp}, we turn to using completions of doctrines to describe completions of categories.  We have seen in Corollary \ref{coro:freegcforcat} that the free geometric completion yields the completion of a cartesian category to a geometric category.  We deduce that, in a similar manner, the sub-geometric completions we have constructed in this subsection yield other completions of cartesian categories.

As already noted in \cite[\S 6]{trotta}, the existential completion of a primary doctrine can be used to recover the \emph{regular completion} of a cartesian category.  For a cartesian category $\cat$, Carboni describes in \cite{carboni} the regular completion ${\bf Reg}(\cat)$ as follows:
\begin{enumerate}
\item the objects of ${\bf Reg}(\cat)$ are arrows $d \xrightarrow{f } c$ of $\cat$;
\item an arrow $[g] \colon f_1 \to f_2$ of ${\bf Reg}(\cat)$ is an equivalence class of arrows $g \colon d_1 \to d_2$ such that
\[\begin{tikzcd}
e \ar[shift left]{r}{h} \ar[shift right]{r}[']{k} & d_1 \ar{r}{f_2 \circ g} & c_2
\end{tikzcd}\]
commutes, where $(h,k)$ are the kernel pair of $f_1$, i.e.
\[\begin{tikzcd}
e \ar{r}{h} \ar{d}{k} & d_1 \ar{d}{f_1} \\
d_1 \ar{r}{f_1} & c_1
\end{tikzcd}\]
is a pullback.  Two such arrows $g,g' \colon d_1 \rightrightarrows d_2$ are equivalent ($[g] = [g']$) if
\[
\begin{tikzcd}
d_1 \ar[shift left]{r}{g} \ar[shift right]{r}[']{g'} & d_2 \ar{r}{f_2} & c_2
\end{tikzcd}
\]
commutes.
\end{enumerate}
In \cite[\S 5]{carboni} it is shown that this defines the action on objects of a quasi 2-adjoint to the 2-embedding
\[
\begin{tikzcd}
{\bf Reg} \ar[hook]{r} & {\bf Cart}
\end{tikzcd}
\]
of cartesian categories into regular categories.  In an analogous manner to Corollary \ref{coro:freegcforcat}, we deduce that ${\bf Syn}(T^\exists\Sub_{\cat})$ satisfies the same universal property as ${\bf Reg}(\cat)$, and hence ${\bf Syn}(T^\exists\Sub_{\cat}) \simeq {\bf Reg}(\cat)$.  Similarly, by considering the category ${\bf Syn}(T^{\rm Coh}\Sub_{\cat})$, for a cartesian category $\cat$, we obtain the universal \emph{coherent completion} of $\cat$.  

\begin{coro}
The 2-embedding ${\bf Coh} \hookrightarrow {\bf Cart}$ has a left quasi 2-adjoint -- the \emph{coherent completion} of a cartesian category.
\end{coro}

\begin{rem}
{\rm
As observed in \cite[\S 6]{trotta}, by composing the (quasi) 2-adjunctions coming from the existential completion, the syntactic category construction and the exact completion of a regular category (see \cite[\S 2.3]{carbonivitale}), as in the diagram
\[\begin{tikzcd}
{{\bf PrimDoc}} && {{\bf ExDoc}} && {{\bf Reg}} && {{\bf Exact},}
\arrow[""{name=0, anchor=center, inner sep=0}, shift left=3, hook', from=1-3, to=1-1]
\arrow[""{name=1, anchor=center, inner sep=0}, shift left=3, from=1-1, to=1-3]
\arrow[""{name=2, anchor=center, inner sep=0}, shift left=3, from=1-5, to=1-3]
\arrow[""{name=3, anchor=center, inner sep=0}, shift left=3, from=1-3, to=1-5]
\arrow[""{name=4, anchor=center, inner sep=0}, shift left=3, hook', from=1-7, to=1-5]
\arrow[""{name=5, anchor=center, inner sep=0}, shift left=3, from=1-5, to=1-7]
\arrow["\dashv"{anchor=center, rotate=-90}, draw=none, from=1, to=0]
\arrow["\dashv"{anchor=center, rotate=-90}, draw=none, from=3, to=2]
\arrow["\dashv"{anchor=center, rotate=-90}, draw=none, from=5, to=4]
\end{tikzcd}\]
we obtain the \emph{exact completion of a primary doctrine} in the sense of \cite{uniexactcomp}.  The work of Maietti, Pasquali and Rosolini \cite{uniexactcomp}, \cite{MPRtriposes} has been fundamental in understanding the exact completion of a doctrine via a series of doctrinal completions.  The functor ${\bf ExDoc} \to {\bf Exact}$ is the so-called \emph{`tripos' construction} (see \cite{triposoriginal}, \cite{triposretro}), also called the partial equivalence relation construction since the objects of the resultant category are \emph{partial equivalence relations} in the internal language of the doctrine.

Since the geometric completion of a primary doctrine interprets geometric logic, we could also consider taking the analogous category whose objects are finite (or infinite) tuples of internal partial equivalence relations.  This would obtain a doctrinal version of the \emph{pretopos completion} (see \cite[\S 8.4]{MR}).  The unifying role of the geometric completion in relating the various `tripos-like' constructions will be the subject of future work.

}
\end{rem}

%%%%%%%%%%%%%%%%%%%%%%%%%%%%%%%%%%%%%%%%%%%%%%%%%%%%%%%%%
%%%%%%%%%%%%%%%%%%%%%%%%%%%%%%%%%%%%%%%%%%%%%%%%%%%%%%%%%%%%%%%%%%%
%%%%%%%%%%%%%%%%%%%%%%%%%%%%%%%%%%%%%%%%%%%%%%%%%%%%%%%%%%%%%

%%%%%%%%%%%%%%%%%%%%%%%%%%%%%%%%%%%%%%%%%%%%%%
%%%%%%%%%%%%%%%%%%%%%%%%%%%%%%%%%%%%%%%%%%%%%%
%%%%%%%%%%%%%%%%%%%%%%%%%%%%%%%%%%%%%%%%%%%%%%%%%%%%

%%%%%%%%%%%%%%%%%%%%%%%%%%%%%%%%%%%%%%%%%%%%%%
%%%%%%%%%%%%%%%%%%%%%%%%%%%%%%%%%%%%%%%%%%%%%%
%%%%%%%%%%%%%%%%%%%%%%%%%%%%%%%%%%%%%%%%%%%%%%%%%%%%

%%%%%%%%%%%%%%%%%%%%%%%%%%%%%%%%%%%%%%%%%%%%%%%%%%%
%%%%%%%%%%%%%%%%%%%%%%%%%%%%%%%%%%%%%%%%%%%%%%%%%%%

%%%%%%%%%%%%%%%%%%%%%%%%%%%%%%%%%%%%%%%%%%%%%%%%%%%%%%%%%%%%%%%%%%%%%%%%%%%%%%%%%%%%%%%%%%

%%%%%%%%%%%%%%%%%%
%%%%%%%%%%%%%%%%

\subsection{Point-wise sub-geometric completions}\label{subsec:exsubgeo}

In this final subsection, we revisit the free top completion as a sub-geometric completion in light of Definition \ref{df:subgeo}.  Since the syntax of geometric logic is often represented by the symbols $\{\, \top, \, \exists, \, \bigvee, \, \land \,\}$, we `complete the set', so to speak, by also briefly sketching that the \emph{free join} and \emph{free binary meet} completions also constitute sub-geometric completions (the completion with respect to the symbol $\{\,\exists\,\}$ being the existential completion previously discussed).  The conditions of Definition \ref{df:subgeo} are easily, but tediously, checked -- and so we omit many of the details.

Since the completions we consider in this subsection are of a `point-wise' nature, we first state some easily deduced facts concerning such completions.  Suppose that we are given a 2-full 2-subcategory $A\text{-}\PreOrd$ of $ \PreOrd$ whose inclusion $A\text{-}\PreOrd \hookrightarrow \PreOrd$ has a left (strict) 2-adjoint $T^A \colon \PreOrd \to A\text{-}\PreOrd$, or equivalently, for each preorder $P$, the completion $T^AP$ has the universal property that for any monotone map $a \colon P \to Q$, where $Q \in A\text{-}\PreOrd$, there is a unique morphism $a^A \colon T^A P \to Q$ of $A\text{-}\PreOrd$ for which the triangle
\[
\begin{tikzcd}
P \ar{rd}[']{a} \ar{r}{\varepsilon} & T^A P \ar[dashed]{d}{a^A} \\
& Q
\end{tikzcd}
\]
commutes (where $\varepsilon$ is the unit of the 2-adjunction).  It is clearly deduced that the functor $T^A$ extends to a (strict) 2-adjunction
\[\begin{tikzcd}
{{[\cat^{op},\PreOrd]}} && {{[\cat^{op},A \text{-} \PreOrd],}}
\arrow[""{name=0, anchor=center, inner sep=0}, shift left=2, hook', from=1-3, to=1-1]
\arrow[""{name=1, anchor=center, inner sep=0}, "T'^A", shift left=2, from=1-1, to=1-3]
\arrow["\dashv"{anchor=center, rotate=-90}, draw=none, from=1, to=0]
\end{tikzcd}\]
and hence a (strict) 2-adjunction
\[\begin{tikzcd}
{{\bf Doc}} && {A\text{-}{\bf Doc},}
\arrow[""{name=0, anchor=center, inner sep=0}, shift left=2, hook', from=1-3, to=1-1]
\arrow[""{name=1, anchor=center, inner sep=0}, "T''^A", shift left=2, from=1-1, to=1-3]
\arrow["\dashv"{anchor=center, rotate=-90}, draw=none, from=1, to=0]
\end{tikzcd}\]
where $A\text{-}{\bf Doc}$ is the category of $A\text{-}\PreOrd$-valued doctrines.

\paragraph{Free top completion.}  Let $T^\top \colon {\bf Doc} \to {\bf Doc}$ denote the free (preserved) top completion monad constructed in \S \ref{subsec:altcont}.  Having preserved top elements, every geometric doctrine $\Lb \colon \cat^{op} \to \Frm_{\rm open}$ can naturally be turned into an algebra for the monad $T^\top$.  We have also already encountered the topology ${J^\top_{\rm triv}}$ on the category $\cat \rtimes P^\top$, where $P$ is a doctrine $P \colon \cat^{op} \to \PreOrd$.  That this choice of Grothendieck topology satisfies the condition of Definition \ref{df:subgeo} is easily shown: for example, that the unit $(P,J_{\rm triv}) \hookrightarrow (P^\top,J_{\rm triv}^\top)$ is a morphism of sites follows from Lemma \ref{addingtopeq}.  Thus, we can apply Theorem \ref{thm:subgeo} to deduce Proposition \ref{topissubgeo}.

\paragraph{Free join completion.}  As previously mentioned, the 2-functor $\GC_{\rm Ex} \colon {\bf ExDoc} \to {\bf GeomDoc}$ sends an existential doctrine $P \colon \cat^{op} \to {\bf MSLat}$ to its `point-wise' join completion $2^{(-)^{op}} \circ P \colon \cat^{op} \to \Frm_{\rm open}$, i.e. $\GC_{\rm Ex}(P)(c)$ is the poset of down-sets of $P(c)$ ordered by inclusion.  We could also conceive of taking the `point-wise' join completion $2^{(-) ^{op}} \circ  P$ of any doctrine $P \in {\bf Doc}$.  Hence an element $J \in 2^{(-) ^{op}} \circ  P(c)$ is a down-set of $P(c)$.  By the above discussion, this yields a left adjoint $T^\lor$ to the inlcusion of ${\bf SupSLat}$-valued doctrines into ${\bf Doc}$, where ${\bf SupSLat}$ is the 2-category of sup-semilattices (i.e. posets with all joins), their homomorphisms, and natural transformations between these.  By the universal property of $T^\lor$, for each geometric doctrine $\Lb \in {\bf GeomDoc}$ there exists a natural transformation $\id_\Lb^\lor\colon T^\lor \Lb \to \Lb$ for which $(\Lb,\id_\Lb^\lor)$ is a $T^\lor$-algebra.

For each doctrine $P \colon \cat^{op} \to \sets$, the choice of the topology $J^\lor_P$, where $J^\lor_P$ is the Grothendieck topology on $\cat \rtimes T^\lor P$ generated by covering families of the form
\[\left\{\, (c,J_i) \xrightarrow{\id_c} \left(c, \bigcup_{i \in I} J_i\right) \,\middle\vert\, i \in I\,\right\},\]
can easily be shown to satisfy the conditions of Definition \ref{df:subgeo}.  Thus, by Theorem \ref{thm:subgeo}, there is a natural isomorphsim $\GC_{\Free}(P) \cong \GC(T^\lor P , J^\lor_P)$ for each doctrine $P$.

\paragraph{Free binary meet completion.}  Finally, we construct the \emph{free binary meet completion} for doctrines, and observe that this is also a sub-geometric completion.  We begin by defining the free binary meet completion for preorders.

\begin{df}
{\rm 
Let $P$ be a preorder.  Consider the set $\mathscr{P}_{\rm fin}(P)\setminus \emptyset$ of non-empty, finite subsets of $P$.  We order $\mathscr{P}_{\rm fin}(P)\setminus \emptyset$ by setting $\{\,x_1,x_2, \, ... \, , x_n\,\} \leqslant \{\,y_1, y_2, \, ... \, , y_m\,\}$ if and only if each $y_i$ is greater than some $x_j$.  We define $P^\land $ as the poset obtained by identifiying two elements $\{\,x_1,x_2, \, ... \, , x_n\,\} ,\, \{\,y_1, y_2, \, ... \, , y_m\,\} $ of $ \mathscr{P}_{\rm fin}(P)\setminus \emptyset$ if
\[ \{\,x_1,x_2, \, ... \, , x_n\,\} \leqslant \{\,y_1, y_2, \, ... \, , y_m\,\} \text{ and } \{\,y_1,y_2, \, ... \, , y_m\,\} \leqslant \{\,x_1, x_2, \, ... \, , x_n\,\}.\]
We denote the equivalence class of $\{\,x_1,x_2, \, ... \, , x_n\,\}$ by $\class{x_1, x_2, \, ... \, , x_n}$.  Alternatively, $P^\land$ can be described as the poset of non-empty, finitely generated up-sets of $P$ ordered by inclusion.
}
\end{df}

It is easily checked that the meet of two elements $\class{x_1,  \, ... \, , x_n}, \class{y_1,  \, ... \, , y_m} \in P^\land$ is given by 
\[\class{x_1,  \, ... \, , x_n, \, y_1, \, ... \, , y_m},\]
and thus the poset $P^\land$ has all binary meets.  The map $\class{-}_P \colon P \to P^\land$ given by sending $x \in P$ to $\class{x} \in P^\land$ is clearly monotone.  Since every element $\class{x_1, \, ... \, , x_n} \in P^\land$ is the finite meet of the elements $\class{x_i} \in P^\land$, we obtain the desired universal property: for each preorder $P$ and any monotone map $a \colon P \to Q$, where $Q$ has binary meets, there exists a unique monotone map $a^\land \colon P^\land \to Q$ that preserves binary meets such that the triangle
\[\begin{tikzcd}
P \ar{rd}[']{a} \ar{r}{\class{-}_P} & P^\land \ar[dashed]{d}{a^\land} \\
& Q
\end{tikzcd}\]
commutes.

Thus, by the discussion above, there exists a left 2-adjoint $T^\land$ to the inclusion of ${\bf BMSLat}$-valued doctrines into ${\bf Doc}$, where ${\bf BMSLat}$ is the 2-category of binary-meet-semilattices, their homomorphisms, and natural transformations between these.  Obviously, there exists a natural transformation $\id_\Lb^\land \colon T^\land \Lb \to \Lb$, induced by the universal property of $T^\land$, which yields a $T^\land$ algebra $(\Lb,\id_\Lb^\land)$ for each geometric doctrine $\Lb \in {\bf GeomDoc}$.

By $J^\land_P$, we denote the Grothendieck topology on $\cat \rtimes T^\land P$ generated by covering families of the form
\[\{\,(c,\class{y}) \xrightarrow{\id_c} (c,\class{x_1, x_2, \, ... \, , x_n})  \mid y \in P(c), \,  y \leqslant x_1, x_2, \, ... \, , x_n\,\}.\]
There are few obstacles to concluding that the choice of topology $J^\land_P$ satisfies the conditions of Definition \ref{df:subgeo}.  Hence we obtain by Theorem \ref{thm:subgeo} that there is a natural isomorphism $\GC_{\Free}(P) \cong \GC(T^\land P,J_P^\land)$ for every doctrine $P$.

%%%%%%%%%%%%%%%%%%%%%%%%%%%%%%%%%%%%%%%%%%%%%%%%%%%%%%%%%%%%%%
%%%%%%%%%%%%%%%%%%%%%%%%%%%%%%%%%%%%%%%%%%%%%%%%%%%%%%%%%%%%%%

%%%%%%%%%%%%%%%%%%%%%%%%%%%%%%%%
%%%%%%%%%%%%%%%%%%%%%%%%%%%%%%5
%%%%%%%%%%%%%%%%%%%%%%%%%%%%%%%%%%

%%%%%%%%%%%%%%%%%%%%%%%%%%%%%%%%%%%%%%%%%%%%%%%%%%%%%%%%%
%%%%%%%%%%%%%%%%%%%%%%%%%%%%%%%%%%%%%%%%%%%%%%%%%%%%%%%%%%%%%%%%%%%
%%%%%%%%%%%%%%%%%%%%%%%%%%%%%%%%%%%%%%%%%%%%%%%%%%%%%%%%%%%%%

%%%%%%%%%%%%%%%%%%%%%%%%%%%%%%%%%%%%%%%%%%%%%%
%%%%%%%%%%%%%%%%%%%%%%%%%%%%%%%%%%%%%%%%%%%%%%
%%%%%%%%%%%%%%%%%%%%%%%%%%%%%%%%%%%%%%%%%%%%%%%%%%%%

%%%%%%%%%%%%%%%%%%%%%%%%%%%%%%%%%%%%%%%%%%%%%%%%%%%
%%%%%%%%%%%%%%%%%%%%%%%%%%%%%%%%%%%%%%%%%%%%%%%%%%%

%%%%%%%%%%%%%%%%%%%%%%%%%%%%%%%%%%%%%%%%%%%%%%%%%%%%%%%%%%%%%%%%%%%%%%%%%%%%%%%%%%%%%%%%%%
%%%%%%%%%%%%%%%%%%%%%%%%%%%%%%%%%%%%%%%%%%%%%%%%%%%%%%%%%%%%%%%%%%%%%%%%%%%%%%%%%%%%%%%%%%
\section*{Acknowledgements}

I thank my supervisor Olivia Caramello for her support, and acknowledge the financial support of the Insubria-Huawei studentship into ``Grothendieck toposes for information and computation''.

\printbibliography

\end{document}